\documentclass[11pt,oneside,reqno]{amsart}
\usepackage[latin9]{inputenc}
\usepackage{geometry}
\usepackage{color}
\usepackage[dvipsnames]{xcolor}
\usepackage{amstext}
\usepackage{amsthm}
\usepackage{amssymb}
\usepackage{setspace}
\usepackage{comment}
\usepackage[T1]{fontenc}
\usepackage{textcomp}
\usepackage[unicode=true,
 bookmarks=false,
 breaklinks=false,pdfborder={0 0 1},backref=section,colorlinks=false]
 {hyperref}
\hypersetup{
    colorlinks,
    linkcolor={red!50!black},
    citecolor={blue!50!black},
    urlcolor={blue!80!black}
}
\usepackage{rhee}
\makeatletter
\numberwithin{equation}{section}
\numberwithin{figure}{section}
\theoremstyle{plain}
\newtheorem{thm}{\protect\theoremname}[section]
\theoremstyle{definition}
\newtheorem{defn}[thm]{\protect\definitionname}
\theoremstyle{remark}
\newtheorem{notation}[thm]{\protect\notationname}
\theoremstyle{remark}
\newtheorem*{notation*}{\protect\notationname}
\theoremstyle{plain}
\newtheorem{assumption}{\protect\assumptionname}
\theoremstyle{remark}
\newtheorem{rem}[thm]{\protect\remarkname}
\theoremstyle{plain}
\newtheorem{lem}[thm]{\protect\lemmaname}
\theoremstyle{plain}
\newtheorem{prop}[thm]{\protect\propositionname}

\usepackage{amsthm}
\usepackage{mathrsfs}
\usepackage{nameref}
\usepackage{cite}
\usepackage{upgreek}
\usepackage{stmaryrd}
\usepackage{etoolbox}
\usepackage{enumitem}
\setlist[itemize]{leftmargin=*}
\setlist[enumerate]{leftmargin=*}

\DeclareFontFamily{U}{matha}{\hyphenchar\font45}
\DeclareFontShape{U}{matha}{m}{n}{
      <5> <6> <7> <8> <9> <10> gen * matha
      <10.95> matha10 <12> <14.4> <17.28> <20.74> <24.88> matha12
      }{}
\DeclareSymbolFont{matha}{U}{matha}{m}{n}
\DeclareFontSubstitution{U}{matha}{m}{n}

\DeclareFontFamily{U}{mathx}{\hyphenchar\font45}
\DeclareFontShape{U}{mathx}{m}{n}{
      <5> <6> <7> <8> <9> <10>
      <10.95> <12> <14.4> <17.28> <20.74> <24.88>
      mathx10
      }{}
\DeclareSymbolFont{mathx}{U}{mathx}{m}{n}
\DeclareFontSubstitution{U}{mathx}{m}{n}

\DeclareMathDelimiter{\vvvert}{0}{matha}{"7E}{mathx}{"17}

\DeclareMathAlphabet{\scal}{U}{dutchcal}{m}{n}

\numberwithin{equation}{section}

\def\th@plain{\thm@notefont{}\itshape}
\def\th@definition{\thm@notefont{}\normalfont}

\makeatother

\providecommand{\assumptionname}{Assumption}
\providecommand{\definitionname}{Definition}
\providecommand{\lemmaname}{Lemma}
\providecommand{\notationname}{Notation}
\providecommand{\propositionname}{Proposition}
\providecommand{\remarkname}{Remark}
\providecommand{\theoremname}{Theorem}

\begin{document}
\title{Exit Time Analysis for Kesten's stochastic recurrence equations}
\author{Chang-Han Rhee, Jeeho Ryu, Insuk Seo}

\begin{abstract}
Kesten's stochastic recurrent equation is a classical subject of research in probability theory and its applications. 
Recently, it has garnered attention as a model for stochastic gradient descent with a quadratic objective function and the emergence of heavy-tailed dynamics in machine learning.
This context calls for analysis of its asymptotic behavior under both negative and positive Lyapunov exponents.  
This paper studies the exit times of the Kesten's stochastic recurrence equation in both cases. 
Depending on the sign of Lyapunov exponent, the exit time scales either polynomially or logarithmically as the radius of the exit boundary increases. 
\end{abstract}

\maketitle
\setcounter{tocdepth}{3}

\section{Introduction}
\noindent
\subject{Model}%
\topic{Kesten's recursion}%
The stochastic recurrence equation 
\begin{equation}\label{eq:kesten-intro}
\mathbf X_{n+1} = 
\mathbf A_{n+1} \mathbf X_{n}  + \mathbf B_{n+1}    
\end{equation}
has been a subject of extensive research over the past decades 
\topic{applications}%
and appears in a wide range of contexts, including GARCH time series models and ARMA processes with random coefficients,  random walks in random environments, Mandelbrot cascades, fractals, and stochastic optimization; see, for example,  \cite{basrak2002regular,barnsley2014fractals, hutchinson1981fractals,liu2000generalized, buraczewski2016stochastic,gurbuzbalaban2021heavy} and references therein. 
\topic{\eqref{eq:kesten-intro} is interesting}%
Despite its simplicity, \eqref{eq:kesten-intro} exhibits rich and intricate dynamics, and analyzing these dynamics is highly non-trivial. 
\point{emergence of heavy-tails}%
In particular, the classical results by Kesten and Goldie \cite{kesten1973random, goldie1991, goldie2000stability} reveal that, under certain assumptions, 
the limiting distribution of $\mathbf X_n$ exhibits heavy tails even when $\mathbf A_n$ and $\mathbf B_n$ are light-tailed.
\point{AI connection}%
More recently, due to the associations between the heavy-tailed behaviors of stochastic gradient descent (SGD) and the generalization performance of the neural network it trains,
\chr{references}%
the emergence and characterization of heavy tails in \eqref{eq:kesten-intro} has been revisited in the context of stochastic optimization \cite{hodgkinson2021multiplicative, gurbuzbalaban2021heavy, damek2024analysing}. 

\subject{Motivation: Edge of Stability}%
\topic{EOS background}%
Mathematically, training a deep neural network can be formulated as an optimization problem for the loss function associated with training data, and gradient-based numerical optimization algorithms are the methods of choice. 
Curious phenomena, referred to as the \emph{edge of stability}, were observed in a wide variety of neural network architectures and artificial intelligence (AI) tasks \cite{cohen2021gradient} and drew significant attention from the machine learning theory community;
see, for example, \cite{arora2022understanding, damian2022self, jastrzkebskirelation, jastrzebskibreak}.
\chr{references}%
\topic{EOS phenomenon}\point{progressive sharpening}%
In a nutshell, the observation is that when deterministic gradient descent is used to train deep neural networks, 
there is an overwhelming tendency for the algorithm to be drawn to regions where the curvature of the objective function is sharper---i.e., regions where the maximum eigenvalue of the training loss Hessian is larger.
\point{edge of stability}%
Once the sharpness reaches a threshold, the algorithm enters a regime where the sharpness oscillates around the threshold. 
\point{continued decrease in training loss}%
Interestingly, the objective value (i.e., training loss) consistently decreases over long timescales. 
\topic{Quadratic objective function and EOS}%
\point{threshold coincides with the boundary of the stability region for quadratic objective functions}%

This phenomenon is termed the edge of stability because the sharpness threshold coincides with the boundary of the stability region for deterministic gradient descent when applied to quadratic objective functions. 
Specifically, if $\eta$ denotes the step size (or the learning rate, as referred to in machine learning), 
deterministic gradient descent becomes unstable (explosive) when the curvature of the quadratic objective exceeds $2/\eta$, whereas it becomes stable (contracting) when the curvature is below this threshold. 
This boundary coincides with the empirically observed threshold in the edge of stability phenomena.
The oscillations around the threshold seem to arise as the gradient descent algorithm alternates between contracting and explosive dynamics by repeatedly entering an energy corridor, sliding down to a sharper area of the corridor, and then transitioning into a new corridor. 
The transitions occur via the explosive dynamics as the sharpness exceeds the stability threshold. 
For more details, see \cite{cohen2021gradient}.  

\topic{EOS implications}%
\point{special structure apparent; implication obscure}%
These observations strongly suggest the presence of special structures in the energy landscape of the deep neural networks, with the dynamics of gradient-based algorithms both within and outside of the stability region playing critical roles in training. 
\point{need for analysis of stochastic counterpart}%
However, the practical implications of these findings remain obscure.
To fully understand and leverage such phenomena, it is imperative to examine their stochastic counterpart, 
because, in practice, \emph{deterministic} gradient descent is almost never employed due to the prohibitive computational cost of exact gradient evaluations. 
Instead, \emph{stochastic} gradient descent (SGD) is the method of choice. 
\point{quadratic objective function for SGD}%
Given the critical role of the contraction/explosion dichotomy for quadratic objective functions in the edge of stability phenomena, a natural starting point is to investigate the large excursions of SGD for quadratic objectives. 
For such objective functions, Kesten's recursion \eqref{eq:kesten-intro} provides a good model for the dynamics of SGD---see, for example, \cite{srikant2019finite, hodgkinson2021multiplicative, gurbuzbalaban2021heavy, damek2024analysing} and Section~\ref{sec31}. 
In this paper, we study the asymptotic behaviors of \eqref{eq:kesten-intro}, both within and outside of its stability region.

\subject{Problem Formulation}%
\topic{Contraction/Explosion Dichotomy for $\mathbf X_n$}%
\point{Intuitively, dichotomy should depend on the Lyapunov exponent.}%
At an intuitive level, it seems evident that the asymptotic behavior of $\mathbf{X}_{n}$ depends on the Lyapunov exponent of $\mathbf A_n$. 
\point{case i. negative Lyapunov exponent: Kesten-Goldie suggests contraction}%
Indeed, when the Lyapunov exponent of $\mathbf A_n$ is negative, it is well known \cite{kesten1973random, goldie1991, goldie2000stability} that $\mathbf{X}_{n}$ is generally stable in the sense that the distribution of $\mathbf{X}_{n}$ converges to an equilibrium distribution as $n$ tends to infinity.
\point{case ii. positive Lyapunov exponent: common sense suggests exponential exlosion}%
Conversely, if the Lyapunov exponent is positive, it is natural to conjecture that $\mathbf X_n$ will explode geometrically. 
\topic{Formulation of dichotomy in terms of expected exit times}
Due to stochasticity, however, the dichotomy is more nuanced than in the deterministic case. 
For instance, $\mathbf X_n$ will eventually exhibit arbitrarily large excursion even in the contractive (or stable) case if one waits long enough.
A proper characterization of the dichotomy in the stochastic context can be formulated in terms of the expected exit times: 
if the system is contractive, the exit time scales superlinearly, whereas if the system is explosive, it scales sublinearly.  
In the case the Lyapunov exponent is negative, and all the entries of $A_n$ and $B_n$ are non-negative, it has been shown in \cite{collamore2018large} that the exit time scales polynomially, with both the polynomial rate and the prefactor explicitly identified. 
However, as illustrated in Section~\ref{sec3}, it is crucial to allow for negative entries and address both contractive and explosive cases to accommodate the applications motivating this paper. 

\subject{Contribution}%
In this paper, we prove that the exit times scale polynomially in the contractive case and logarithmically in the explosive case.
Notably, our results do not require the entries of $\mathbf A_n$ and $\mathbf B_n$ to be non-negative. 
Specifically, Theorem \ref{Thm_main_cont} establishes that when the Lyapunov exponent of random matrix $\mathbf A_n$ is negative,
\[
\lim_{R\rightarrow\infty}\frac{\log\mathbb{E}[\tau_{R}(\mathbf{x}_{0})]}{\log R}=\alpha\;\;\;\text{for all }\mathbf{x}_{0}\in\mathbb{R}^{d}\;,
\]
where $\tau_{R}(\mathbf{x})$ is the exit time from the closed ball 
$
\mathcal{B}_{R}=
\{\mathbf{y}\in\mathbb{R}^{d}:|\mathbf{y}|\le R\}
$ and $\alpha>0$ is the solution to the equation $h_{\mathbf{A}}(\alpha)=1$, with $h_{\mathbf{A}}$ defined in \eqref{eq:h_A}.
This $\alpha$ coincides with the index of the power-law associated with the tail of the stationary distribution of $\mathbf X_n$.
On the other hand, Theorem \ref{Thm_main_expl} shows that when the Lyapunov exponent is positive, 
\[
\frac{1}{\gamma_{L}}\le\liminf_{R\rightarrow\infty}\frac{\mathbb{E}[\tau_{R}(\mathbf{x}_{0})]}{\log R}\le\limsup_{R\rightarrow\infty}\frac{\mathbb{E}[\tau_{R}(\mathbf{x}_{0})]}{\log R}<\infty\;.
\]
where $\gamma_L$ denotes Lyapunov exponent.  
Summing up, under a reasonable set of assumptions,
we show that in the regime $R\rightarrow\infty$, the mean exit time grows as follows:
\[
\mathbb{E}\left[\tau_{R}(\mathbf{x}_{0})\right]\sim\begin{cases}
R^{\alpha} & \text{if Lyapunov exponent is negative}\\
\log R & \text{if Lyapunov exponent is positive}
\end{cases}
\]
where $\alpha>0$. 
Thus, the behavior of the process \eqref{eq:AMP} undergoes a qualitative shift as the Lyapunov exponent changes sign. 
In the critical case, i.e., the case $\gamma_{L}=0$, preliminary simulation experiments suggest that the mean exit time grows poly-logarithimically. 
However, at present, it is unclear whether this is the only possible behavior.
The study of the critical case appears to be very challenging and is an interesting direction for future research. 

\subject{Related literature}%
\topic{Exit Times of Processes with Scaling Limit}%
Finally, we note that exit time analysis for stochastic processes has a long and rich history in probability theory.
Classical results are the Eyring-Kramers formula
\cite{eyring1935chemical, kramers1940brownian, glasstone1941theory} in chemical reaction rate theory, 
Freidlin-Wentzell theory and pathwise approach based on large deviations
\cite{freidlin1970onsmall, freidlin1973some, freidlin1984random, cassandro1984metastable, Olivieri_Vares_2005},
and potential theoretic approach \cite{bovier2001metastability, bovier2004metastability, bovier2005metastability, bovier2016metastability}.
More recent developments include
extensions of the potential theoretic approach to non-reversible Markov chains\cite{slowik2012note, landim2014metastability, gaudilliere2014dirichlet, lee2022non}, 
as well as analyses of processes driven by heavy-tailed noise \cite{imkeller2006first, imkeller2008levy, pavlyukevich2008metastable, imkeller2010first, wang2023large}. 
In particular, the heavy-tailed context has gained renewed attention due to the empirical success of deep neural networks and the observed association between their remarkably good performance and SGD's ability to escape local attraction fields and explore multiple local minima \cite{nguyen2019first, wang2022eliminating}.
\chr{more references in particular, light-tailed contexts}%
\topic{Our paper considers different regime}%
It is important to note that the aforementioned results primarily focus on how the exit times scale as the learning rate (or step size) decreases, and hence, investigates a different asymptotic regime from that studied in \cite{collamore2018large} and in this paper.


\subject{Outline of the paper}%
The remainder of the article is organized as follows. 
Section~\ref{sec2} presents our main results. 
In Section \ref{sec3},
we provide several important time series models for which we can apply
our main results. In Sections \ref{sec4} and \ref{sec5}, we analyze
the contractive regime and prove Theorems \ref{Thm_main_cont} and
\ref{Thm_main_cont_uni}. In Section \ref{sec6}, we investigate the
explosive regime and prove Theorem \ref{Thm_main_expl}. In Appendix
A, we provide some sufficient conditions to check Assumption \ref{Ass_Cont4},
and in Appedix B, we explain why the examples given in Section \ref{sec3}
fall into our framework.

\section{Preliminaries}
This section sets the notation and reviews known results regarding Kesten's stochastic recurrence equation. 
Throughout the paper, we write
\newnota{Z-mathbb^+}{$\mathbb{Z}^{+}$}
and 
\newnota{Z-mathbb_0^+}{$\mathbb{Z}_{0}^{+}$} 
to denote the 
\defnota{Z-mathbb^+}{set of positive integers} 
and the
\defnota{Z-mathbb_0^+}{set of non-negative integers}, respectively. 
\begin{defn}
\label{def:AMP}Let $d\ge1$ and let 
$(\mathbf{A}_{n},\,\mathbf{B}_{n})_{n\in\mathbb{Z}^{+}}$
be an independent and identically distributed sequence of random matrices
and vectors, where 
\newnota{A-mathbf-n}{$\mathbf{A}_{n}$}
is a
\defnota{A-mathbf-n}{$d\times d$ random matrix}
and 
\newnota{B-mathbf-n}{$\mathbf{B}_{n}$}
is a 
\defnota{B-mathbf-n}{$d\times1$ random vector}. 
Then, Kesten's stochastic recurrence equation on $\mathbb{R}^{d}$
associated with $(\mathbf{A}_{n},\mathbf{B}_{n})_{n\in\mathbb{Z}^{+}}$
refers to the process 
\newnota{X-mathbf-n}{$(\mathbf{X}_{n})_{n\in\mathbb{Z}_{0}^{+}}$} 
\mdefnota{X-mathbf-n}{$\mathbf{X}_{n+1}=\mathbf{A}_{n+1}\mathbf{X}_{n}+\mathbf{B}_{n+1}$}
defined by 
\begin{equation}\label{eq:AMP}
    \mathbf{X}_{n+1}=\mathbf{A}_{n+1}\mathbf{X}_{n}+\mathbf{B}_{n+1},\quad n\in\mathbb{Z}_{0}^{+}.
\end{equation}
\begin{enumerate}
\item 
    For $\mathbf{x}\in\mathbb{R}^{d}$, we denote by 
    \newnota{X-mathbf-n(mathbf-x)}
        {$(\mathbf{X}_{n}(\mathbf{x}))_{n\in\mathbb{Z}_{0}^{+}}$}
    \mdefnota{X-mathbf-n(mathbf-x)}
        {$\mathbf{X}_{n}$ with initial condition $\mathbf X_0 = \mathbf x$}
    the process \eqref{eq:AMP} starting at $\mathbf{x}$, i.e., $\mathbf{X}_{0}(\mathbf{x})=\mathbf{x}$. 
    \chra{$(\mathbf{X}_{n}(\mathbf{x}))_{n\in\mathbb{Z}_{0}^{+}}$ and $(\mathbf{X}_{n}(\mu))_{n\in\mathbb{Z}_{0}^{+}}$ should be
    $(\mathbf{X}_{n}(\mathbf{x}))_{n\in\mathbb{Z}_0^{+}}$ and $(\mathbf{X}_{n}(\mu))_{n\in\mathbb{Z}_0^{+}}$}%
\item 
    For a Borel probability measure $\mu$ on $\mathbb{R}^{d}$, we denote by 
    \newnota{X-mathbf-n(mu)}
        {$(\mathbf{X}_{n}(\mu))_{n\in\mathbb{Z}_{0}^{+}}$}
    \mdefnota{X-mathbf-n(mu)}
        {$\mathbf{X}_{n}$ with initial condition $\mathbf X_0 \sim \mu$}
    the process driven by \eqref{eq:AMP} with initial distribution
    $\mu$.
\item 
 We denote by 
    \newnota{Omega-mathcal-F-mathbb-P}
    {$(\Omega,\,\mathcal{F},\,\mathbb{P})$} 
    \defnota{Omega-mathcal-F-mathbb-P} the probability space associated with the randomness introduced so far, i.e., with the elements
    $(\mathbf{A}_{n},\,\mathbf{B}_{n})_{n\in\mathbb{Z}^{+}}$ and the initial distribution of the process. We denote by $\mathbb{E}$ the expectation associated with $\mathbb{P}$. 
\item Denote by 
    \newnota{F-mathcal-n}{$\mathcal{F}_{n}$} 
    \defnota{F-mathcal-n}{the $\sigma$-algebra on $\Omega$ generated
    by $(\mathbf{A}_{k},\,\mathbf{B}_{k})_{k=1}^{n}$} 
    and $\mathcal{F}_{0}$ the null $\sigma$-algebra.
    Then, it is clear from \eqref{eq:AMP} that $\mathbf{X}_{n}(\mathbf{x})$ is $\mathcal{F}_{n}$-measurable
    for each $n\in\mathbb{Z}_{0}^{+}$ and furthermore $(\mathbf{X}_{n}(\mathbf{x}))_{n\in\mathbb{Z}_{0}^{+}}$
    is a Markov process adapted to the filtration $(\mathcal{F}_{n})_{n\in\mathbb{Z}_{0}^{+}}$.
\end{enumerate}
\end{defn}

For $R>0$, denote by $\tau_{R}(\mathbf{x})$ the exit time of the process $\mathbf{X}_{n}(\mathbf{x})$
from the ball 
\begin{equation}
\newnota{B-mathcal_R}{\mathcal{B}_{R}}=
\defnota{B-mathcal_R}{\{\mathbf{y}\in\mathbb{R}^{d}:|\mathbf{y}|\le R\}},\label{eq:B_R}
\end{equation}
i.e., 
\[
\newnota{tau_R}{\tau_{R}(\mathbf{x})}
:=
\defnota{tau_R}
    {\inf\{n\in\mathbb{Z}_{0}^{+}:|\mathbf{X}_{n}(\mathbf{x})|>R\}},
\] where we set $\inf \phi :=\infty $ as usual. In this article, our primary concern is the asymptotic behavior of
the mean exit time $\mathbb{E}[\tau_{R}(\mathbf{x})]$, in the regime
$R\rightarrow\infty$. \\
For $d\times d$ matrix $\mathbf{M}$, we write 
\[
\dnewnota{Vert-Vert}{\Vert\cdot\Vert}
    {\Vert\mathbf{M}\Vert}
:=
    {\sup_{|\mathbf{x}|=1}|\mathbf{M}\mathbf{x}|}
\]
\mdefnota{Vert-Vert}{%
${\Vert\mathbf{M}\Vert}:={\sup_{|\mathbf{x}|=1}|\mathbf{M}\mathbf{x}|}$
}%
the matrix norm of $\mathbf{M}$. 
Throughout this article, it will be frequently used that this norm is sub-multiplicative: for any $d\times d$ matrices $\mathbf{M}$ and $\mathbf{N}$,
\begin{equation}
\Vert\mathbf{M}\mathbf{N}\Vert\le\Vert\mathbf{M}\Vert\Vert\mathbf{N}\Vert\;.\label{eq:subm-1}
\end{equation}
For $n\in\mathbb{Z}^{+}$, define 
\begin{equation}
\newnota{Pi_n}{\Pi_{n}}
:=
\defnota{Pi_n}{\mathbf{A}_{n}\mathbf{A}_{n-1}\cdots\mathbf{A}_{1}}
.
\label{eq:Pi_n}
\end{equation}
Then, it is well-known (cf. \cite[Theorem 1, page 457]{furstenberg1960products} and \cite[Theorem 2, page 460]{furstenberg1960products})
that the limit 
\begin{equation}
\newnota{gamma_L}{\gamma_L}=
\lim_{n\rightarrow\infty}\frac{1}{n}\mathbb{E}\left[\log\Vert\Pi_{n}\Vert\right]
\label{eq:lya}
\end{equation}
exists, and moreover it almost surely holds that 
\begin{equation}
\gamma_{L}=\lim_{n\rightarrow\infty}\frac{1}{n}\log\Vert\Pi_{n}\Vert
\label{eq:lyaas}
\end{equation}
\mdefnota{gamma_L}{$\displaystyle \lim_{n\rightarrow\infty}\frac{1}{n}\mathbb{E}\left[\log\Vert\Pi_{n}\Vert\right] \stackrel{a.s.}{=} \lim_{n\rightarrow\infty}\frac{1}{n}\log\Vert\Pi_{n}\Vert$}%
by the multiplicative ergodic theorem. The constant $\gamma_{L}$
is called a \newnota{Lyapunov-exponent}{\emph{Lyapunov exponent}}\mdefnota{Lyapunov-exponent}{$\gamma_L = \displaystyle \lim_{n\rightarrow\infty}\frac{1}{n}\mathbb{E}\left[\log\Vert\Pi_{n}\Vert\right] \stackrel{a.s.}{=} \lim_{n\rightarrow\infty}\frac{1}{n}\log\Vert\Pi_{n}\Vert$} associated with sequence $(\mathbf{A}_{n})_{n\in\mathbb{Z}^{+}}.$

At a heuristic level, it is clear that the behavior of the process $(\mathbf{X}_{n}(\mathbf{x}))_{n\in\mathbb{Z}_{0}^{+}}$ depends crucially on the sign of
the Lyapunov exponent $\gamma_{L}$. 
If $\gamma_{L}<0$, as the elements of $\Pi_{n}$ converges to $0$ exponentially fast, the process $\mathbf{X}_{n}$
contracts to its stable attractors.
On the other hand, if $\gamma_{L}>0$, as the elements of $\Pi_{n}$ exponentially diverges to $+\infty$,
we can expect that the process $\mathbf{X}_{n}$ is explosive in the sense that it escapes any bounded domain very quickly. 
The investigation of the mean exit time $\mathbb{E}[\tau_{R}(\mathbf{x})]$ also depends on the sign of the Lyapunov exponent.

\subsection{Known results}

For the contractive regime, i.e., for the case when $\gamma_{L}$
defined in \eqref{eq:lya} is negative, the asymptotic behavior of the process \eqref{eq:AMP} has been investigated extensively over the past decades. 
We briefly summarize important results before introducing our main contributions. 

\subsubsection*{Existence and uniqueness of stationary measure}
The following set of assumptions is commonly assumed in the study
of the contractive regime. 
\begin{notation}
We denote by $(\mathbf{A},\,\mathbf{B})$ the pair of random matrix
and vector having the same distribution with $(\mathbf{A}_{n},\,\mathbf{B}_{n})$
appeared in Definition \ref{def:AMP}. 
\end{notation}

\begin{assumption}[Standard conditions for contractive regime]
\label{Ass_Cont1} $\,$
\begin{enumerate}
\item The Lyapunov exponent is negative, i.e., $\gamma_{L}<0$.
\item There exists $s>0$ such that $\mathbb{E}\left[\Vert\mathbf{A}\Vert^{s}\right]<\infty$.
\item $\mathbb{E}\log^{+}|\mathbf{B}|<\infty$ where $\log^{+}x:=\log(x\lor1)$.
\item There is no fixed point in the sense that $\mathbb{P}(\mathbf{A}\mathbf{x}+\mathbf{B}=\mathbf{x})<1$
holds for all $\mathbf{x}\in\mathbb{R}^{d}$. 
\end{enumerate}
\end{assumption}
 
We will always work under Assumption~\ref{Ass_Cont1} in the context of a contractive regime. 
In the next theorem, we recall Kesten's theorem (cf. \cite{kesten1973random}), which proves that Assumption \ref{Ass_Cont1} ensures the existence of a unique stationary distribution. 
\begin{thm}[Kesten's Theorem]
\label{Thm_kes0}Under Assumption \ref{Ass_Cont1}, there exists a
unique stationary distribution \newnota{nu_infty}{$\nu_{\infty}$}\mdefnota{nu_infty}{unique stationary distribution of $\mathbf X_n$} of the process \eqref{eq:AMP}. 
\end{thm}

\begin{proof}
We refer to \cite[Theorem 6. page 247]{kesten1973random} for a proof. 
\end{proof}

\subsubsection*{Tail of the stationary measure}

An important object in the study of the contractive regime is the
function \newnota{h_A(s)}{$h_{\mathbf{A}}:[0,\,\infty)\rightarrow[0,\,\infty]$} defined
by
\begin{equation}
\defnota{h_A(s)}{
h_{\mathbf{A}}(s)=
\lim_{n\rightarrow\infty}\left[\mathbb{E}\Vert\Pi_{n}\Vert^{s}\right]^{\frac{1}{n}}}\;\;\;\;;\;s\ge0\;.\label{eq:h_A}
\end{equation}
whose existence is guaranteed by Kingman's subadditive theorem; see
\cite[page 167]{buraczewski2016stochastic}. We now summarize basic facts regarding this
function. 
\begin{enumerate}
\item Although in general it is notoriously difficult to
compute the exact value of $h_{\mathbf{A}}(s)$, we can readily get
lower and upper bounds when $\mathbb{P}(\det\mathbf{A}=0)=0$.
\chra{How strong is this condition?}
\jhra{Elaborated in the verification of applications section}%
Inserting (cf.\ \eqref{eq:subm-1}) the bound 
\[
\frac{1}{\Vert\mathbf{A}_{1}^{-1}\Vert\cdots\Vert\mathbf{A}_{n}^{-1}\Vert}\le\Vert\Pi_{n}\Vert\le\Vert\mathbf{A}_{1}\Vert\cdots\Vert\mathbf{A}_{n}\Vert\;,
\]
to the definition \eqref{eq:h_A}, we get 
\begin{equation}
\mathbb{E}\frac{1}{\Vert\mathbf{A}^{-1}\Vert^{s}}\le h_{\mathbf{A}}(s)\le\mathbb{E}\Vert\mathbf{A}\Vert^{s}\;,\;\;\;\quad s\ge0\;.\label{eq:lbh}
\end{equation}
\item Write\footnote{We note that in all of our application explained in Section \ref{sec3},
the tail of $\mathbf{A}$ is light so that $\alpha_{\infty}=\infty$.}
\mdefnota{alpha_infty}{$\alpha_{\infty}=\sup\left\{ s:\mathbb{E}\Vert\mathbf{A}\Vert^{s}<\infty\right\}$}
\begin{equation}
\newnota{alpha_infty}{\alpha_\infty}:=\sup\left\{ s:\mathbb{E}\Vert\mathbf{A}\Vert^{s}<\infty\right\},\label{eq:s_infty}
\end{equation}
where $\alpha_{\infty}>0$ by Assumption \ref{Ass_Cont1}-(2), so
that by \eqref{eq:lbh}, we have 
\[
h_{\mathbf{A}}(\beta)<\infty\text{ for all }\beta\in[0,\,\alpha_{\infty})\;.
\]
\item By the H\"{o}lder inequality, we can readily check that the function
$\log h_{\mathbf{A}}$ and hence $h_{\mathbf{A}}$ is convex on $[0,\,\infty)$.  
Furthermore, it is known that in the case $\gamma_L<0$, the function $h_{\mathbf{A}}$ is decreasing at $0$ due to Theorem \ref{Thm_h_A} below. 
From this observation, we notice that there are two possible cases
as illustrated in Figure \ref{fig:h_A}: 
\begin{enumerate}
\item Firstly, the function $h_{\mathbf{A}}$ may be a decreasing function
as in Figure \ref{fig:h_A}-(left), and in this case Goldie-Gr\"{u}bel
theorem (e.g., \cite[Theorem 2.4.1]{buraczewski2016stochastic}) shows that, with several
additional technical assumptions, the tail of $\nu_{\infty}$ is exponentially
light. This case is of independent interest but will not be handled
in the current article.
\chra{The left panel of the figure 2.1 is missing.}
\item Another possibility is the case where $h_{\mathbf{A}}(s)$ diverges
to $+\infty$ as $s$ increases to $+\infty$ as in Figure \ref{fig:h_A}-(right),
and this case is handled by Kesten in \cite{kesten1973random}. 
We will focus mainly on this case; see Assumption~\ref{Ass_Cont2}.
In this  case, as $h_{\mathbf{A}}(0)=1$, by
the intermediate value theorem, we have a unique solution to the equation
$h_{\mathbf{A}}(s)=1$ (cf. as in Figure \ref{fig:h_A}-(right)). 

\end{enumerate}
\end{enumerate}

\begin{figure}[h!]
\label{fig:h_A}
    \centering
    \begin{minipage}{0.45\textwidth}
        \centering
        \includegraphics[width=\linewidth]{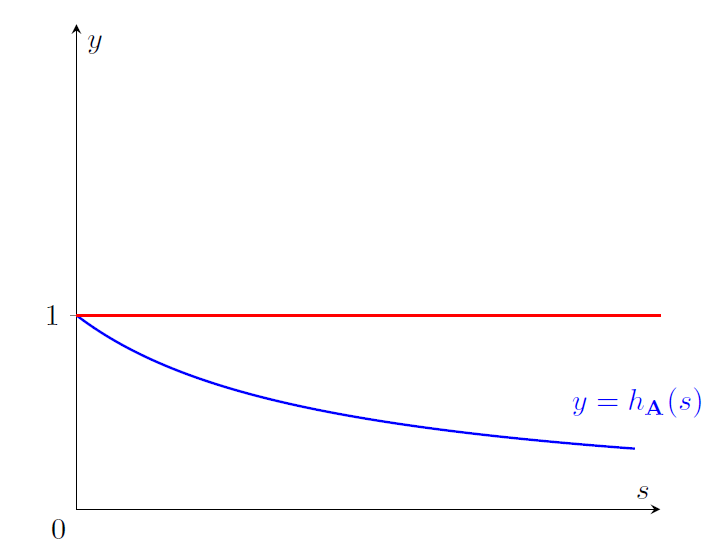}
        \label{fig:figure1}
    \end{minipage}
    \hfill
    \begin{minipage}{0.45\textwidth}
        \centering
        \includegraphics[width=\linewidth]{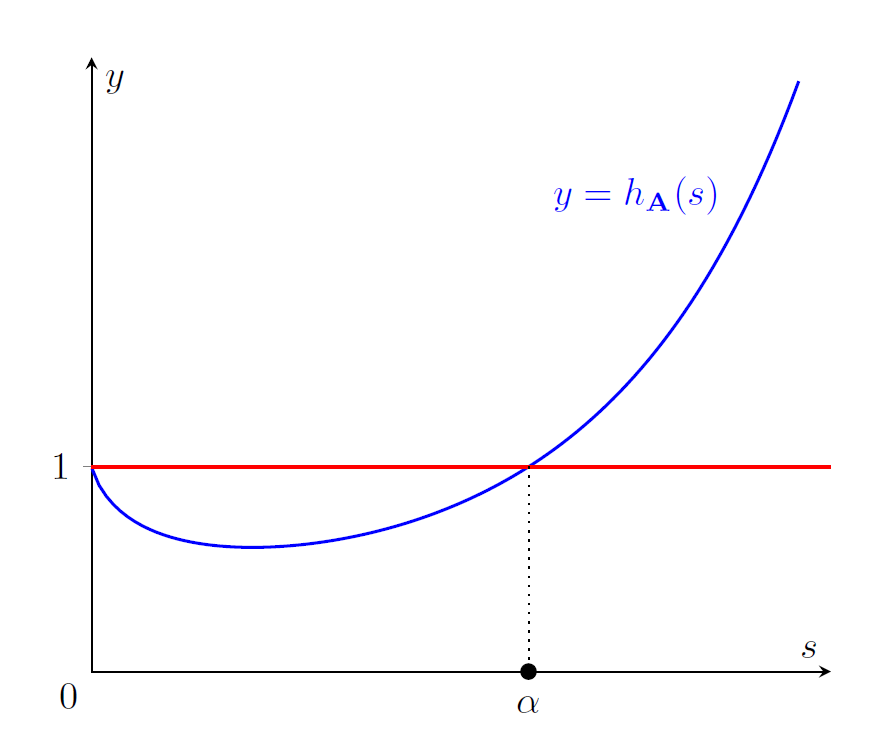}
        \label{fig:figure2}
    \end{minipage}
    \caption{Two possible shapes of $h_\mathbf{A}$.}
\end{figure}

\begin{thm}
\label{Thm_h_A} {Assume that the   Lyapunov exponent $\gamma_L$ is negative and $h_\mathbf{A}(s)$ is finite for some $s>0$. Then there exists $s_0>0$ such that $h_\mathbf{A}(s_0)<1$.}
\end{thm}

\begin{proof}
  We refer to \cite[Lemma 4.4.2. page 168]{buraczewski2016stochastic} or \cite[page 231]{kesten1973random} for the proofs.  
\end{proof}

\begin{assumption}[\emph{Condition on $h_{\mathbf{A}}$}]
\label{Ass_Cont2} 
There exists $s_1>0$ such that $h_{\mathbf{A}}(s_1)>1$. In particular, under the Assumption \ref{Ass_Cont1}, by the continuity of $h_{\mathbf{A}}$ which follows from its convexity, Theorem \ref{Thm_h_A}, and the intermediate value theorem, there exists \newnota{alpha}{$\alpha\in(0,\,\alpha_{\infty})$}
such that \defnota{alpha}{$h_{\mathbf{A}}(\alpha)=1$}.
\end{assumption} 

\begin{rem}[Remarks on Assumption \ref{Ass_Cont2}]
\label{rem:h_A} $\,$
\begin{enumerate}
\item Throughout our discussion on contractive regime in the current article,
the constant $\alpha>0$ always refer to the one appeared in Assumption
\ref{Ass_Cont2}.
\item In order to check Assumption \ref{Ass_Cont2} for a specific model,
by \eqref{eq:lbh} and theorem \ref{Thm_h_A} above, it suffices to find
$s>0$ such that $\mathbb{E}\frac{1}{\Vert\mathbf{A}^{-1}\Vert^{s}}>1$.
Clearly, the last inequality is guaranteed if and
only if $\mathbb{P}(\Vert\mathbf{A}^{-1}\Vert<1)>0$. Thus, this assumption
is easy to check whenever $\mathbf{A}$ is invertible. 
\item By the convexity of $h_{\mathbf{A}}$, we have $h_{\mathbf{A}}<1$
on $(0,\,\alpha)$ and $h_{\mathbf{A}}>1$ on $(\alpha,\,\infty)$. 
\end{enumerate}
\end{rem}

Kesten \cite{kesten1973random} proved that under Assumptions \ref{Ass_Cont1} and \ref{Ass_Cont2},
along with several technical assumptions, the stationary distribution
$\nu_{\infty}$ has a heavy-tailed distribution with a power-law tail of
exponent $\alpha$ given in Assumption \ref{Ass_Cont2}. More precisely,
it is proven that, there exists a positive function $g:\mathbb{S}^{d-1}\rightarrow(0,\,\infty)$
such that, for all $\mathbf{u}\in\mathbb{S}^{d-1}$,
\begin{equation}
\lim_{z\rightarrow\infty}z^{\alpha}\nu_{\infty}(\{\mathbf{x}:\mathbf{u}\cdot\mathbf{x}>z\})=g(\mathbf{u}),
\label{eq:kesten}
\end{equation}
where 
\begin{equation}
\newnota{S-mathbb^d-1}{\mathbb{S}^{d-1}}
:=
\defnota{S-mathbb^d-1}{\left\{ \mathbf{x}\in\mathbb{R}^{d}:|\mathbf{x}|=1\right\}}
.
\label{eq:sphere}
\end{equation}
This result, known as Kesten's theorem for process \eqref{eq:AMP},
is interesting in that the stationary distribution is power-law
tailed even in the case where $\mathbf{A}$ and $\mathbf{B}$ are
light-tailed (e.g., normal-type distribution). An interested reader
may refer to the original paper \cite[Theorem 6]{kesten1973random} or textbook 
\cite[page 169-171]{buraczewski2016stochastic} for more details.

\section{\label{sec2}Main Results}

\subsection{Main results for the contractive regime}

This section presents our main result on the contractive regime, which provides
asymptotics of the mean exit time $\mathbb{E}[\tau_{R}(\mathbf{x})]$
for $\mathbf{x}\in\mathbb{R}^{d}$ in the regime $R\rightarrow\infty$. 

\subsubsection*{Technical assumptions }
Our main result, stated in Theorem \ref{Thm_main_cont}, requires additional assumptions. 
Kesten's theorem \cite[Theorem 6]{kesten1973random} provides a set of sufficient conditions. 
However, working with them require further introduction of definitions and notations, and the verification of those assumptions are difficult for the example we consider. 
Instead, we provide a set of assumptions which are concise and easy to confirm
for the examples given in Section \ref{sec3}. 
Our assumptions are weaker than those of \cite[Theorem 6]{kesten1973random}; see Remarks \ref{rem:ass_cont3}, \ref{rem:ass_cont4}, and \ref{rem:ass_cont5}.
\chra{Do they still guarantee existence and uniqueness of stationary distribution?}%
\jhra{Assumption 1 ensures it and we assume the assumption 1 throughout the whole text}%

The first additional assumption is on the unboundedness of the support
of the stationary distribution. 
\begin{assumption}
\label{Ass_Cont3}The support of $\nu_{\infty}$ is unbounded. 
\end{assumption}

This assumption is obviously necessary in our context. 
Without this, the process \eqref{eq:AMP}
is confined to the support of $\nu_{\infty}$, and hence, the mean of
the exit time from a domain $\mathcal{B}_{R}$ with sufficiently large
$R$ would be $+\infty$. 
\begin{rem}[Remarks on Assumption \ref{Ass_Cont3}]\label{rem:ass_cont3}$\,$
\begin{enumerate}
\item In view of \eqref{eq:kesten}, the assumptions of \cite[Theorem 6]{kesten1973random}
imply Assumption \ref{Ass_Cont3}. In particular, the assumption of
\cite[Condition (A) in Section 4.4]{kesten1973random,buraczewski2016stochastic} regarding the denseness
of the additive subgroup of $\mathbb{R}$ generated by the log of
absolute value of largest eigenvalue of $\Pi_{n}$ implies Assumption
\ref{Ass_Cont3}. 
\item Note that, when the support of a random vector $\mathbf{Ax+B}$ is
unbounded for all $\mathbf{x}\in\mathbb{R}^{d}$, then it is immediate
that the Assumption \ref{Ass_Cont3} holds. This is indeed the case
for the models we discuss in Section \ref{sec3}. 
More detailed scheme to verify Assumptipon \ref{Ass_Cont3} when the support
of $\mathbf{A}$ and $\mathbf{B}$ are bounded is explained in  \cite[Proposition 4.3.1, page 160]{buraczewski2016stochastic} at which a characterization of the support $\nu_{\infty}$ is provided. 
\end{enumerate}
\end{rem}

The following technical assumption prevents the effect of single outcome of random variable
$\mathbf{A},\mathbf{B}$ from dominating the long-term dynamics of \eqref{eq:AMP}; see, for example, \cite{konstantinides2005large, buraczewski2012asymptotics}.
Again, it is weaker than assumptions of \cite[Theorem 6]{kesten1973random}.
\begin{assumption}
\label{Ass_Cont4}There exist $R_{0},\,z_{0}>1$, $C_{0}>0$, and
\newnota{alpha-plus}{$\alpha_{+}\in(\alpha,\,\alpha_{\infty}]$} such that, we have 
\begin{equation}
\defnota{alpha-plus}{
\mathbb{P}\big(|\mathbf{A}\mathbf{x}+\mathbf{B}|>zR\big)\le\frac{C_{0}}{z^{\alpha_{+}}}\,\mathbb{P}\big(|\mathbf{A}\mathbf{x}+\mathbf{B}|>R\big)}\label{eq:nc2-1}
\end{equation}
for all $R\in[R_{0},\,\infty),\,z\in[z_{0},\,\infty)$ and $\mathbf{x}\in\mathcal{B}_{R}$.
\end{assumption}

This assumption guarantees that $\mathbf{A}\mathbf{x}+\mathbf{B}$
has tails strictly lighter than power-law distribution with index
$\alpha$. If it has thicker tails than power-law distribution with
index $\alpha$, the overall tail behavior of the process \eqref{eq:AMP} is simply
dominated by the tail behavior of either $\mathbf{A}$ or $\mathbf{B}$, rather
than the structural properties of stochastic recurrence equations. 
\begin{rem}[Remarks on Assumption \ref{Ass_Cont4}]$\,$
\label{rem:ass_cont4}

\begin{enumerate}
\item We say that the random matrix $\mathbf{A}$ or the random vector $\mathbf{B}$
has a tail lighter than the power-law distribution of index $\alpha$
if we can find $R_{1},\,z_{1}>1$, $C_{1}>0$, and $\alpha_{1}>\alpha$
such that 
\begin{equation}
\mathbb{P}(|\mathbf{A}\mathbf{x}|>Rz)\le\frac{C_{1}}{z^{\alpha_{1}}}\,\mathbb{P}(|\mathbf{A}\mathbf{x}|>R)\;\;\;\;\text{or}\;\;\;\;\mathbb{P}(|\mathbf{B}|>Rz)\le\frac{C_{1}}{z^{\alpha_{1}}}\,\mathbb{P}(|\mathbf{B}|>R)\label{eq:tail_cond}
\end{equation}
for all $R\in[R_{1},\,\infty),\,z\in[z_{1},\,\infty)$ and $\mathbf{x}\in\mathbb{S}^{d-1}$.
Then, one can conjecture that Assumption
\ref{Ass_Cont4} is valid when both $\mathbf{A}$ and $\mathbf{B}$
have tails lighter than power-law distribution of index $\alpha$.
 
\item Considering the case $\mathbf{x}=\mathbf{0}$ in \eqref{eq:nc2-1}, we note that Assumption
\ref{Ass_Cont4} implies that the random vector $\mathbf{B}$ has
a tail lighter than power-law distribution of index $\alpha$. In
particular, by applying the layer-cake formula, we get 
\begin{equation}
\mathbb{E}|\mathbf{B}|^{\gamma}<\infty\;\;\;\text{for all }\gamma\in[0,\,\alpha_{+})\;.\label{eq:nc2-2}
\end{equation}
 
\end{enumerate}
\end{rem}

The final assumption concerns the non-degeneracy of the process \eqref{eq:AMP}. 
\begin{assumption}
\label{Ass_Cont5}There exists $n_{0}\in\mathbb{Z}^{+}$ such that,
for all $\mathbf{x},\,\mathbf{y}\in\mathbb{R}^{d}\setminus\left\{ \mathbf{0}\right\} $
(cf. \eqref{eq:Pi_n})
\[
\mathbb{P}\left[\mathbf{x}\cdot\Pi_{n_{0}}\mathbf{y}=0\right]<1\;.
\]
\end{assumption}

\begin{rem}[Remarks on Assumption \ref{Ass_Cont5}]\label{rem:ass_cont5}$\,$
\begin{enumerate}
\item Of course, this assumption with $n_{0}=1$ holds if the elements of $\mathbf{A}$ are stochastically linearly independent in the sense that 
\[ \mathbb{P} \left( \sum_{i,j=1}^{d} c_{ij}\mathbf{A}_{i,j}=0 \right)=1 \] implies that all the coefficients $c_{ij}$ are zero.
We note that our models in Section \ref{sec3} satisfy this assumption. 
\item On the other hand, our assumption is strictly weaker than this invertibility
assumption. 
For instance, let $d=2$ and $\mathbf{A}=\begin{pmatrix}c & cZ\\
cZ & cZ^{2}
\end{pmatrix}$ where $Z$ is a standard normal random variables, and $c>0$ is a
constant small enough to guarantee that the Lyapunov exponent associated
with $\mathbf{A}$ is negative. Then $\mathbf{A}$ is always non-invertible,
while for all $\mathbf{x},\mathbf{y}\in\mathbb{R}^{d}\setminus\left\{ \mathbf{0}\right\} $,
we have $\mathbb{P}(\mathbf{x}\cdot\mathbf{A}\mathbf{y}=0)=0$ so
that Assumption \ref{Ass_Cont5} is satisfied with $n_{0}=1$. 
\end{enumerate}
\end{rem}

\subsubsection*{Main results}

Now we are ready to state main results. We assume in the next theorem
that Assumptions \ref{Ass_Cont1}, \ref{Ass_Cont2}, \ref{Ass_Cont3},
\ref{Ass_Cont4} and \ref{Ass_Cont5} are in force. 
We emphasize that, taken together, all of these assumptions are weaker than those of Kesten's
theorem. 
\begin{thm}
\label{Thm_main_cont} For all $\mathbf{x}_{0}\in\mathbb{R}^{d}$,
it holds that 
\[
\lim_{R\rightarrow\infty}\frac{\log\mathbb{E}[\tau_{R}(\mathbf{x}_{0})]}{\log R}=\alpha\;\;\;\text{for all }\mathbf{x}_{0}\in\mathbb{R}^{d}\;,
\]
where $\alpha>0$ is the exponent that appeared in Assumption \ref{Ass_Cont2}. 
\end{thm}

The proof of Theorem \ref{Thm_main_cont}, given in Section \ref{sec5}, is based on the refined construction of suitable sub- and super-martingales.
More precisely, 
we prove in Propositions \ref{prop:cont_ubd} and \ref{prop:cont_lbd}
that, for all $\gamma_{1}\in(0,\,\alpha)$ and $\gamma_{2}\in(\alpha,\,\alpha_{\infty})$
where $\alpha_{\infty}$ is the one defined in \eqref{eq:s_infty},
we find constants $C_{\gamma_{1}},\,C'_{\gamma_{1}},\,C_{\gamma_{2}}>0$
such that for all $\mathbf{x}_{0}\in\mathbb{R}^{d}$ and $R>0$, 
\begin{equation}
C_{\gamma_{1}}R^{\gamma_{1}}-C'_{\gamma_{1}}\left(|\mathbf{x}_{0}|^{\gamma_{1}}+1\right)\le\mathbb{E}\left[\tau_{R}(\mathbf{x}_{0})\right]\le C_{\gamma_{2}}\left(R^{\gamma_{2}}+1\right)\;.\label{eq:bdn}
\end{equation}
Then, Theorem \ref{Thm_main_cont} is a direct consequence of this
bound. 

\subsubsection*{Univariate case }

Theorem \ref{Thm_main_cont} implies that $\mathbb{E}[\tau_{R}(\mathbf{x}_{0})]$ grows on the order of $R^{\alpha}$ possibly multiplied by a sub-polynomial prefactor. 
One can also expect more precise estimate, e.g., 
\begin{equation}
\lim_{R\rightarrow\infty}\frac{\mathbb{E}[\tau_{R}(\mathbf{x}_{0})]}{R^{\alpha}}=C
\label{eq:cont_conj}
.
\end{equation}
Obtaining such a sharp estimate in full generality under the current set of assumptions seems to be very difficult, as proving a bound of the form \eqref{eq:bdn} is not feasible at this time for $\gamma_{1}=\alpha$ or $\gamma_{2}=\alpha$. 
However, for the univariate case, namely the case with $d=1$, we can get a stronger result close to \eqref{eq:cont_conj}. 
The main
benefit of the univariate case is that the Lyapunov exponent $\gamma_{L}$
and the function $h_{\mathbf{A}}(\cdot)$ simplify to
\begin{equation}
\gamma_{L}=\mathbb{E}[\log|\mathbf{A}|]\;\;\;\text{and}\;\;\;h_{\mathbf{A}}(s)=\mathbb{E}|\mathbf{A}|^{s}\label{eq:univ_lya_h}
\end{equation}
as \eqref{eq:subm-1} now becomes an equality as the norm $\Vert\cdot\Vert$
is reduced to the absolute value of a real number. 
Thanks to this simplification, we get the following refinement of Theorem~\ref{Thm_main_cont}.
\begin{thm}
\label{Thm_main_cont_uni}Suppose that $d=1$ and that $\alpha\ge2$.
Then, for all $\mathbf{x}_{0}\in\mathbb{R}$, we have that 
\begin{equation}
0<\liminf_{R\rightarrow\infty}\frac{\mathbb{E}[\tau_{R}(\mathbf{x}_{0})]}{R^{\alpha}}\le\limsup_{R\rightarrow\infty}\frac{\mathbb{E}[\tau_{R}(\mathbf{x}_{0})]}{R^{\alpha}}<\infty\;.\label{eq:uni_bd}
\end{equation}
\end{thm}

\begin{rem}
We note that condition $\alpha\ge2$ is equivalent to $\mathbb{E}\left[\mathbf{A}^{2}\right]\le1$.
\end{rem}

The proof of Theorem \ref{Thm_main_cont} is given in Section \ref{sec55}.
We note that, the proof therein shows the lower bound 
\[
0<\liminf_{R\rightarrow\infty}\frac{\mathbb{E}[\tau_{R}(\mathbf{x}_{0})]}{R^{\alpha}}
\]
for all $\alpha>0$, and the condition $\alpha\ge2$ is required only
in the proof of the last inequality of \eqref{eq:uni_bd}. We believe
that for $\alpha\ge2$, the estimate of the form \eqref{eq:cont_conj}
holds, but for $\alpha<2$, we are not able to exclude the possibility
of $\mathbb{E}[\tau_{R}(\mathbf{x}_{0})]\simeq\kappa(R)R^{\alpha}$
where $\kappa(R)$ is a sub-polynomial prefactor satisfying $\kappa(R)\rightarrow+\infty$
as $R\rightarrow\infty$ slower than any polynomial. 

\subsection{Main results for the explosive regime}

Now we consider the explosive regime, i.e., the case when the Lyapunov
exponent of $\gamma_{L}$ is positive, namely,
\begin{equation}
\lim_{n\rightarrow\infty}\frac{1}{n}\log\Vert\Pi_{n}\Vert=\gamma_{L}>0\label{eq:exp_lya}
\end{equation}
almost surely. To estimate the mean of the exit time under this condition,
we impose the following assumptions. 
\begin{assumption}[Assumptions for explosive regime]$\,$
\label{Ass_Exp}
\begin{enumerate}
\item $\mathbb{P}(\mathbf{\mathbf{A}}\text{ is singular})=0$ and moreover
$\mathbf{A}$ is irreducible, i.e. there is no proper nontrivial subspace $V$
of $\mathbb{R}^{d}$ such that $\mathbf{A}V\subset V$ almost surely. 
\item We have that 
\begin{equation}
\inf_{\mathbf{x}\in\mathbb{R}^{d}}\mathbb{E}\log\left|\mathbf{A}\mathbf{x}+\mathbf{B}-\mathbf{x}\right|>-\infty\;.\label{eq:ass_expl}
\end{equation}
\item There exist $R_{0},\,z_{0}>1$, $C_{0}>0$, and $\beta_{0}>1$ such
that, we have 
\begin{equation}
\mathbb{P}\left[|\mathbf{A}\mathbf{x}+\mathbf{B}|>zR\right]\le\frac{C_{0}}{(\log z)^{\beta_{0}}}\,\mathbb{P}\left[|\mathbf{A}\mathbf{x}+\mathbf{B}|>R\right]\label{eq:nc2-1-1}
\end{equation}
for all $R\in[R_{0},\,\infty),\,z\in[z_{0},\,\infty)$ and $\mathbf{x}\in\mathcal{B}_{R}$.
{[}This is a much weaker version of Assumption \ref{Ass_Cont4}
of the contractive regime{]}
\end{enumerate}
\end{assumption}

Part (1) of the previous assumption is standard in the study of Kesten's stochastic recurrence equation \eqref{eq:AMP}, as most models trivially satisfy it. We will explain
the meaning of part (2) and (3) of Assumption \ref{Ass_Exp} after
stating the main theorem. In the explosive regime, the norm $\Vert\Pi_{n}\Vert=\Vert\mathbf{A}_{n}\cdots\mathbf{A}_{1}\Vert$
exponentially diverges to $+\infty$, we can guess that the exit time
is logarithmically small, and the next result confirm that it is indeed
the case. Of course, we assume \eqref{eq:exp_lya} and Assumption
\ref{Ass_Exp} in the next theorem. 
\begin{thm}
\label{Thm_main_expl}For all $\mathbf{x}_{0}\in\mathbb{R}^{d}$,
we have that 
\[
\frac{1}{\gamma_{L}}\le\liminf_{R\rightarrow\infty}\frac{\mathbb{E}[\tau_{R}(\mathbf{x}_{0})]}{\log R}\le\limsup_{R\rightarrow\infty}\frac{\mathbb{E}[\tau_{R}(\mathbf{x}_{0})]}{\log R}<\infty\;.
\]
 
\end{thm}

The proof of this theorem is given in Section \ref{sec6}. In particular,
we prove in Proposition \ref{prop:exp_lb} a strong lower bound of
the form 
\[
\liminf_{R\rightarrow\infty}\frac{\tau_{R}(\mathbf{x}_{0})}{\log R}\ge\frac{1}{\gamma_{L}}
\]
which holds almost surely. Then, the lower bound given in Theorem
\ref{Thm_main_expl} is a direct consequence of Fatou's lemma. The
proof of upper bound is quite technical and again based on highly
complicated construction of a sub-martingale which leads us an estimate
of the form 
\[
\mathbb{E}\left[\tau_{R}(\mathbf{x}_{0})\right]\le\kappa_{1}(1+\log R)-\kappa_{2}\log^{+}|\mathbf{x}_{0}|
\]
for some constants $\kappa_{1},\,\kappa_{2}>0$. We further note that
we can even compute these two constants explicitly. 
\begin{rem}[Remarks on Assumption \ref{Ass_Exp} and Theorem \ref{Thm_main_expl}]$\,$
\label{rem:Ass_Exp}
\begin{enumerate}
\item If part (2) of Assumption \ref{Ass_Exp} is violated, the escape from
the domain $\mathcal{B}_{R}$ may take much longer time scale, esepcially
when the dynamical system starts from $\mathbf{x}_{0}$ such that
\begin{equation}
\mathbb{E}\log\left|\mathbf{A}\mathbf{x}_{0}+\mathbf{B}-\mathbf{x}_{0}\right|=-\infty\;.\label{eq:condens}
\end{equation}
In particular, \eqref{eq:condens} implies that $\mathbf{X}_{1}(\mathbf{x}_{0})$
is very close to $\mathbf{x}_{0}$, and indicates that it might be
very hard to escape from a neighborhood of $\mathbf{x}_{0}$. 
\item In view of Theorem \ref{Thm_main_expl}, we confirmed that the escaping
time from a domain $\mathcal{B}_{R}$ is of $O(\log R)$. However,
if the random vector $\mathbf{A}\mathbf{x}+\mathbf{B}$ for some $\mathbf{x}$
is super-heavy tail of the form $\mathbb{P}\left[|\mathbf{Ax}+\mathbf{B}|\ge z\right]\sim\frac{1}{(\log z)^{\gamma}}$
for some $\gamma<1$, then the escape from the domain is not governed
by the structure of the process \eqref{eq:AMP} but by a single large
outcome and hence the analysis is trivialized significantly. For this
reason we assume part (3) of Assumption \ref{Ass_Exp} to investigate
the structural property of the process \eqref{eq:AMP}. 
\item We conjecture that it holds that 
\[
\lim_{R\rightarrow\infty}\frac{\mathbb{E}[\tau_{R}(\mathbf{x}_{0})]}{\log R}=\frac{1}{\gamma_{L}}
\]
as the multiplicative factor $\mathbf{A}_{n}$ satisfies, in the sense
of \eqref{eq:exp_lya}, $\Vert\mathbf{A}_{n}\cdots\mathbf{A}_{1}\Vert\sim e^{\gamma_{L}n}$
and hence we can guess that $|\mathbf{X}_{n}|$ grows as $e^{\gamma_{L}n}$.
Our lower bound is in coincidence with this conjecture, while getting
a tight upper bound leading this result seems to be very difficult
at this moment and probably requires additional assumption on the
behavior of $(\mathbf{A},\,\mathbf{B})$ for the reason indicated
in part (1) of the current remark. 
\item If $\mathbb{P}(\mathbf{\mathbf{A}}\text{ is singular})=0$ is satisfied,
Assumption \ref{Ass_Cont5} implies (1) of Assumption \ref{Ass_Exp}.
Indeed, suppose that there exists such $V\subset\mathbb{R}^{d}$.
Taking any nonzero vectors $v\in V$ and $v^{\perp}\in V^{\perp}$,
it holds that
\[
\mathbb{P}(v^{\perp}\cdot\Pi_{n_{0}}v=0)=1\ ,
\]
violating Assumption \ref{Ass_Cont5}. Also, as mentioned in the assumption,
(3) of Assumption \ref{Ass_Exp} is a direct consequence of Assumption
\ref{Ass_Cont4}. Hence, if Assumptions \ref{Ass_Cont4}
and \ref{Ass_Cont5} are guaranteed, only (2) of Assumption \ref{Ass_Exp}
is left to be shown to fulfill the whole Assumption for explosive
regime.
\end{enumerate}
\end{rem}

\section{\label{sec3}Applications }

Before delving into the main proofs, we will first introduce some
significant examples that emerge in various fields such as machine
learning or economics that fall into the framework explained in the
previous section. 

\subsection{\label{sec31}Stochastic gradient descent on a quadratic loss function}

\subsubsection*{Vanilla Mini-batch stochastic gradient descent}

The first example appears in the field of machine learning, as observed
in \cite{gurbuzbalaban2021heavy}. To be concrete, the stochastic gradient descent (SGD)
method with mini-batch runs on a quadratic loss function has a form
of \eqref{eq:AMP}. The quadratic loss function $F:\mathbb{R}^{d}\rightarrow\mathbb{R}$
for this model is defined by 

\[
F(\mathbf{x})=\frac{1}{2}\mathbb{E}_{(\mathbf{a},\,b)}\left[(\mathbf{a}\cdot\mathbf{x}-b)^{2}\right]
\]
where the data $(\mathbf{a},\,b)\in\mathbb{R}^{d}\times\mathbb{R}$
is a random element with unknown distribution $\mu$. We assume at
this moment that we can sample the random data $(\mathbf{a},\,b)$
according to $\mu$ although we do not know exactly what the distribution
is. Then ,the mini-batch SGD finding approximate minimum of $F$ with
batch size $m$ is conducted as follows. 
\begin{enumerate}
\item We set $\mathbf{X}_{0}=\mathbf{x}_{0}\in\mathbb{R}^{d}$ any initial
point.
\item At each $n$th step, $n\in\mathbb{Z}^{+}$, sample $m$ independent
data $(\mathbf{a}_{1}^{(n)},\,b_{1}^{(n)}),\;\dots,\,(\mathbf{a}_{m}^{(n)},\,b_{m}^{(n)})$
(according to the law $\mu$) and then define 
\[
F_{n}(\mathbf{x})=\frac{1}{m}\sum_{i=1}^{m}\frac{1}{2}(\mathbf{a}_{i}^{(n)}\cdot\mathbf{x}-b_{i}^{(n)})^{2}\;.
\]
\item Perform the gradient descent with learning rate $\eta>0$ to the approximating
function $F_{n}$, i.e., 
\begin{equation}
\mathbf{X}_{n+1}=\mathbf{X}_{n}-\eta\nabla F_{n+1}(\mathbf{X}_{n})\;.\label{eq:sgd}
\end{equation}
\end{enumerate}
Then, a simple computation shows that the \eqref{eq:sgd} can be re-written
as 
\[
\mathbf{X}_{n+1}=\mathbf{A}_{n+1}\mathbf{X}_{n}+\mathbf{B}_{n+1}\;\;\;;\;n\in\mathbb{Z}_{0}^{+}
\]
where
\begin{equation}
\mathbf{A}_{n}=\mathbf{I}_{d}-\frac{\eta}{m}\sum_{i=1}^{m}\mathbf{a}_{i}^{(n)}(\mathbf{a}_{i}^{(n)})^{\dagger}\;\;\;\;\text{and}\;\;\;\;\mathbf{B}_{n}=\frac{\eta}{m}\sum_{i=1}^{m}b_{i}^{(n)}\mathbf{a}_{i}^{(n)}\;.\label{eq:sgd_AB}
\end{equation}
Here, $\mathbf{I}_{d}$ denotes $d\times d$ identity matrix and $\mathbf{v}^{\dagger}$
denotes the transpose of $\mathbf{v}$ and therefore $\mathbf{a}_{i}^{(n)}(\mathbf{a}_{i}^{(n)})^{\dagger}$
denotes a $d\times d$ matrix. 

By assuming that the data are normal as in \cite{gurbuzbalaban2021heavy} , i.e.,
\begin{equation}
\mathbf{\text{\ensuremath{\mathbf{a}}\ensuremath{\sim}}}\mathcal{N}(\mathbf{0},\,\Sigma)\;\;\;\;\text{and}\;\;\;\;b\sim\mathcal{N}(0,\,\sigma_{B}^{2})\label{eq:sgd_ass}
\end{equation}
for some $\sigma_{B}>0$, we can verify that all the assumptions of
the previous section holds. We refer to Appendix B for a brief explanation. 

An interesting remark is that, by tuning learning rate $\eta$ from
a large number to small one as in the real application of this algorithm,
the lyapunov exponent decrease from positive number to the negative
number. The algorithm will be stabilized and find the minimum when
we take $\eta$ sufficiently small, (but not that small so that the
algorithm is way too slow) so that the Lyapunov exponenent associated
with the matrix $\mathbf{A}_{n}$ becomes negative. 

\subsubsection*{Mini-batch stochastic gradient descent with momentum}
We can sometimes accelerate the mini-batch SGD algorithim by exploiting
the momentum. This algorithm is a variant of the vanilla SGD explained in \eqref{eq:sgd}. 
All the settings of mini-batch SGD with momentum are identical to the one with vanilla one, but we update not only $\mathbf{X}_{n}$ but also momentum $\mathbf{V}_{n}$ together. 
For a parameter $\gamma\in[0,\,1)$ tuning the strength of the momentum, we update $(\mathbf{X}_{n+1},\,\mathbf{V}_{n+1})$ by 
\[
\begin{cases}
\mathbf{X}_{n+1}=\mathbf{X}_{n}-\eta\mathbf{V}_{n+1}\;,\\
\mathbf{V}_{n+1}=\gamma\mathbf{V}_{n}+(1-\gamma)\nabla F_{n+1}(\mathbf{X}_{n})\;.
\end{cases}
\]
Note that the case $\gamma=0$ corresponds to the vanilla mini-batch
SGD. Then, writing 
\[
\mathbf{Y}_{n}:=\begin{pmatrix}\mathbf{X}_{n}\\
\mathbf{V}_{n}
\end{pmatrix}\;,\;\;\;\;\mathbf{C}_{n}=\begin{pmatrix}\mathbf{I}_{d}-\eta(1-\gamma)\mathbf{A}_{n} & \ -\eta\gamma\mathbf{I}_{d}\\
(1-\gamma)\mathbf{A}_{n} & \gamma\mathbf{I}_{d}
\end{pmatrix},\;\;\;\ \mathbf{D}_{n}=\begin{pmatrix}-\eta(1-\gamma)\mathbf{B}_{n}\\
(1-\gamma)\mathbf{B}_{n}
\end{pmatrix}\;,
\]
where $\mathbf{A}_{n}$ and $\mathbf{B}_{n}$ are the ones defined
in \eqref{eq:sgd_AB}, we get 
\[
\mathbf{Y}_{n+1}=\mathbf{C}_{n+1}\mathbf{Y}_{n}+\mathbf{D}_{n+1}\;\;\;;\;n\in\mathbb{Z}_{0}^{+}
\]
and therefore $(\mathbf{Y}_{n})_{n\in\mathbb{Z}_{0}^{+}}$ can be
regarded as a $2d$-dimensional process \eqref{eq:AMP}. Again, by
assuming \eqref{eq:sgd_ass}, we can check that the process $(\mathbf{Y}_{n})_{n\in\mathbb{Z}_{0}^{+}}$
also satisfies the assumptions of Section \ref{sec2}. We will explain this in a companion paper of the current article. 

\subsection{\label{sec32}Time series models }

The remaining examples refer to financial models. As mentioned in
\cite[Chapter 1]{buraczewski2016stochastic}, the ARCH$(p)$ and GARCH$(1,q)$ are well
known examples of stochastic recurrence equation in econometrics or financial engineering. 

\subsubsection*{Autoregressive conditional heteroskedasticity (ARCH) model }

ARCH($p$) is a time series modeling log-returns of assets. More precisely,
in this model, the speculative prices of certain asset at time $t$
is denoted by $P_{t}>0$ and define the log-return of $P_{t}$ as
\[
X_{t}=\log\frac{P_{t+1}}{P_{t}}\;\;\;\;;\;t\in\mathbb{Z}^{+}\;.
\]
Then, for each $p\in\mathbb{Z}^{+}$, ARCH$(p)$ models the behavior
of $X_{t}$ as 
\begin{equation}
X_{t}=\sigma_{t}Z_{t}\;\;\;\;;\;t\ge p\label{eq:arch}
\end{equation}
where $(Z_{t})_{t\in\mathbb{Z}^{+}}$ is an i.i.d. sequence of random
variables with mean zero and variance one, and where 
\begin{equation}
\sigma_{t}^{2}=\alpha_{0}+\sum_{i=1}^{p}\alpha_{t}X_{t-i}^{2}\label{eq:arch2}
\end{equation}
for some non-negative constants $\alpha_{0},\,\alpha_{1},\,\dots,\,\alpha_{p}\ge0$
such that $\alpha_{0}\alpha_{p}>0$. We can re-interpret this equation
as 
\[
\begin{pmatrix}X_{t+1}^{2}\\
X_{t}^{2}\\
X_{t-1}^{2}\\
\vdots\\
X_{t-p+2}^{2}
\end{pmatrix}=\begin{pmatrix}\alpha_{1}Z_{t+1}^{2} & \alpha_{2}Z_{t+1}^{2} & \cdots & \alpha_{p-1}Z_{t+1}^{2} & \alpha_{p}Z_{t+1}^{2}\\
1 & 0 & \cdots & 0 & 0\\
0 & 1 & \cdots & 0 & 0\\
\vdots & \vdots & \cdots & \vdots & \vdots\\
0 & 0 & \cdots & 1 & 0
\end{pmatrix}\begin{pmatrix}X_{t}^{2}\\
X_{t-1}^{2}\\
X_{t-2}^{2}\\
\vdots\\
X_{t-p+1}^{2}
\end{pmatrix}+\begin{pmatrix}\alpha_{0}Z_{t+1}^{2}\\
0\\
0\\
\vdots\\
0
\end{pmatrix}
\]
and therefore, by letting 
\begin{equation}
\text{\ensuremath{\mathbf{X}_{t}=\begin{pmatrix}X_{t}^{2}\\
X_{t-1}^{2}\\
X_{t-2}^{2}\\
\vdots\\
X_{t-p+1}^{2}
\end{pmatrix}\;,\;\;\;\;}}\mathbf{A}_{t}=\begin{pmatrix}\alpha_{1}Z_{t}^{2} & \alpha_{2}Z_{t}^{2} & \cdots & \alpha_{p-1}Z_{t}^{2} & \alpha_{p}Z_{t}^{2}\\
1 & 0 & \cdots & 0 & 0\\
0 & 1 & \cdots & 0 & 0\\
\vdots & \vdots & \cdots & \vdots & \vdots\\
0 & 0 & \cdots & 1 & 0
\end{pmatrix}\;,\;\;\;\ \mathbf{B}_{t}=\begin{pmatrix}\alpha_{0}Z_{t}^{2}\\
0\\
0\\
\vdots\\
0
\end{pmatrix}\;,\label{eq:app_ex3}
\end{equation}
we can write the ARCH$(p)$ model exactly the same as \eqref{eq:AMP}.

Note that one might be interested in the exit time of $X_{t}$ itself,
instead of $\mathbf{X}_{t}$. This is not an issue since, if we denote
by $\tau_{R}$ and $\widehat{\tau}_{R}$ the exit time of the process
$\mathbf{X}_{t}$ and $X_{t}^2$, respectively (disregarding the initial
location at this moment), we have that, for all $R>0$
\[
\tau_{R}\le\widehat{\tau}_{R}\le\tau_{\sqrt{p}R}+p\ ,
\]
and therefore we can deliver the estimate of $\tau_{R}$ directly
to that of $\widehat{\tau}_{R}$. 

If Z is assumed to be normally distributed and $\alpha_1 \neq 0$, we proved that this model satisfies all the assumptions for our result. The proof is presented in the appendix A.

\subsubsection*{Generalized autoregressive conditional heteroskedasticity (GARCH)
model }

For $q\in\mathbb{Z}^{+}$, GARCH$(1,\,q)$ is a variation of ARCH
and defined by a sequence $(X_{t})_{t\in\mathbb{Z}^{+}}$ of random
variables as in \eqref{eq:arch} for $t\ge q$ with the same sequence
$(Z_{t})_{t\in\mathbb{Z}^{+}}$ of random variables and with a new
relation for $\sigma_{t}$:
\[
\sigma_{t}^{2}=\alpha_{0}+\alpha_{1}X_{t-1}^{2}+\sum_{j=1}^{q}\beta_{j}\sigma_{t-j}^{2}\;,
\]
for some non-negative constants $\alpha_{0},\,\alpha_{1},\,\beta_{1},\,\dots,\,\beta_{q}\ge0$
such that $\alpha_{0}\alpha_{1}\beta_{q}>0$. This equation can be
rewritten as:
\[
\begin{pmatrix}\sigma_{t+1}^{2}\\
\sigma_{t}^{2}\\
\sigma_{t-1}^{2}\\
\vdots\\
\sigma_{t-q+2}^{2}
\end{pmatrix}=\begin{pmatrix}\alpha_{1}Z_{t+1}^{2}+\beta_{1} & \beta_{2} & \cdots & \beta_{q-1} & \beta_{q}\\
1 & 0 & \cdots & 0 & 0\\
0 & 1 & \cdots & 0 & 0\\
\vdots & \vdots & \cdots & \vdots & \vdots\\
0 & 0 & \cdots & 1 & 0
\end{pmatrix}\begin{pmatrix}\sigma_{t}^{2}\\
\sigma_{t-1}^{2}\\
\sigma_{t-2}^{2}\\
\vdots\\
\sigma_{t-q+1}^{2}
\end{pmatrix}+\begin{pmatrix}\alpha_{0}\\
0\\
0\\
\vdots\\
0
\end{pmatrix}\;,
\]
and therefore, by letting 
\begin{equation}
\mathbf{X}_{t}=\begin{pmatrix}\sigma_{t}^{2}\\
\sigma_{t-1}^{2}\\
\sigma_{t-2}^{2}\\
\vdots\\
\sigma_{t-q+1}^{2}
\end{pmatrix}\;,\;\;\;\;\mathbf{A}_{t}=\begin{pmatrix}\alpha_{1}Z_{t}^{2}+\beta_{1} & \beta_{2} & \cdots & \beta_{q-1} & \beta_{q}\\
1 & 0 & \cdots & 0 & 0\\
0 & 1 & \cdots & 0 & 0\\
0 & 0 & \cdots & 0 & 0\\
0 & 0 & \cdots & 1 & 0
\end{pmatrix}\;,\;\;\;\;\mathbf{B}_{t}=\begin{pmatrix}\alpha_{0}\\
0\\
0\\
\vdots\\
0
\end{pmatrix}\;,\label{eq:app_ex4}
\end{equation}
we again get the expression of \eqref{eq:AMP} for the GARCH
model. 

Similar to the ARCH model, we proved that the assumptions holds true if Z is normally distributed, see appendix.

\section{\label{sec4}Preliminary Results for contractive Regime}

We study in this and the next section the contractive regime. In particular,
we shall always assume that $\gamma_{L}<0$ and that Assumptions \ref{Ass_Cont1},
\ref{Ass_Cont2}, \ref{Ass_Cont3}, \ref{Ass_Cont4}, and \ref{Ass_Cont5}
throughout Sections \ref{sec4} and \ref{sec5}. 

The proof of Theorems \ref{Thm_main_cont} and \ref{Thm_main_cont_uni}
are given in the next section. In this section, we provide several
preliminary results required in the proof given in the next section. 

\subsection{Coupling} 
Recall that $(\Omega,\,\mathcal{E},\,\mathbb{P})$ denotes the probability
space containing $\{(\mathbf{A}_{n},\,\mathbf{B}_{n}):n\in\mathbb{Z}^{+}\}$.
Since the process $(\mathbf{X}_{t}(\mathbf{x}))_{t\ge0}$ is completely
determined by outcomes of $(\mathbf{A}_{n},\,\mathbf{B}_{n})_{n\in\mathbb{Z}^{+}}$
and its starting location $\mathbf{X}_{0}$, we can couple all the
processes $(\mathbf{X}_{n}(\mathbf{y}))_{n\in\mathbb{Z}_{0}^{+}},\,\mathbf{y}\in\mathbb{R}^{d}$,
at the same probability space. Note that, under this coupling, all
the processes $(\mathbf{X}_{n}(\mathbf{y}))_{n\in\mathbb{Z}_{0}^{+}},\,\mathbf{y}\in\mathbb{R}^{d}$
share the outcomes $(\mathbf{A}_{n},\,\mathbf{B}_{n})_{n\in\mathbb{Z}^{+}}$.
From this moment on, when we consider several processes
start at different starting point, we shall assume that they are always
coupled throughout this manner. 

Note that, we can deduce by direct computation from \eqref{eq:AMP}
that 
\begin{equation}
\mathbf{X}_{n}(\mathbf{x})=\Pi_{n}\mathbf{x}+\mathbf{A}_{2}\cdots\mathbf{A}_{n}\mathbf{B}_{1}+\cdots+\mathbf{A}_{n}\mathbf{B}_{n-1}+\mathbf{B}_{n}\;.\label{eq:AMP_rep}
\end{equation}
Thus, under the coupling explained above, for all $n\in\mathbb{Z}^{+}$
and $\mathbf{y},\,\mathbf{z}\in\mathbb{R}^{d}$, we have 
\begin{equation}
\mathbf{X}_{n}(\mathbf{y})-\mathbf{X}_{n}(\mathbf{z})=\Pi_{n}(\mathbf{y}-\mathbf{z})\;.\label{eq:couple}
\end{equation}

\begin{rem}
\label{rem:ass}We shall assume in this section Assumptions 
\ref{Ass_Cont1}, \ref{Ass_Cont2},
and \ref{Ass_Cont3} so that by Theorem \ref{Thm_kes0} there exists
$\alpha>0$ satisfying $h_{\mathbf{A}}(\alpha)=1$ and there exists
a stationary measure $\nu_{\infty}$ of \eqref{eq:AMP}. Also, for the contractive case, it is known that
the infinite series
\[
\sum_{n=1}^{\infty}\mathbf{A}_{1}\cdots\mathbf{A}_{n-1}\mathbf{B}_{n}
\]
converges to a random variable $\mathbf{R}$ with distribution $\nu_{\infty}$
almost surely. see \cite[page 235]{kesten1973random}. This representation will be
used to prove Lemma \ref{lem:pre1}.
\end{rem}

\subsection{Finiteness of exit time}

In the contractive regime, since the process $\mathbf{X}_{t}$ tends
to $\mathbf{0}$ due to contracting force, and therefore it is not
even clear whether the exit time $\tau_{R}$ for large $R$ is finite
or not. In this section, we prove that for any $R>0$, the exit time
$\tau_{R}$ is finite almost surely even in the contractive regime. 
\begin{rem}
\label{Rem: lim_Pi}Before starting, we will summarize handy facts
regarding the evolution of the moment of $\Vert\Pi_{n}\Vert$ which
will be used frequently in the argument given in Sections \ref{sec4}
and \ref{sec5}. Namely, by Remark \ref{rem:h_A}-(3), we have that
\begin{equation}
\begin{cases}
\lim_{n\rightarrow\infty}\mathbb{E}\Vert\Pi_{n}\Vert^{\gamma}=0, & \gamma\in(0,\,\alpha)\;,\\
\lim_{n\rightarrow\infty}\mathbb{E}\Vert\Pi_{n}\Vert^{\gamma}=\infty. & \gamma\in(\alpha,\,\infty)\;.
\end{cases}\label{eq:mom_sm}
\end{equation}
Moreover, for all $\gamma\in(\alpha,\,\infty)$, we can find $n_{\gamma}\in\mathbb{Z}_{0}^{+}$
and $\delta_{\gamma}>0$ such that 
\begin{equation}
\mathbb{E}\Vert\Pi_{n}\Vert^{\gamma}\ge\left(1+\delta_{\gamma}\right)^{n}\ \;\;\;\text{for all }n>n_{0}\;.\label{eq:mom_sm2}
\end{equation}
\end{rem}

We first nvestigate the tail of the stationary measure $\nu_{\infty}$
(cf. Remark \ref{rem:ass}). Note that we did not assume the requirements
of Kesten's theorem, and therefore the next lemma cannot be obtained
from it. 
\begin{lem}
\label{lem:pre1}For all $\beta\in(0,\,\alpha)$, we have that
\begin{equation}
\int_{\mathbb{R}^{d}}|\mathbf{x}|^{\beta}\nu_{\infty}(d\mathbf{x})<\infty\;.\label{eq:lp2}
\end{equation}
\end{lem}

\begin{proof}
Let us fix $\beta\in(0,\,\alpha)$. By Remark \ref{Rem: lim_Pi},
we can take $n_{1}\in\mathbb{Z}^{+}$ such that 
\begin{equation}
\mathbb{E}\Vert\Pi_{n_{1}}\Vert^{\beta}<1\;.\label{eq:condn_1}
\end{equation}
For $n\in\mathbb{Z}^{+}$, define 
\[
\mathbf{R}_{n}=\sum_{k=0}^{n}\Pi_{kn_{1}}\mathbf{W}_{k}
\]
where, for $k\in\mathbb{Z}_{0}^{+},$ 
\[
\mathbf{W}_{k}:=\mathbf{B}_{kn_{1}+1}+\mathbf{A}_{kn_{1}+1}\mathbf{B}_{kn_{1}+2}+\cdots+\mathbf{A}_{kn_{1}+1}\cdots\mathbf{A}_{(k+1)n_{1}-1}\mathbf{B}_{(k+1)n_{1}}
\]
so that by Remark \ref{rem:ass}, the series $(\mathbf{R}_{n})_{n\in\mathbb{Z}^{+}}$
converges almost surely to a random variable $\mathbf{R}$ with distribution
$\nu_{\infty}$. 

By the elementary inequality 
\begin{equation}
(x_{1}+\cdots+x_{n})^{r}\le n^{r}(x_{1}^{r}+\cdots+x_{n}^{r})\label{elem_ine}
\end{equation}
which holds for all $x_{1},\,\dots,\,x_{n},\,r\ge0$, the submultiplicativeness
of the matrix norm, and independence of sequence $(\mathbf{A}_{n},\,\mathbf{B}_{n})_{n\in\mathbb{Z}^{+}}$,
we have 
\begin{equation}
\mathbb{E}|\mathbf{W}_{k}|^{\beta}\le n_{1}^{\beta}\sum_{i=0}^{n_{1}-1}\left(\mathbb{E}\Vert\mathbf{A}\Vert^{\beta}\right)^{i}\,\mathbb{E}|\mathbf{B}|^{\beta}<\infty\;,\label{eq:momcon}
\end{equation}
where $\mathbb{E}\Vert\mathbf{A}\Vert^{\beta}<\infty$ follows from
Assumption \ref{Ass_Cont2}. For $d$-dimensional random vector $\mathbf{X}$,
we define the norm $\Vert\cdot\Vert_{\beta}$ by
\[
\Vert\mathbf{X}\Vert_{\beta}:=\left(\mathbb{E}|\mathbf{X}|^{\beta}\right)^{\max\left\{ 1,\,1/\beta\right\} }\;.
\]
Then, by the triangle inquality and the independence of random variables,
we obtain
\[
\Vert\mathbf{R}_{n}\Vert_{\beta}\le\sum_{k=0}^{n}\left\Vert \Pi_{kn_{1}}\mathbf{W}_{k}\right\Vert _{\beta}\le\sum_{k=0}^{n}\left\Vert \Pi_{n_{1}}\right\Vert _{\beta}^{k}\left\Vert \mathbf{W}_{1}\right\Vert _{\beta}\le\frac{\left\Vert \mathbf{W}_{1}\right\Vert _{\beta}}{1-\left\Vert \Pi_{n_{1}}\right\Vert _{\beta}}\;.
\]
where the last inequality follows from $\left\Vert \Pi_{n_{1}}\right\Vert _{\beta}<1$
which comes from \eqref{eq:condn_1}. Hence, by Fatou's inequality,
we get 
\[
\Vert\mathbf{R}\Vert_{\beta}\le\liminf_{n\rightarrow\infty}\Vert\mathbf{R}_{n}\Vert_{\beta}\le\frac{\left\Vert \mathbf{W}_{1}\right\Vert _{\beta}}{1-\left\Vert \Pi_{n_{1}}\right\Vert _{\beta}}<\infty
\]
which leads to \eqref{eq:lp2}.
\end{proof}
The next lemma regarding the ergodic behavior of the process \eqref{eq:AMP} is a consequence of the previous one.
\begin{lem}
\label{lem:pre1-2}The followings hold. 
\begin{enumerate}
\item For all $\psi\in C_{c}^{\infty}(\mathbb{R}^{d})$ and a Borel probability
measure $\pi$, we have 
\[
\lim_{n\rightarrow\infty}\mathbb{E}\psi(\mathbf{X}_{n}(\pi))=\int_{\mathbb{R}^{d}}\psi d\nu_{\infty}\;.
\]
\item For all $\mathbf{x}\in\mathbb{R}^{d}$ and $n\in\mathbb{Z}^{+}$,
we have $\mathbb{P}(\mathbf{X}_{n}(\mathbf{x})=\mathbf{x})<1$. 
\end{enumerate}
\end{lem}

\begin{proof}
(1) Let us fix $\psi\in C_{c}^{\infty}(\mathbb{R}^{d})$. By writing
$\alpha_{0}=\min\left\{ \frac{\alpha}{2},\,1\right\} $ where $\alpha$
as defined in \ref{Thm_kes0}, we deduce from \eqref{eq:mom_sm} that
\begin{equation}
\lim_{n\rightarrow\infty}\mathbb{E}\Vert\Pi_{n}\Vert^{\alpha_{0}}=0\;.\label{eq:pre1}
\end{equation}
Since $\psi$ is compactly supported, it is straightforward that $\psi$
is H{\"o}lder continuous with exponent $\alpha_{0}$, i.e., there exists
$M_{\psi}>0$ such that 
\[
|\psi(\mathbf{x})-\psi(\mathbf{y})|\le M_{\psi}|\mathbf{x}-\mathbf{y}|^{\alpha_{0}}\;\;\;\text{for all }\mathbf{x},\,\mathbf{y}\in\mathbb{R}^{d}\;,
\]
and therefore, for any $\mathbf{y},\,\mathbf{z}\in\mathbb{R}^{d}$,
by \eqref{eq:couple}, we have 
\[
\left|\mathbb{E}\left[\psi(\mathbf{X}_{n}(\mathbf{y}))-\psi(\mathbf{X}_{n}(\mathbf{z}))\right]\right|\le M_{\psi}\mathbb{E}\left[\left|\mathbf{X}_{n}(\mathbf{y})-\mathbf{X}_{n}(\mathbf{\mathbf{z}})\right|^{\alpha_{0}}\right]\le M_{\psi}|\mathbf{y}-\mathbf{z}|^{\alpha_{0}}\cdot\mathbb{E}\Vert\Pi_{n}\Vert^{\alpha_{0}}\;.
\]
Hence, for each $\mathbf{x}\in\mathbb{R}^{d}$,
\begin{align*}
\left|\mathbb{E}\psi(\mathbf{X}_{n}(\mathbf{x}))-\int_{\mathbb{R}^{d}}\psi\,d\mathbf{\nu_{\infty}}\right| & =\left|\int_{\mathbb{R}^{d}}\mathbb{E}\left[\psi(\mathbf{X}_{n}(\mathbf{x}))-\psi(\mathbf{X}_{n}(\mathbf{y}))\right]\nu_{\infty}(d\mathbf{y})\right|\\
 & \le M_{\psi}\,\mathbb{E}\Vert\Pi_{n}\Vert^{\alpha_{0}}\,\int_{\mathbb{R}^{d}}|\mathbf{x}-\mathbf{y}|^{\alpha_{0}}\nu_{\infty}(d\mathbf{y})
\end{align*}
where the first identity uses the stationarity of $\nu_{\infty}$.
Since $\alpha_{0}<\alpha$, by Lemma \ref{lem:pre1} and \eqref{eq:pre1},
we get 
\[
\lim_{n\rightarrow\infty}\mathbb{E}\psi(\mathbf{X}_{n}(\mathbf{x}))=\int_{\mathbb{R}^{d}}\psi d\nu_{\infty}\;\;\;\;\text{for all }\mathbf{x}\in\mathbb{R}^{d}\;.
\]
Thus, the conclusion of the lemma follows from the dominated convergence
theorem as $\psi$ is bounded. 

\medskip{}

\noindent (2) Suppose that there exists $n_{0}\in\mathbb{Z}^{+}$
and $\mathbf{x}_{0}\in\mathbb{R}^{d}$ such that $\mathbb{P}(\mathbf{X}_{n_{0}}(\mathbf{x}_{0})=\mathbf{x}_{0})=1$
so that we have 
\begin{equation}
\mathbb{P}(\mathbf{X}_{kn_{0}}(\mathbf{x}_{0})=\mathbf{x}_{0})=1\;\;\;\text{for all }k\in\mathbb{Z}^{+}\;.\label{eq:ap3}
\end{equation}
Hence, for $\psi\in C_{c}^{\infty}(\mathbb{R}^{d})$, we have 
\[
\lim_{k\rightarrow\infty}\mathbb{E}\left[\psi(\mathbf{X}_{kn_{0}}(\mathbf{x}_{0}))\right]=\int_{\mathbb{R}^{d}}\psi\,d\nu_{\infty}\;.
\]
Thus, by \eqref{eq:ap3}, we have $\int_{\mathbb{R}^{d}}\psi\,d\nu_{\infty}=\psi(\mathbf{x}_{0})$.
Namely, $\nu_{\infty}$ is the Dirac measure at $\mathbf{x}_{0}$.
This contradicts to Assumption \ref{Ass_Cont1}-(3). 
\end{proof}
Now we are ready to prove the finiteness of exit time. 
\begin{prop}
\label{Prop_cont_exit time finite}For all $\mathbf{x}\in\mathbb{R}^{d}$
and $R>0$, we have $\tau_{R}(\mathbf{x})<\infty$ almost surely. 
\end{prop}

\begin{proof}
We fix $\mathbf{x}\in\mathbb{R}^{d}$ and $R>0$, and suppose on the
contrary that 
\begin{equation}
\mathbb{P}\left[\tau_{R}(\mathbf{x})=\infty\right]=p_{0}>0\;.\label{eq:p_0}
\end{equation}
It is immediate from \eqref{eq:p_0} that $\mathbf{x}\in\mathcal{B}_{R}$. 

\smallskip{}

\noindent {[}Step 1{]} We first recursively define some probabilities
$(q_{n})_{n\in\mathbb{Z}^{+}}$ and distributions $(\mu_{n})_{n\in\mathbb{Z}^{+}}$.
First, we set
\begin{equation*}
\begin{aligned}
\newnota{q_1}{q_{1}} & =\defnota{q_1}{\mathbb{P}\left[\mathbf{A}\mathbf{x}+\mathbf{B}\in\mathcal{B}_{R}\right]}\\
\mu_{1}(\cdot) & =\mathbb{P}\left[\mathbf{A}\mathbf{x}+\mathbf{B}\in\cdot\mid\mathbf{A}\mathbf{x}+\mathbf{B}\in\mathcal{B}_{R}\right]\;.
\end{aligned}
\end{equation*}
Note that $\mu_{1}$ is supported on $\mathcal{B}_{R}$. Then, we
recursively define, for $n\in\mathbb{Z}^{+}$, 
\begin{align*}
q_{n+1} & =\mathbb{P}\left[\mathbf{A}\mathbf{Z}_{n}+\mathbf{B}\in\mathcal{B}_{R}\right]\;,\\
\mu_{n+1}(\cdot) & =\mathbb{P}\left[\mathbf{A}\mathbf{Z}_{n}+\mathbf{B}\in\cdot\mid\mathbf{A}\mathbf{Z}_{n}+\mathbf{B}\in\mathcal{B}_{R}\right]\;,
\end{align*}
where $\mathbf{Z}_{n}\in\mathbb{R}^{d}$ is a random vector distributed
according to $\mu_{n}$ and independent of $\mathbf{A}$ and $\mathbf{B}$.
Then, by the Markov property, we can write 
\begin{align}
 & \mathbb{P}\left[\mathbf{X}_{1}(\mathbf{x}),\,\dots,\,\mathbf{X}_{n}(\mathbf{x})\in\mathcal{B}_{R}\right]=q_{1}q_{2}\cdots q_{n}\;\;\text{and}\label{eq:rec1}\\
 & \mathbb{P}\left[\mathbf{X}_{1}(\mu_{k}),\,\dots,\,\mathbf{X}_{n}(\mu_{k})\in\mathcal{B}_{R}\right]=q_{k+1}q_{k+2}\cdots q_{k+n}\label{eq:rec2}
\end{align}
for all $n,\,k\in\mathbb{Z}^{+}$. In particular, by \eqref{eq:rec1},
we have, for all $n\in\mathbb{Z}^{+}$, 
\[
p_{0}:=\mathbb{P}\left[\tau_{R}(\mathbf{x})=\infty\right]\le\mathbb{P}\left[\mathbf{X}_{1}(\mathbf{x}),\,\dots,\,\mathbf{X}_{n}(\mathbf{x})\in\mathcal{B}_{R}\right]=q_{1}q_{2}\cdots q_{n}\;.
\]
Since $p_{0}>0$, we must have 
\begin{equation}
\lim_{n\rightarrow\infty}q_{n}=1\;.\label{eq:conv_qn}
\end{equation}

\smallskip{}

\noindent {[}Step 2{]} Compactness argument along with Skorokhod representation
theorem

\noindent Since each measure $\mu_{n}$ is supported on the compact
set $\mathcal{B}_{R}$, by the Prokhorov theorem we can find a subsequence
$\{\mu_{n_{k}}\}_{k\in\mathbb{Z}^{+}}$of $\{\mu_{n}\}_{n\in\mathbb{Z}^{+}}$
that converges weakly to a probability measure $\mu_{\infty}$ supported
on $\mathcal{B}_{R}$. 

Let us fix $n\in\mathbb{Z}^{+}$. Note from \eqref{eq:rec2} that
\[
\mathbb{P}\left[\mathbf{X}_{n}(\mu_{n_{k}})\in\mathcal{B}_{R}\right]\ge q_{n_{k}+1}q_{n_{k}+2}\cdots q_{n_{k}+n}
\]
and therefore by \eqref{eq:conv_qn}, we have 
\begin{equation}
\lim_{k\rightarrow\infty}\mathbb{P}\left[\mathbf{X}_{n}(\mu_{n_{k}})\in\mathcal{B}_{R}\right]=1\;.\label{eq:limPk}
\end{equation}
We next claim that 
\begin{equation}
\mathbb{P}\left[\mathbf{X}_{n}(\mu_{\infty})\in\mathcal{B}_{R}\right]=1\;.\label{eq:lim1}
\end{equation}
To that end, it suffices to prove that 
\begin{equation}
\mathbb{P}\left[\mathbf{X}_{n}(\mu_{\infty})\in\mathcal{B}_{R+\epsilon}\right]=1\text{ for all }\epsilon>0\;.\label{eq:lim2}
\end{equation}
Fix $\epsilon>0$. By the Skorokhod representation thoerem, we can
find a sequence of random variables $\{\mathbf{y}_{k}\}_{k\in\mathbb{Z}^{+}}$
and $\mathbf{y}_{\infty}$ such that $\mathbf{y}_{k}\sim\mu_{n_{k}}$
for $k\in\mathbb{Z}^{+}$, $\mathbf{y}_{\infty}\sim\mu_{\infty}$
and $\mathbf{y}_{k}\rightarrow\mathbf{y}_{\infty}$ almost surely.
Then, by \eqref{eq:couple}, we have 

\begin{align*}
\mathbb{P}\left[\mathbf{X}_{n}(\mu_{n_{k}})\in\mathcal{B}_{R}\right]-\mathbb{P}\left[\mathbf{X}_{n}(\mu_{\infty})\in\mathcal{B}_{R+\epsilon}\right] & =\mathbb{P}\left[\mathbf{X}_{n}(\mathbf{y}_{k})\in\mathcal{B}_{R}\right]-\mathbb{P}\left[\mathbf{X}_{n}(\mathbf{y}_{\infty})\in\mathcal{B}_{R+\epsilon}\right]\\
 & \le\mathbb{P}\left[\left|\mathbf{X}_{n}(\mathbf{y}_{k})-\mathbf{X}_{n}(\mathbf{y}_{\infty})\right|>\epsilon\right]\\
 & =\mathbb{P}\left[\left|\Pi_{n}(\mathbf{y}_{k}-\mathbf{y}_{\infty})\right|>\epsilon\right]\\
 & \le\mathbb{P}\left[\left|\mathbf{y}_{k}-\mathbf{y_{\infty}}\right|\ge\frac{\epsilon}{M}\right]+\mathbb{P}\left[\Vert\Pi_{n}\Vert\ge M\right]\;.
\end{align*}
Since $\mathbf{y}_{k}$ converges to $\mathbf{y}_{\infty}$ in probability,
by sending $k\rightarrow\infty$ and then $M\rightarrow\infty$, and
recalling \eqref{eq:limPk}, we get 
\[
1-\mathbb{P}_{\mu_{\infty}}\left[\mathbf{X}_{n}\in\mathcal{B}_{R+\epsilon}\right]\le0\;.
\]
This implies \eqref{eq:lim2}, and hence \eqref{eq:lim1}. 

\smallskip{}

\noindent {[}Step 3{]} Contradiction to the non-boundedness of support
of $\nu_{\infty}$

\noindent Let us take an auxiliary function $\psi:\mathbb{R}^{d}\rightarrow[0,\,1]$
which is a smooth function on $\mathbb{R}^{d}$ satisfying
\[
\psi|_{\mathcal{B}_{R}}\equiv0\;,\;\;\psi|_{\mathcal{B}_{R+1}^{c}}\equiv1\;,\text{ and}\;\;|D\psi|_{\infty}\le2\;.
\]
Then, by Lemma \ref{lem:pre1-2}, we get 
\[
\lim_{n\rightarrow\infty}\mathbb{E}\psi(\mathbf{X}_{n}(\mu_{\infty}))=\int_{\mathbb{R}^{d}}\psi\,d\nu_{\infty}\;.
\]
However, by \eqref{eq:lim1} and the fact that $\psi|_{\mathcal{B}_{R}}\equiv0$,
we can notice that the left-hand side of the previous display is $0$.
This implies that $\int_{\mathbb{R}^{d}}\psi\,d\nu_{\infty}=0$, and
in particular $\nu_{\infty}(\mathcal{B}_{R+1}^{c})=0$. This contradicts
to the non-boundedness of the support of $\nu_{\infty}$ from the
Assumption \ref{Ass_Cont3}, and the proof is completed. 
\end{proof}

\subsection{Scattering property of $\Pi_{n}$}

We denote by 
\[
\mathbb{S}^{d-1}=\{\mathbf{x}\in\mathbb{R}^{d}:|\mathbf{x}|=1\}
\]
the $d$-dimensional unit sphere and denote by $\sigma_{d}$ the uniform
surface measure on $\mathbb{S}^{d-1}$ normalized in a way that $\sigma_{d}(\mathbb{S}^{d-1})=1$.
We denote by 
\[
\mathcal{D}_{r}(\mathbf{z}):=\left\{ \mathbf{y}\in\mathbb{S}^{d-1}:|\mathbf{y}-\mathbf{z}|<r\right\} 
\]
the disk of radius $r>0$ (in fact, the intersection of ball of radius
$r$ and unit sphere) on the surface $\mathbb{S}^{d-1}$ centered
at $\mathbf{z}$. 

The main result of this subsection is the following
proposition which is a consequence of Assumption \ref{Ass_Cont5}
and asserts that, the matrix $\Pi_{n}$ is uniformly acts on a fixed
vector $\mathbf{x}\in\mathbb{R}^{d}$ regardless of its direction
in terms of $\beta$-th moment. 
\begin{prop}
\label{prop:scat}Suppose that Assumption \ref{Ass_Cont5}
holds. Then, for all $\beta\in(0,\,\alpha_{\infty})$, there exist $\delta_{0}>0$ and $N_{0}>0$ such
that, for all $n\ge N_{0}$ and $\mathbf{x}\in\mathbb{R}^{d}$, we
have
\[
\mathbb{E}|\Pi_{n}\mathbf{x}|^{\beta}\ge\delta_{0}\mathbb{E}\Vert\Pi_{n}\Vert^{\beta}\cdot|\mathbf{x}|^{\beta}\;.
\]
\end{prop}

The proof of this proposition requires several preliminary technical
lemmata below. We start with an elemetary observation about finite
Borel measures on $\mathbb{S}^{d-1}$.
\begin{lem}
\label{lem_dmu}Let $\mu$ be a finite Borel measure on $\mathbb{S}^{d-1}$
and let $A\subset\mathbb{S}^{d-1}$ be a Borel set. Suppose that 
\[
\liminf_{\delta\rightarrow0}\frac{\mu(\mathcal{D}_{\delta}(\mathbf{z}))}{\sigma_{d}(\mathcal{D}_{\delta}(\mathbf{z}))}=0\;\;\;\text{for all }\mathbf{z}\in A\;.
\]
Then, we have $\mu(A)=0$. 
\end{lem}

\begin{proof}
By the Lebesgue decomposition theorem, we can decompose $\mu=\mu_{a}+\mu_{s}$
such that $\mu_{a}\ll\sigma_{d}$ and $\mu_{s}\perp\sigma_{d}$ where
both $\mu_{a}$ and $\mu_{s}$ are positive measures. Then, by \cite[Theorem 7.14 at page 143]{rudin1987real},
for almost all (with respect to $\sigma_{d}$ or equivalently $\mu_{a}$)
$\mathbf{z}\in A$, we have 
\[
\frac{d\mu_{a}}{d\sigma_{d}}(\mathbf{z})=\lim_{\delta\rightarrow0}\frac{\mu_{a}(\mathcal{D}_{\delta}(\mathbf{z}))}{\sigma_{d}(\mathcal{D}_{\delta}(\mathbf{z}))}\le\liminf_{\delta\rightarrow0}\frac{\mu(\mathcal{D}_{\delta}(\mathbf{z}))}{\sigma_{d}(\mathcal{D}_{\delta}(\mathbf{z}))}=0
\]
and thus we get $\mu_{a}(A)=0$.

On the other hand, by \cite[Theorem 7.15 at page 143]{rudin1987real}, we know
that 
\[
\lim_{\delta\rightarrow0}\frac{\mu_{s}(\mathcal{D}_{\delta}(\mathbf{z}))}{\sigma_{d}(\mathcal{D}_{\delta}(\mathbf{z}))}=\infty,\quad\text{a.e.}\ \mu_{s}\;.
\]
Since the limit at the left-hand side is $0$ for all $\mathbf{z}\in A$,
we can conclude that $\mu_{s}(A)=0$ as well. This completes the proof. 
\end{proof}
The next technical lemma is a key in the proof of Proposition \ref{prop:scat}. 
\begin{lem}
\label{lem_dmu2}Suppose that Assumption \ref{Ass_Cont5} holds. Let $\phi:\mathbb{R}^{d}\rightarrow[0,\infty)$ be a non-negative measurable
function satisfies $\phi>0$ on $\mathbb{R}^{d}\setminus\left\{ \mathbf{0}\right\} $.
Fix $n\in\mathbb{Z}^{+}$, $\mathbf{x}\in\mathbb{R}^{d}\setminus\left\{ \mathbf{0}\right\} $
and suppose that $\mathbb{E}[\phi(\Pi_{n}\mathbf{x})]<\infty$. Define
a subset $S=S(\mathbf{x})$ of $\mathbb{S}^{d-1}$ by the collection
of all $\mathbf{z}\in\mathbb{S}^{d-1}$ such that 
\[
\liminf_{\delta\rightarrow0}\frac{1}{\sigma_{d}(\mathcal{D}_{\delta}(\mathbf{z}))}\mathbb{E}\left[\phi(\Pi_{n}\mathbf{x})\cdot\mathbf{1}\left\{ \Pi_{n}\mathbf{x}\neq\mathbf{0},\,\frac{\Pi_{n}\mathbf{x}}{|\Pi_{n}\mathbf{x}|}\in\mathcal{D}_{\delta}(\mathbf{z})\right\} \right]>0\;.
\]
Then, the set $S$ contains a normal basis of $\mathbb{R}^{d}$.
\end{lem}

\begin{proof}
Define a positive Borel measure $\mu$ on $\mathbb{S}^{d-1}$ by
\[
\mu(D):=\mathbb{E}\left[\phi(\Pi_{n}\mathbf{x})\cdot\mathbf{1}\left\{ \Pi_{n}\mathbf{x}\neq\mathbf{0},\,\frac{\Pi_{n}\mathbf{x}}{|\Pi_{n}\mathbf{x}|}\in D\right\} \right],
\]
for any Borel measurable set $D$ of $\mathbb{S}^{d-1}$. Then, by
Lemma \ref{lem_dmu}, we have $\mu(\mathbb{S}^{d-1}\setminus S)=0$
and thus by condition on $\phi,$ we have 
\[
\mathbf{1}\left\{ \Pi_{n}\mathbf{x}\neq\mathbf{0},\,\frac{\Pi_{n}\mathbf{x}}{|\Pi_{n}\mathbf{x}|}\in\mathbb{S}^{d-1}\setminus S\right\} =0
\]
almost surely with respect to $\mathbb{P}$. 

Now suppose on the contrary that $S$ does not contain a basis of
$\mathbb{R}^{d}$ so that there is a non-zero vector $\mathbf{w}\in\mathbb{R}^{d}$
perpendicular to all the vectors in $S$. Then, we have 
\[
\mathbb{P}(\mathbf{w}\cdot\Pi_{n}\mathbf{x}\neq0)\le\mathbb{P}\left(\Pi_{n}\mathbf{x}\neq\mathbf{0},\,\frac{\Pi_{n}\mathbf{x}}{|\Pi_{n}\mathbf{x}|}\in\mathbb{S}^{d-1}\setminus S\right)=0
\]
and thus we get a contradiction to Assumption \ref{Ass_Cont5}.
\end{proof}
The next lemma asserts that, for each matrix $\mathbf{M}\in\mathbb{R}^{d\times d}$
and basis $\{\mathbf{z}_{1},\,\dots,\,\mathbf{z}_{d}\}\in\mathbb{S}^{d-1}$,
the norm of vectors $\mathbf{M}\mathbf{z}_{i}$ can be controlled
from below. 
\begin{lem}
\label{lem_Mz}Let $\mathbf{M}\in\mathbb{R}^{d\times d}$, let $\{\mathbf{z}_{1},\,\dots,\,\mathbf{z}_{d}\}\in\mathbb{S}^{d-1}$
be a basis of $\mathbb{R}^{d}$, and let $\beta>0$. Denote by $\mathbf{Z}\in\mathbb{R}^{d\times d}$
the matrix whose $k$th column is $\mathbf{z}_{k}$ for $k=1,\,\dots,\,d$.
Then, we have 
\begin{equation}
\sum_{i=1}^{d}|\mathbf{M}\mathbf{z}_{i}|^{\beta}\ge\frac{\Vert\mathbf{M}\Vert^{\beta}}{d^{\beta}\Vert\mathbf{Z}^{-1}\Vert^{\beta}}\;.\label{eq:cond_k_0}
\end{equation}
\end{lem}

\begin{proof}
Let $\mathbf{x}_{0}\in\mathbb{S}^{d-1}$ be the unit vector satisfy
$|\mathbf{M}\mathbf{x}_{0}|=\Vert\mathbf{M}\Vert$. Write $\mathbf{x}_{0}=a_{1}\mathbf{z}_{1}+\cdots+a_{d}\mathbf{z}_{d}=\mathbf{Z}\mathbf{a}$
where $\mathbf{a}=(a_{1},\,\dots,\,a_{d})^{\dagger}$ so that 
\[
|\mathbf{a}|=|\mathbf{Z}^{-1}\mathbf{x}_{0}|\le\Vert\mathbf{Z}^{-1}\Vert
\]
since $|\mathbf{x}_{0}|=1$. Thus, by \eqref{elem_ine}, we have 
\[
\Vert\mathbf{M}\Vert^{\beta}=|\mathbf{M}\mathbf{x}_{0}|^{\beta}\le d^{\beta}\sum_{i=1}^{d}|a_{i}|^{\beta}|\mathbf{M}\mathbf{z}_{i}|^{\beta}\le d^{\beta}\Vert\mathbf{Z}^{-1}\Vert^{\beta}\sum_{i=1}^{d}|\mathbf{M}\mathbf{z}_{i}|^{\beta}\;.
\]
\end{proof}
Now we are ready to prove Proposition \ref{prop:scat}. 
\begin{proof}[Proof of Proposition \ref{prop:scat}]
 Through scaling, it suffices to show that
\[
\mathbb{E}|\Pi_{n}\mathbf{x}|^{\beta}\ge\delta_{0}\mathbb{E}\Vert\Pi_{n}\Vert^{\beta}
\]
holds for all $\mathbf{x}\in\mathbb{S}^{d-1}$. Suppose that this
statement of proposition does not hold so that we can find a sequence
$N_{1}<N_{2}<\cdots$ of positive integers and vectors $\mathbf{x}_{1},\mathbf{x}_{2},\cdots\in\mathbb{S}^{d-1}$
such that 
\begin{equation}
\mathbb{E}|\Pi_{N_{k}}\mathbf{x}_{k}|^{\beta}<\frac{1}{k}\mathbb{E}\Vert\Pi_{N_{k}}\Vert^{\beta}.\label{eq:cond11}
\end{equation}
We can assume here that $N_{1}>n_{0}$. Moreover, since $\mathbb{S}^{d-1}$
is compact, by extracting a suitable subseuqence, we can assume in
addition that there exists $\mathbf{x}_{\infty}\in\mathbb{S}^{d-1}$
such that 
\begin{equation}
\lim_{k\rightarrow\infty}|\mathbf{x}_{\infty}-\mathbf{x}_{k}|=0\;.\label{eq:cond12}
\end{equation}
Then, by \eqref{elem_ine} and \eqref{eq:cond11}, 
\begin{align*}
\mathbb{E}|\Pi_{N_{k}}\mathbf{x}_{\infty}|^{\beta} & \le2^{\beta}(\mathbb{E}|\Pi_{N_{k}}\mathbf{x}_{k}|^{\beta}+\mathbb{E}|\Pi_{N_{k}}(\mathbf{x}_{\infty}-\mathbf{x}_{k})|^{\beta})\\
 & \le2^{\beta}\,\mathbb{E}\Vert\Pi_{N_{k}}\Vert^{\beta}\,\left(\frac{1}{k}+|\mathbf{x}_{\infty}-\mathbf{x}_{k}|^{\beta}\right)\;.
\end{align*}
Thus, write $\rho_{k}=2^{\beta}(\frac{1}{k}+|\mathbf{x}_{\infty}-\mathbf{x}_{k}|^{\beta})$
so that by \eqref{eq:cond12}, we have 
\begin{equation}
\lim_{k\rightarrow\infty}\rho_{k}=0\;\;\;\;\text{and\;\;\;\;}\mathbb{E}|\Pi_{N_{k}}\mathbf{x}_{\infty}|^{\beta}\le\rho_{k}\mathbb{E}\Vert\Pi_{N_{k}}\Vert^{\beta}\;.\label{eq:e10}
\end{equation}

Now applying Lemma \ref{lem_dmu2} with $\phi(\mathbf{x})=|\mathbf{x}|^{\beta}$
and $\mathbf{x}=\mathbf{x}_{\infty}$, which is possible since $\beta\in(0,\,\alpha_{\infty})$, (cf. Assumption \ref{Ass_Cont2})
\[
\mathbb{E}[\phi(\Pi_{n_{0}}\mathbf{x}_{\infty})]=\mathbb{E}\left[|\Pi_{n_{0}}\mathbf{x}_{\infty}|^{\beta}\right]\le\mathbb{E}\Vert\Pi_{n_{0}}\Vert^{\beta}<\infty\;,
\]
we can find a normal basis $\left\{ \mathbf{z}_{1},\dots,\mathbf{z}_{d}\right\} \in S(\mathbf{x}_{\infty})$
{[}see the definition in Lemma \ref{lem_dmu2} for $S(\mathbf{x}_{\infty})${]}
of $\mathbb{R}^{d}$. We next claim that, for all $\ell\in\{1,\,\dots,\,d\}$,
\begin{equation}
\lim_{k\rightarrow\infty}\max\frac{\mathbb{E}|\Pi_{N_{k}-n_{0}}\mathbf{z}_{\ell}|^{\beta}}{\mathbb{E}\Vert\Pi_{N_{k}-n_{0}}\Vert^{\beta}}=0\;.\label{eq:step2}
\end{equation}
To that end, for $\ell\in\{1,\,\dots,\,d\}$ and $\delta>0$, define
\[
\mathcal{T}_{\ell,\,\delta}:=\left\{ \Pi_{n_{0}}\mathbf{x}_{\infty}\neq\mathbf{0},\,\frac{\Pi_{n_{0}}\mathbf{x}_{\infty}}{|\Pi_{n_{0}}\mathbf{x}_{\infty}|}\in\mathcal{D}_{\delta}(\mathbf{z}_{\ell})\right\} \;.
\]
Since $\mathbf{z}_{\ell}\in S(\mathbf{x}_{\infty})$, we can find
$\epsilon_{0},\,\delta_{0}>0$ such that, for all $\delta\in(0,\,\delta_{0})$,
\begin{equation}
\mathbb{E}\left[|\Pi_{n_{0}}\mathbf{x}_{\infty}|^{\beta}\cdot\mathbf{1}_{\mathcal{T}_{\ell,\,\delta}}\right]\ge\epsilon_{0}\cdot\sigma_{d}(\mathcal{D}_{\delta}(\mathbf{z}_{\ell}))\;.\label{eq:e11}
\end{equation}
We note from \eqref{elem_ine} that, for all $\mathbf{y}\in\mathcal{D}_{\delta}(\mathbf{z}_{\ell})$
and $n\in\mathbb{Z}^{+}$, 
\begin{align*}
|\Pi_{n}\mathbf{z}_{\ell}|^{\beta} & \le2^{\beta}(|\Pi_{n}\mathbf{y}|^{\beta}+|\Pi_{n}(\mathbf{z}_{\ell}-\mathbf{y})|^{\beta})\le2^{\beta}(|\Pi_{n}\mathbf{y}|^{\beta}+\delta^{\beta}\Vert\Pi_{n}\Vert^{\beta})\;.
\end{align*}
Since $\frac{\Pi_{n_{0}}\mathbf{x}_{\infty}}{|\Pi_{n_{0}}\mathbf{x}_{\infty}|}\in\mathcal{D}_{\delta}(\mathbf{z}_{\ell})$
under $\mathcal{T}_{\ell,\,\delta}$, applying the previous bound
with $\mathbf{y}=\frac{\Pi_{n_{0}}\mathbf{x}_{\infty}}{|\Pi_{n_{0}}\mathbf{x}_{\infty}|}$,
we get, for any $k\in\mathbb{Z}^{+}$, (recall that $N_{k}\ge N_{1}>n_{0}$)
\begin{align*}
 & \mathbb{E}\left[|\Pi_{N_{k}-n_{0}}\mathbf{z}_{\ell}|^{\beta}\right]\cdot\mathbb{E}\left[|\Pi_{n_{0}}\mathbf{x}_{\infty}|^{\beta}\cdot\mathbf{1}_{\mathcal{T}_{\ell,\,\delta}}\right]\\
 & =\mathbb{E}\left[|\mathbf{A}_{N_{k}}\cdots\mathbf{A}_{n_{0}+1}\mathbf{z}_{\ell}|^{\beta}\cdot|\Pi_{n_{0}}\mathbf{x}_{\infty}|^{\beta}\cdot\mathbf{1}_{\mathcal{T}_{\ell,\,\delta}}\right]\\
 & \le2^{\beta}\mathbb{E}\left[\left(\left|\mathbf{A}_{N_{k}}\cdots\mathbf{A}_{n_{0}+1}\frac{\Pi_{n_{0}}\mathbf{x}_{\infty}}{|\Pi_{n_{0}}\mathbf{x}_{\infty}|}\right|^{\beta}+\delta^{\beta}\Vert\mathbf{A}_{N_{k}}\cdots\mathbf{A}_{n_{0}+1}\Vert^{\beta}\right)\cdot|\Pi_{n_{0}}\mathbf{x}_{\infty}|^{\beta}\cdot\mathbf{1}_{\mathcal{T}_{\ell,\,\delta}}\right]\\
 & \le2^{\beta}\,\mathbb{E}|\Pi_{N_{k}}\mathbf{x}_{\infty}|^{\beta}+2^{\beta}\delta^{\beta}\,\mathbb{E}\Vert\Pi_{N_{k}-n_{0}}\Vert^{\beta}\cdot\mathbb{E}\left[|\Pi_{n_{0}}\mathbf{x}_{\infty}|^{\beta}\cdot\mathbf{1}_{\mathcal{T}_{\ell,\,\delta}}\right]\\
 & \le2^{\beta}\mathbb{E}\Vert\Pi_{N_{k}-n_{0}}\Vert^{\beta}\cdot\left(\rho_{k}\mathbb{E}\Vert\Pi_{n_{0}}\Vert^{\beta}+\delta^{\beta}\mathbb{E}\left[|\Pi_{n_{0}}\mathbf{x}_{\infty}|^{\beta}\cdot\mathbf{1}_{\mathcal{T}_{\ell,\,\delta}}\right]\right)
\end{align*}
where the last line follows from \eqref{eq:e10}
and submultiplicativity of matrix operator norm. Dividing both sides
by $\mathbb{E}\Vert\Pi_{N_{k}-n_{0}}\Vert^{\beta}\mathbb{E}\left[|\Pi_{n_{0}}\mathbf{x}_{\infty}|^{\beta}\cdot\mathbf{1}_{\mathcal{T}_{\ell,\,\delta}}\right]$
and applying \eqref{eq:e11}, we get, for all $\delta\in(0,\,\delta_{0})$,
\begin{equation}
\frac{\mathbb{E}|\Pi_{N_{k}-n_{0}}\mathbf{z}_{\ell}|^{\beta}}{\mathbb{E}|\Pi_{N_{k}-n_{0}}\mathbf{z}_{\ell}|^{\beta}}\le2^{\beta}\left(\frac{\rho_{k}}{\epsilon_{0}\cdot\sigma_{d}(\mathcal{D}_{\delta}(\mathbf{z}_{\ell}))}\mathbb{E}\Vert\Pi_{n_{0}}\Vert^{\beta}+\delta^{\beta}\right)\,\le C_{\beta}\left(\frac{\rho_{k}}{\delta^{d-1}}+\delta^{\beta}\right)\label{eq:bdd_Piz}
\end{equation}
for some constant $C_{\beta}>0$ since $\sigma_{d}(\mathcal{D}_{\delta}(\mathbf{z}_{\ell}))=O(\delta^{d-1})$.
Hence, by \eqref{eq:e10}, we get 
\[
\limsup_{k\rightarrow\infty}\frac{\mathbb{E}|\Pi_{N_{k}-n_{0}}\mathbf{z}_{\ell}|^{\beta}}{\mathbb{E}\Vert\Pi_{N_{k}-n_{0}}\Vert^{\beta}}\le C\delta^{\beta}
\]
and hency by letting $\delta\rightarrow0$ we get \eqref{eq:step2}. 

On the other hand, by Lemma \ref{lem_Mz}, we get 
\[
\sum_{\ell=1}^{d}\frac{\mathbb{E}|\Pi_{N_{k}-n_{0}}\mathbf{z}_{\ell}|^{\beta}}{\mathbb{E}\Vert\Pi_{N_{k}-n_{0}}\Vert^{\beta}}\ge\frac{1}{d^{\beta}\Vert\mathbf{Z}^{-1}\Vert^{\beta}}
\]
where $\mathbf{Z}\in\mathbb{R}^{d\times d}$ is the matrix whose $k$th
column is $\mathbf{z}_{k}$ for $k=1,\,\dots,\,d$. Letting $k\rightarrow\infty$
and recalling \eqref{eq:step2} yield a contradiction.
\end{proof}

\section{\label{sec5}Escape Time Analysis for contractive Regime}

The purpose of the current section is to prove Theorem \ref{Thm_main_cont}.
Hence, we assume that the Lyapunov exponent $\gamma_{L}$ is negative,
and moreover assume Assumptions \ref{Ass_Cont1}, \ref{Ass_Cont2},
\ref{Ass_Cont3}, \ref{Ass_Cont4} and \ref{Ass_Cont5} throughout
section. 

The proof of Theorem \ref{Thm_main_cont} is based on the following
propositions on the estimate of the mean exit time in the super- and
sub-critical regime. We recall the constnat $\alpha_{+}\in(\alpha,\,\infty)$
from Assumption \ref{Ass_Cont4}. We note that, we shall use $C_{\gamma},\,D_{\gamma},\,\dots$
to denote constants depending only on $\gamma$ (except for the distribution
of $(\mathbf{A},\,\mathbf{B})$), and different appearances of $C_{\gamma}$,
for instance, may denote different constants. 
\begin{prop}
\label{prop:cont_ubd}For each $\gamma\in(\alpha,\,\alpha_{+})$,
there exists a constant $C_{\gamma}>0$ such that, for all $\mathbf{x}_{0}\in\mathbb{R}^{d}$
and $R>0$, 
\begin{equation}
\mathbb{E}\left[\tau_{R}(\mathbf{x}_{0})\right]\le C_{\gamma}\left(R^{\gamma}+1\right)\;.\label{eq:pp1}
\end{equation}
\end{prop}

\begin{prop}
\label{prop:cont_lbd}For each $\gamma\in(0,\,\alpha)$, there exist
constants $C_{\gamma},\,D_{\gamma}>0$ such that, for all $\mathbf{x}_{0}\in\mathbb{R}^{d}$
and $R>0$,
\begin{equation}
\mathbb{E}\left[\tau_{R}(\mathbf{x}_{0})\right]\ge C_{\gamma}R^{\gamma}-D_{\gamma}\left(|\mathbf{x}_{0}|^{\gamma}+1\right)\;.\label{eq:pp1-1}
\end{equation}
\end{prop}

The proof of these propositions are given in Sections \ref{sec52}
and \ref{sec53}, respectively. We assume these propositions now and
then complete the proof of Theorem \ref{Thm_main_cont}. 
\begin{proof}[Proof of Theorem \ref{Thm_main_cont}]
\noindent Fix $\mathbf{x}_{0}\in\mathbb{R}^{d}$. Let $\gamma\in(\alpha,\,\alpha_{+})$.
Then, by Proposition \ref{prop:cont_ubd}, we have 
\[
\frac{\log\mathbb{E}\left[\tau_{R}(\mathbf{x}_{0})\right]}{\log R}\le\frac{\log C_{\gamma}}{\log R}+\frac{\log(R^{\gamma}+1)}{\log R}
\]
and therefore by letting $R\rightarrow\infty$ and then $\gamma\searrow\alpha$,
we get 
\begin{equation}
\limsup_{R\rightarrow\infty}\frac{\log\mathbb{E}\left[\tau_{R}(\mathbf{x}_{0})\right]}{\log R}\le\alpha\;.\label{eq:ubf}
\end{equation}
On the other hand, for $\gamma\in(0,\,\alpha)$, we take $R$ sufficiently
large so that, by Proposition \ref{prop:cont_ubd}, $\mathbb{E}\left[\tau_{R}(\mathbf{x}_{0})\right]\ge C_{\gamma}R^{\gamma}$.
Then, taking $\log$arithm both sides, dividing both sides by $\log R$,
and sending $R\rightarrow\infty$ and then $\gamma\nearrow\alpha$,
we get 
\begin{equation}
\liminf_{R\rightarrow\infty}\frac{\log\mathbb{E}\left[\tau_{R}(\mathbf{x}_{0})\right]}{\log R}\ge\alpha\;.\label{lbf}
\end{equation}
Combining \eqref{eq:ubf} and \eqref{lbf} completes the proof. 
\end{proof}

\subsection{\label{sec51}Control of escaping location}

We control, based on Assumption \ref{Ass_Cont4}, that we can not
exit the ball $\mathcal{B}_{R}$ far away from the ball. This is consequence
of the assumption that the tail of $\mathbf{A}\mathbf{x}+\mathbf{B}$
is not that heavy. 
\begin{notation*}
From now on, we simply write $\mathbf{X}_{\tau_{R}}(\mathbf{x})$
instead of $\mathbf{X}_{\tau_{R}(\mathbf{x})}(\mathbf{x})$ as there
is no risk of confusion. 
\end{notation*}
\begin{lem}
\label{lem:exit}For all $\gamma\in[0,\,\alpha_{+})$, there exists
a constant $C_{\gamma}>0$ such that 
\[
\mathbb{E}|\mathbf{X}_{\tau_{R}}(\mathbf{x})|^{\gamma}\le C_{\gamma}(R^{\gamma}+1)\;.
\]
\end{lem}

\begin{proof}
Let $\gamma\in[0,\,\alpha_{+})$, $R>R_{0}$ and $\mathbf{x}\in\mathcal{B}_{R}$.
By the layer-cake formula and the fact that 
\[
\lim_{z\rightarrow\infty}z^{\gamma}\mathbb{P}(|\mathbf{A}\mathbf{x}+\mathbf{B}|>z)=0
\]
which follows from \eqref{eq:nc2-1}, we can write
\begin{align*}
 & \mathbb{E}\left[|\mathbf{A}\mathbf{x}+\mathbf{B}|^{\gamma}\cdot\mathbf{1}\left\{ |\mathbf{A}\mathbf{x}+\mathbf{B}|>R\right\} \right]\\
 & \quad=R^{\gamma}\left[\mathbb{P}(|\mathbf{A}\mathbf{x}+\mathbf{B}|>R)+\int_{1}^{\infty}\mathbb{P}(|\mathbf{A}\mathbf{x}+\mathbf{B}|>Rz)\cdot\gamma z^{\gamma-1}dz\right]\;.
\end{align*}
Since $\gamma\in[0,\,\alpha_{+})$, again by \eqref{eq:nc2-1}, we
have 
\[
\int_{z_{0}}^{\infty}\mathbb{P}(|\mathbf{A}\mathbf{x}+\mathbf{B}|>Rz)\cdot\gamma z^{\gamma-1}dz\le\frac{\gamma C_{0}}{\alpha_{+}-\gamma}\mathbb{P}(|\mathbf{A}\mathbf{x}+\mathbf{B}|>R)
\]
plus
\[
\int_{1}^{z_{0}}\mathbb{P}(|\mathbf{A}\mathbf{x}+\mathbf{B}|>Rz)\cdot\gamma z^{\gamma-1}dz\le\mathbb{P}(|\mathbf{A}\mathbf{x}+\mathbf{B}|>R)\cdot\int_{1}^{z_{0}}\gamma z^{\gamma-1}dz=z_{0}^{\gamma}\mathbb{P}(|\mathbf{A}\mathbf{x}+\mathbf{B}|>R)\ ,
\]
therefore we can conclude that
\begin{equation}
\mathbb{E}\left[|\mathbf{A}\mathbf{x}+\mathbf{B}|^{\gamma}\cdot\mathbf{1}\left\{ |\mathbf{A}\mathbf{x}+\mathbf{B}|>R\right\} \right]\le\left(\frac{\gamma C_{0}}{\alpha_{+}-\gamma}+z_{0}^{\gamma}\right)R^{\gamma}\cdot\mathbb{P}(|\mathbf{A}\mathbf{x}+\mathbf{B}|>R)\;.\label{eq:exit-1}
\end{equation}
Consequently, by the strong Markov property, 
\[
\mathbb{E}|\mathbf{X}_{\tau_{R}}(\mathbf{x})|^{\gamma}\le\sup_{\mathbf{y}\in\mathcal{B}_{R}}\mathbb{E}\left[|\mathbf{A}\mathbf{y}+\mathbf{B}|^{\gamma}\,\big|\,|\mathbf{A}\mathbf{y}+\mathbf{B}|>R\right]\le\left(\frac{\gamma C_{0}}{\alpha_{+}-\gamma}+z_{0}^{\gamma}\right)R^{\gamma}\;.
\]
\end{proof}

\subsection{Moment estimates}
\begin{lem}
\label{lem:gamma1}Let $\gamma\in[0,\,\alpha_{\infty})$
and $n\in\mathbb{Z}^{+}$. Then, there exists a constant $C_{n,\,\gamma}>0$
such that 
\[
\mathbb{E}|\mathbf{X}_{n}(\mathbf{x})|^{\gamma}<C_{n,\,\gamma}(|\mathbf{x}|^{\gamma}+1)\;\;\;\text{for all }\mathbf{x}\in\mathbb{R}^{d}.
\]
In particular, $\mathbb{E}|\mathbf{X}_{n}(\mathbf{x})|^{\gamma}<\infty$. 
\end{lem}

\begin{proof}
\noindent Since $\gamma\in[0,\,\alpha_{\infty})$, we have that $\mathbb{E}\Vert\mathbf{A}\Vert^{\gamma}<\infty$
and $\mathbb{E}|\mathbf{B}|^{\gamma}<\infty$ by Assumption \ref{Ass_Cont2}
and \eqref{eq:nc2-2}, respectively. Thus, by \eqref{eq:AMP_rep}
and \eqref{elem_ine}, we have
\[
\mathbb{E}|\mathbf{X}_{n}(\mathbf{x})|^{\gamma}\le(n+1)^{\gamma}\left((\mathbb{E}\Vert\mathbf{A}\Vert^{\gamma})^{n}|\mathbf{x}|^{\gamma}+\sum_{i=0}^{n-1}(\mathbb{E}\Vert\mathbf{A}\Vert^{\gamma})^{n-1}\mathbb{E}|\mathbf{B}|^{\gamma}\right)
\]
and hence we get the desired result. 
\end{proof}
\begin{lem}
\label{lem:gamma2}Let $\gamma\in[0,\,\alpha_{\infty})$
and $n\in\mathbb{Z}^{+}$. It holds that 
\begin{equation}
\inf_{\mathbf{x}:|\mathbf{x}|=1}\mathbb{E}|\Pi_{n}\mathbf{x}-\mathbf{x}|^{\gamma}>0\;.\label{eq:estlb}
\end{equation}
\end{lem}

\begin{proof}
We define $F_{n,\,\gamma}:\mathbb{R}^{d}\rightarrow\mathbb{R}$ as
\[
F_{n,\,\gamma}(\mathbf{x})=\mathbb{E}|\Pi_{n}\mathbf{x}-\mathbf{x}|^{\gamma}\;.
\]
We will prove that $F_{n,\,\gamma}>0$ on $\mathbb{R}^{d}\setminus\{\mathbf{0}\}$
and that $F_{n,\,\gamma}$ is continuous on $\mathbb{R}^{d}$. Then,
\eqref{eq:estlb} follows immediately. 

Let us first show the continuity of $F_{n,\,\gamma}$. If $\gamma<1$,
by the elementary inequality 
\[
|a^{\gamma}-b^{\gamma}|\le|a-b|^{\gamma}\;\;\;\text{and\;\;\;}(a+b)^{\gamma}\le a^{\gamma}+b^{\gamma}\;\;\;\text{for }a,\,b\ge0\;,
\]
we have 
\begin{equation}
|F_{n,\,\gamma}(\mathbf{x})-F_{n,\,\gamma}(\mathbf{y})|\le\mathbb{E}|\Pi_{n}(\mathbf{x}-\mathbf{y})-(\mathbf{x}-\mathbf{y})|^{\gamma}\le(\mathbb{E}|\Pi_{n}|^{\gamma}+1)|\mathbf{x}-\mathbf{y}|^{\gamma}\label{eq:diff1-1}
\end{equation}
and the continuity follows immediately. On the other hand, if $\gamma\ge1$,
by the elementary inequality 
\[
|a^{\gamma}-b^{\gamma}|\le\gamma|a-b|(a+b)^{\gamma-1}\;\;\;\;;\;a,\,b\ge0\;,
\]
which holds because of the mean-value theorem, then by H{\"o}lder inequality
we have 
\begin{align*}
 & |F_{n,\,\gamma}(\mathbf{x})-F_{n,\,\gamma}(\mathbf{y})|\\
 & \le\gamma\mathbb{E}\left[|\Pi_{n}(\mathbf{x}-\mathbf{y})-(\mathbf{x}-\mathbf{y})|\cdot(||\Pi_{n}\mathbf{x}-\mathbf{x}|+|\Pi_{n}\mathbf{y}-\mathbf{y}|)^{\gamma-1}\right]\\
 & \le\gamma\mathbb{E}\left[\Pi_{n}(\mathbf{x}-\mathbf{y})-(\mathbf{x}-\mathbf{y})|^{\gamma}\right]^{\frac{1}{\gamma}}\mathbb{E}\left[(|\Pi_{n}\mathbf{x}-\mathbf{x}|+|\Pi_{n}\mathbf{y}-\mathbf{y}|)^{\gamma}\right]^{\frac{\gamma-1}{\gamma}}\;
\end{align*}
Hence, by the bound (cf. \eqref{elem_ine}), 
\begin{equation}
\mathbb{E}|\Pi_{n}\mathbf{x}-\mathbf{x}|^{\gamma}\le2^{\gamma}\mathbb{E}(\Vert\Pi_{n}\Vert^{\gamma}+1)|\mathbf{x}|^{\gamma}\label{eq:bdpx}
\end{equation}
we can conclude that 
\[
|F_{n,\,\gamma}(\mathbf{x})-F_{n,\,\gamma}(\mathbf{y})|\le C_{n,\,\gamma}|\mathbf{x}-\mathbf{y}|\left(|\mathbf{x}|^{\gamma-1}+|\mathbf{y}|^{\gamma-1}+1\right)
\]
and hence the continuity follows. 

Next we shall prove that $F_{n,\,\gamma}>0$. Suppose that $F_{n,\,\gamma}(\mathbf{x}_{0})=0$,
i.e., $\Pi_{n}\mathbf{x}_{0}=\mathbf{x}_{0}$ almost surely, for some
$\mathbf{x}_{0}\in\mathbb{R}^{d}\setminus\{\mathbf{0}\}.$ Then, inductively,
we have that $\Pi_{kn}\mathbf{x}_{0}=\mathbf{x}_{0}$ almost surely
for all $k\ge1$. Let us take $\beta\in(\alpha,\,\alpha_{+})$ so
that $h_{\mathbf{A}}(\beta)>1$ by Theorem \ref{Thm_kes0}-(4). Then,
by Proposition \ref{prop:scat}, we have 
\[
|\mathbf{x}_{0}|^{\beta}=\mathbb{E}|\Pi_{kn}\mathbf{x}_{0}|^{\beta}\ge\delta_{0}\mathbb{E}\Vert\Pi_{kn}\Vert^{\beta}|\mathbf{x}_{0}|\;,
\]
and therefore, as $\mathbf{x}_{0}\neq0$, we get $\delta_{0}\mathbb{E}\Vert\Pi_{kn}\Vert^{\beta}\le1$
for all $k\in\mathbb{Z}^{+}$. This implies that 
\[
h_{\mathbf{A}}(\beta)=\lim_{k\rightarrow\infty}\left[\mathbb{E}\Vert\Pi_{kn}\Vert^{\beta}\right]^{\frac{1}{kn}}\le\lim_{k\rightarrow\infty}\delta_{0}^{-\frac{1}{kn}}=1\;.
\]
This contradicts to $h_{\mathbf{A}}(\beta)>1$ (as $\beta\in(\alpha,\,\alpha_{\infty})$)
and we are done. 
\end{proof}
\begin{lem}
\label{lem:gamma3}Let $\gamma\in[0,\,\alpha_{\infty})$
and $n\in\mathbb{Z}^{+}$. It holds that
\[
\inf_{\mathbf{x}\in\mathbb{R}^{d}}\mathbb{E}|\mathbf{X}_{n}(\mathbf{x})-\mathbf{x}|^{\gamma}>0\ .
\]
\end{lem}

\begin{proof}
Define $F=F_{n,\,\gamma}:\mathbb{R}^{d}\rightarrow\mathbb{R}$ as
\[
F_{n,\,\gamma}(\mathbf{x})=\mathbb{E}|\mathbf{X}_{n}(\mathbf{x})-\mathbf{x}|^{\gamma}\;.
\]
Then, as in the proof of Lemma \ref{lem:gamma2}, we can verify that
$F_{n,\,\gamma}$ is continuous on $\mathbb{R}^{d}$. The only difference
in the proof is that, in the case $\gamma>1$, we need to substitute
the bound of $\mathbb{E}|\Pi_{n}\mathbf{x}-\mathbf{x}|^{\gamma}$
explained in \eqref{eq:bdpx} with a bound of $\mathbb{E}|\mathbf{X}_{n}(\mathbf{x})-\mathbf{x}|^{\gamma}$.
This can be done with Lemma \ref{lem:gamma1} which provides 
\[
\mathbb{E}|\mathbf{X}_{n}(\mathbf{x})-\mathbf{x}|^{\gamma}<C_{n,\,\gamma}(|\mathbf{x}|^{\gamma}+1)\;\;\;\text{for all }\mathbf{x}\in\mathbb{R}^{d}.
\]
Moreover, by Lemma \ref{lem:pre1-2}-(2), we have $F_{n,\,\gamma}>0$.
Finally, by \eqref{eq:couple} and \eqref{elem_ine}, we have 
\begin{equation}
F_{n,\,\gamma}(\mathbf{x})\ge\frac{1}{2^{\gamma}}\mathbb{E}|\Pi_{n}\mathbf{x}-\mathbf{x}|^{\gamma}-\mathbb{E}|\mathbf{X}_{n}(\mathbf{0})|^{\gamma}\ge c_{n,\,\gamma}|\mathbf{x}|^{\gamma}-c_{n,\,\gamma}'\;,\label{eq:lbF}
\end{equation}
where $c_{n,\,\gamma}'=\mathbb{E}|\mathbf{X}_{n}(\mathbf{0})|^{\gamma}$
and 
\[
c_{n,\,\gamma}=\frac{1}{2^{\gamma}}\inf_{\mathbf{y}:|\mathbf{y}|=1}\mathbb{E}|\Pi_{n}\mathbf{y}-\mathbf{y}|^{\gamma}>0
\]
where the last strict inequality is the content of Lemma \ref{lem:gamma2}.
Summing up, $F_{n,\,\gamma}$ is a positive continuous function, which
diverges to $+\infty$ as $\mathbf{|x}|\rightarrow\infty$ thanks
to \eqref{eq:lbF}, and hence the assertion of the lemma follows. 
\end{proof}
\begin{lem}
\label{lem:pre30}Let $\gamma\in(\alpha,\,\alpha_{\infty})$.
Then, there exists $N_{\gamma}>0$ such that, for all $n\ge N_{\gamma}$
and $\mathbf{x}\in\mathbb{R}^{d}$, 
\[
\mathbb{E}\left[|\mathbf{X}_{n}(\mathbf{x})|^{\gamma}\,\big|\,\mathbf{X}_{n-N_{\gamma}}(\mathbf{x})\right]\ge2|\mathbf{X}_{n-N_{\gamma}}(\mathbf{x})|^{\gamma}-\mathbb{E}|\mathbf{X}_{N_{\gamma}}(\mathbf{0})|^{\gamma}\;,
\]
and in particular, we have 
\[
\mathbb{E}|\mathbf{X}_{n}(\mathbf{x})|^{\gamma}\ge2\mathbb{E}|\mathbf{X}_{n-N_{\gamma}}(\mathbf{x})|^{\gamma}-\mathbb{E}|\mathbf{X}_{N_{\gamma}}(\mathbf{0})|^{\gamma}\;.
\]
\end{lem}

\begin{proof}
Fix $\gamma\in(\alpha,\,\alpha_{\infty})$. Since $h_{\mathbf{A}}(\gamma)>1$
by Theorem \ref{Thm_kes0}-(1), there exist constants $N_{1}=N_{1}(\gamma)>0$
and $\rho_{0}=\rho_{0}(\gamma)>0$ such that 
\begin{equation}
\mathbb{E}\Vert\Pi_{n}\Vert^{\gamma}\ge(1+\rho_{0})^{n}\;\;\;\text{for all}\;n>N_{1}\ .\label{eq:div}
\end{equation}
Recall the constants $N_{0}$ and $\delta_{0}$ obtained in Proposition
\ref{prop:scat}. Then, by Proposition \ref{prop:scat}, for all $\mathbf{x}\in\mathbb{R}^{d}$
and $n>\max(N_{0},\,N_{1})$, we have 
\begin{equation}
\mathbb{E}|\Pi_{n}\mathbf{x}|^{\gamma}\ge\delta_{0}\cdot(1+\rho_{0})^{n}|\mathbf{x}|^{\gamma}\;.\label{eq:pre31}
\end{equation}
Let us take $N_{\gamma}>\max(N_{0},\,N_{1})$ large enough so that
\begin{equation}
\frac{\delta_{0}}{2^{\gamma}}\cdot(1+\rho_{0})^{N_{\gamma}}\ge2\;.\label{eq:pre32}
\end{equation}
By the same computation with \eqref{eq:AMP_rep}, we can write
\begin{equation}
\mathbf{X}_{n}(\mathbf{x})=\mathbf{A}_{n}\mathbf{A}_{n-1}\cdots\mathbf{A}_{n-N_{\gamma}+1}\mathbf{X}_{n-N_{\gamma}}(\mathbf{x})+\mathbf{W}\label{eq_amp2}
\end{equation}
where 
\begin{equation}
\mathbf{W}=\mathbf{A}_{n}\cdots\mathbf{A}_{n-N_{\gamma}+2}\mathbf{B}_{n-N_{\gamma}+1}+\cdots+\mathbf{A}_{n}\mathbf{B}_{n-1}+\mathbf{B}_{n}\;.\label{eq_amp2-2}
\end{equation}
Note that $\mathbf{W}$ has the same distribution with $\mathbf{X}_{N_{\gamma}}(\mathbf{0})$
by \eqref{eq:AMP_rep}. By \eqref{elem_ine}, Proposition \ref{prop:scat},
\eqref{eq:pre31}, and \eqref{eq:pre32}, for all $n\ge N_{\gamma}$,
\begin{align*}
\mathbb{E}\left[|\mathbf{X}_{n}(\mathbf{x})|^{\gamma}\,\big|\,\mathbf{X}_{n-N_{\gamma}}(\mathbf{x})\right] & \ge\frac{1}{2^{\gamma}}\mathbb{E}\left[|\mathbf{A}_{n}\mathbf{A}_{n-1}\cdots\mathbf{A}_{n-N_{\gamma}+1}\mathbf{X}_{n-N_{\gamma}}(\mathbf{x})|^{\gamma}\,\big|\,\mathbf{X}_{n-N_{\gamma}}(\mathbf{x})\right]-\mathbb{E}|\mathbf{X}_{N_{\gamma}}(\mathbf{0})|^{\gamma}\\
 & \ge\frac{\delta_{0}}{2^{\gamma}}\cdot(1+\rho_{0})^{N_{\gamma}}\mathbb{E}\Vert\mathbf{A}_{n}\mathbf{A}_{n-1}\cdots\mathbf{A}_{n-N_{\gamma}+1}\Vert^{\gamma}\cdot|\mathbf{X}_{n-N_{\gamma}}(\mathbf{x})|^{\gamma}-\mathbb{E}|\mathbf{X}_{N_{\gamma}}(\mathbf{0})|^{\gamma}\\
 & \ge2|\mathbf{X}_{n-N_{\gamma}}(\mathbf{x})|^{\gamma}-\mathbb{E}|\mathbf{X}_{N_{\gamma}}(\mathbf{0})|^{\gamma}\;.
\end{align*}
This completes the proof of the first assertion. The second assertion
follows directly from the first one by the tower property. 
\end{proof}
From this moment on, for $\gamma\in(\alpha,\,\alpha_{\infty})$, the
constant $N_{\gamma}$ always refers to the one obtained in the previous
lemma. 
\begin{lem}
\label{lem:pre3-0}Let $\gamma\in(\alpha,\,\alpha_{\infty})$. There
exists $M_{\gamma}>0$ such that, for all $n\ge M_{\gamma}$ and $\mathbf{x}\in\mathbb{R}^{d}$,
\[
\mathbb{E}|\mathbf{X}_{n}(\mathbf{x})|^{\gamma}>\mathbb{E}|\mathbf{X}_{N_{\gamma}}(\mathbf{0})|^{\gamma}+1\;.
\]
\end{lem}

\begin{proof}
Denote by $\mathbf{X}_{N_{\gamma}}'(\mathbf{x})$ an independent copy
of $\mathbf{X}_{N_{\gamma}}(\mathbf{x})$ so that $\mathbf{X}_{n+N_{\gamma}}(\mathbf{x})$
has the same distribution with $\mathbf{X}_{n}(\mathbf{X}_{N_{\gamma}}'(\mathbf{x}))$.
Then, by \eqref{elem_ine}, \eqref{eq:couple},and Proposition \ref{prop:scat},
there exists $\delta_{0}(\gamma)>0$ and $N_{0}(\gamma)\in\mathbb{Z}^{+}$
such that, for all $n\ge N_{0}(\gamma)$ and $\mathbf{x}\in\mathbb{R}^{d}$,
\[
\mathbb{E}|\Pi_{n}\mathbf{x}|^{\gamma}\ge\delta_{0}\mathbb{E}\Vert\Pi_{n}\Vert^{\gamma}\cdot|\mathbf{x}|^{\gamma}
\]
holds. For all $n\in\mathbb{Z}^{+}$, by \eqref{elem_ine}, \eqref{eq:couple}
and Proposition \ref{prop:scat}, we have 
\begin{align*}
\mathbb{E}|\mathbf{X}_{n+N_{\gamma}}(\mathbf{x})|^{\gamma}+\mathbb{E}|\mathbf{X}_{n}(\mathbf{x})|^{\gamma} & =\mathbb{E}\left[|\mathbf{X}_{n}(\mathbf{X}_{N_{\gamma}}'(\mathbf{x}))|^{\gamma}+|\mathbf{X}_{n}(\mathbf{x})|^{\gamma}\right]\\
 & \ge\frac{1}{2^{\gamma}}\mathbb{E}|\mathbf{X}_{n}(\mathbf{X}_{N_{\gamma}}'(\mathbf{x}))-\mathbf{X}_{n}(\mathbf{x})|^{\gamma}\\
 & \ge\frac{\delta_{0}}{2^{\gamma}}\mathbb{E}\Vert\Pi_{n}\Vert^{\gamma}\cdot\mathbb{E}|\mathbf{X}_{N_{\gamma}}(\mathbf{x})-\mathbf{x}|^{\gamma}\\
 & \ge\frac{\delta_{0}}{2^{\gamma}}\mathbb{E}\Vert\Pi_{n}\Vert^{\gamma}\cdot\left(\inf_{\mathbf{x}\in\mathbb{R}^{d}}\mathbb{E}|\mathbf{X}_{N_{\gamma}}(\mathbf{x})-\mathbf{x}|^{\gamma}\right)\;.
\end{align*}
By Lemma \ref{lem:gamma1}-(2) and \eqref{eq:div}, we can find $N_{1}(\gamma)>0$
so that the right-hand side of the previous bound is greater than
$2\mathbb{E}|\mathbf{X}_{N_{\gamma}}(\mathbf{0})|^{\gamma}+2$ for
all $n\ge N_{1}(\gamma)$. Let us take 
\[
M_{\gamma}=\max\left\{ N_{0}(\gamma),\,N_{1}(\gamma)\right\} +N_{\gamma}
\]
so that for $n\ge M_{\gamma}$, the previous bound implies that, for
all $\mathbf{x}\in\mathbb{R}^{d}$, 
\[
\mathbb{E}|\mathbf{X}_{n}(\mathbf{x})|^{\gamma}+\mathbb{E}|\mathbf{X}_{n-N_{\gamma}}(\mathbf{x})|^{\gamma}>2\mathbb{E}|\mathbf{X}_{N_{\gamma}}(\mathbf{0})|^{\gamma}+2\;.
\]
This implies $\mathbb{E}|\mathbf{X}_{n}(\mathbf{x})|^{\gamma}>\mathbb{E}|\mathbf{X}_{N_{\gamma}}(\mathbf{0})|^{\gamma}+1$
since we have 
\[
\mathbb{E}|\mathbf{X}_{n-N_{\gamma}}(\mathbf{x})|^{\gamma}\le\frac{1}{2}\left[\mathbb{E}|\mathbf{X}_{n}(\mathbf{x})|^{\gamma}+\mathbb{E}|\mathbf{X}_{N_{\gamma}}(\mathbf{0})|^{\gamma}\right]
\]
by Lemma \ref{lem:pre30}. 
\end{proof}

\subsection{\label{sec52}Upper bound of mean exit time}

We start from a lemma which is useful in the proof of both upper and
lower bounds. We denote by $\mathcal{F}_{n}$ the $\sigma$-algebra
generated by $(\mathbf{A}_{k},\,\mathbf{B}_{k})_{k\in\llbracket0,\,n\rrbracket}$
so that the process $(\mathbf{X}_{n})_{n\in\mathbb{Z}^{+}}$
is adapted to the filtration $(\mathcal{F}_{n})_{n\in\mathbb{Z}^{+}}$. 
\begin{lem}
\label{lem:gub}Let $\gamma\in(0,\,\alpha_{\infty})$
and $N\in\mathbb{Z}^{+}$. Then, there exist a constant $C_{\gamma,\,N}>0$
such that, for all $\mathbf{x}\in\mathbb{R}^{d}$ and $(\mathcal{F}_{n})_{n\in\mathbb{Z}^{+}}$-stopping
time $\tau$, 
\[
\max\left\{ \mathbb{E}|\mathbf{X}_{\tau}(\mathbf{x})|^{\gamma},\,\mathbb{E}|\mathbf{X}_{\tau+1}(\mathbf{x})|^{\gamma},\,\dots,\,\mathbb{E}|\mathbf{X}_{\tau+N-1}(\mathbf{x})|^{\gamma}\right\} \le C_{\gamma,\,N}\left(\mathbb{E}|\mathbf{X}_{\tau}(\mathbf{x})|^{\gamma}+1\right)\;.
\]
\end{lem}

\begin{proof}
Fix $\gamma\in(0,\,\alpha_{\infty})$. In view of \eqref{eq:AMP_rep}
and \eqref{elem_ine}, for all $n\in\mathbb{Z}^{+}$ and $\mathbf{x}\in\mathbb{R}^{d}$,
\[
\mathbb{E}|\mathbf{X}_{n}(\mathbf{x})|^{\gamma}\le2^{\gamma}\mathbb{E}\Vert\Pi_{n}\Vert^{\gamma}\cdot|\mathbf{x}|^{\gamma}+2^{\gamma}\mathbb{E}\left[\left|\mathbf{A}_{2}\cdots\mathbf{A}_{n}\mathbf{B}_{1}+\cdots+\mathbf{A}_{n}\mathbf{B}_{n-1}+\mathbf{B}_{n}\right|^{\gamma}\right]\;.
\]
Hence, for all $\mathbf{x}\in\mathbb{R}^{d},$we have 
\begin{equation}
\max\left\{ \mathbb{E}|\mathbf{X}_{0}(\mathbf{x})|^{\gamma},\,\mathbb{E}|\mathbf{X}_{1}(\mathbf{x})|^{\gamma},\,\dots,\,\mathbb{E}|\mathbf{X}_{N-1}(\mathbf{x})|^{\gamma}\right\} \le C_{\gamma,\,N}\left(|\mathbf{x}|^{\gamma}+1\right)\;,\label{eq:gub}
\end{equation}
where 
\[
C_{\gamma,\,N}=\max_{n\in\llbracket0,\,N-1\rrbracket}\left\{ 2^{\gamma}\mathbb{E}\Vert\Pi_{n}\Vert^{\gamma},\,2^{\gamma}\mathbb{E}\left[\left|\mathbf{A}_{2}\cdots\mathbf{A}_{n}\mathbf{B}_{1}+\cdots+\mathbf{A}_{n}\mathbf{B}_{n-1}+\mathbf{B}_{n}\right|^{\gamma}\right]\right\} \;.
\]
We note that $C_{\gamma,\,N}<\infty$ because of $\mathbb{E}\Vert\mathbf{A}\Vert^{\gamma}<\infty$
and $\mathbb{E}|\mathbf{B}|^{\gamma}<\infty$ from Theorem \ref{Thm_kes0}-(2)
and \eqref{eq:nc2-2}. Now the conclusion of the lemma follows from
\eqref{eq:gub} and the strong Markov property since $\mathbf{X}_{\tau+n}(\mathbf{x})=\mathbf{X}_{n}(\mathbf{X}_{\tau}(\mathbf{x}))$
for $n\in\mathbb{Z}^{+}$ and stopping time $\tau$. 
\end{proof}
For $\gamma\in(\alpha,\,\alpha_{+})$, recall the constant $N_{\gamma}>0$
from Lemma \ref{lem:pre30}. We can infer from Lemma \ref{lem:pre30}
that the norm $|\mathbf{X}_{n}(\mathbf{x})|^{\gamma}$ is diverging
when we count the process by $N_{\gamma}$ skipping. Hence, we write
in this subsection 
\begin{equation}
\widehat{\mathbf{X}}_{n}(\mathbf{x})=\mathbf{X}_{nN_{\gamma}}(\mathbf{x})\;\;\;\;;\;n\in\mathbb{Z}^{+}\;.\label{eq:Xhat}
\end{equation}
To analyze the behavior of this accelerated process, we define 
\begin{equation}
\widehat{\tau}_{R}(\mathbf{x})=\inf\left\{ n\in\mathbb{Z}^{+}:\tau_{R}(\mathbf{x})\le nN_{\gamma}\right\} \;.\label{eq:tauhat}
\end{equation}
We note that $\widehat{\mathbf{X}}_{n}(\mathbf{x})$ and $\widehat{\tau}_{R}(\mathbf{x})$
depend on $\gamma$, although we did not emphasize this dependency. 
\begin{lem}
\label{lem:dtau}For each $\gamma\in(\alpha,\,\alpha_{+})$, there
exists a constant $C_{\gamma}>0$ such that, for all $R>0$ and $\mathbf{x}\in\mathbb{R}^{d}$,
\[
\mathbb{E}|\widehat{\mathbf{X}}_{\widehat{\tau}_{R}(\mathbf{x})}(\mathbf{x})|^{\gamma}<C_{\gamma}\left(R^{\gamma}+1\right)\ .
\]
\end{lem}

\begin{proof}
Fix $\gamma\in(\alpha,\,\alpha_{+})$, $R>0$ and $\mathbf{x}\in\mathbb{R}^{d}$.
By definition \eqref{eq:tauhat} of $\widehat{\tau}_{R}(\mathbf{x})$,
we have 
\[
\widehat{\tau}_{R}(\mathbf{x})N_{\gamma}\in\left\{ \tau_{R}(\mathbf{x}),\,\tau_{R}(\mathbf{x})+1,\,\dots,\,\tau_{R}(\mathbf{x})+N_{\gamma}-1\right\} \;.
\]
Since $\tau_{R}(\mathbf{x})$ is $(\mathcal{F}_{n})_{n\in\mathbb{Z}^{+}}$-stopping
time, by Lemma \ref{lem:gub}, we have 
\[
\mathbb{E}|\widehat{\mathbf{X}}_{\widehat{\tau}_{R}(\mathbf{x})}(\mathbf{x})|^{\gamma}=\mathbb{E}|\mathbf{X}_{\widehat{\tau}_{R}(\mathbf{x})N_{\gamma}}(\mathbf{x})|^{\gamma}\le C_{\gamma}\left(\mathbb{E}|\mathbf{X}_{\tau_{R}(\mathbf{x})}(\mathbf{x})|^{\gamma}+1\right)\ .
\]
This completes the proof since $\mathbb{E}|\mathbf{X}_{\tau_{R}(\mathbf{x})}(\mathbf{x})|^{\gamma}\le C_{\gamma}R^{\gamma}$
by Lemma \ref{lem:exit}.
\end{proof}
Now we are ready to prove Proposition \ref{prop:cont_ubd}.
\begin{proof}[Proof of Proposition \ref{prop:cont_ubd}]
Fix $\gamma\in(\alpha,\,\alpha_{+})$ and $R>0$. We also fix $\mathbf{x}\in\mathbb{R}^{d}$
and simply write $\widehat{\mathbf{X}}_{n}=\widehat{\mathbf{X}}_{n}(\mathbf{x})$. 

Write 
\[
G_{\gamma}:=\mathbb{E}|\mathbf{X}_{N_{\gamma}}(\mathbf{0})|^{\gamma}>0
\]
where the strict inequality follows from Lemma \eqref{lem:pre1-2}-(2)
and let us define $F=F_{\gamma}:\mathbb{R}^{d}\rightarrow\mathbb{R}$
by
\[
F(\mathbf{x}):=|\mathbf{x}|^{\gamma}-G_{\gamma}\;\;\;\;;\;\mathbf{x}\in\mathbb{R}^{d}\;.
\]
For $n\in\mathbb{Z}^{+}$, denote by $\widehat{\mathcal{F}}_{n}$
the $\sigma$-algebra generated by $(\mathbf{A}_{k},\,\mathbf{B}_{k})_{k\in\llbracket0,\,nN_{\gamma}\rrbracket}$,
and define a $\widehat{\mathcal{F}}_{n}$-measurable random variable
$M_{n}$ by 
\[
M_{n}:=|\widehat{\mathbf{X}}_{n}|^{\gamma}-\sum_{k=0}^{n-1}F(\widehat{\mathbf{X}}_{k})\;.
\]
Then, by Lemma \ref{lem:pre30}, we have 
\[
\mathbb{E}\left[M_{n+1}\,|\,\mathcal{\widehat{F}}_{n}\right]\ge M_{n}
\]
and therefore $(M_{n})_{n\in\mathbb{Z}^{+}}$ is indeed an $(\mathcal{\widehat{F}}_{n})_{n\in\mathbb{Z}^{+}}$-submartingale. By Lemma \ref{lem:pre3-0}, we can find $K_{\gamma}>0$ independent
of $\mathbf{x}$ such that 
\begin{equation}
\frac{1}{n}\sum_{k=0}^{n-1}\mathbb{E}F(\widehat{\mathbf{X}}_{k})\ge\frac{1}{2}\;\;\;\;\text{for all }n\ge K_{\gamma}\;.\label{eq:lba}
\end{equation}
Since $\widehat{\tau}_{R}(\mathbf{x})$ defined in \eqref{eq:tauhat}
is an $(\mathcal{\widehat{F}}_{n})_{n\in\mathbb{Z}^{+}}$-stopping
time, by the optional stopping theorem, for all $n\in\mathbb{Z}^{+}$,
\[
\mathbb{E}|\mathbf{\widehat{\mathbf{X}}}_{n\wedge\widehat{\tau}_{R}(\mathbf{x})}|^{\gamma}-\mathbb{E}\sum_{k=0}^{n\land\widehat{\tau}_{R}(\mathbf{x})-1}F(\mathbf{\widehat{\mathbf{X}}}_{k})=\mathbb{E}\left[M_{n\wedge\widehat{\tau}_{R}(\mathbf{x})}\right]\ge M_{0}=|\mathbf{x}|^{\gamma}\;.
\]
Hence, for $n\ge K_{\gamma}$, by \eqref{eq:lba}, we have 
\begin{align}
\mathbb{E}|\mathbf{\widehat{\mathbf{X}}}_{n\wedge\widehat{\tau}_{R}(\mathbf{x})}|^{\gamma} & \ge\mathbb{E}\left[\sum_{k=0}^{n\land\widehat{\tau}_{R}(\mathbf{x})-1}\mathbb{E}\left[F(\mathbf{\widehat{\mathbf{X}}}_{k})\right]\left(\mathbf{1}\left\{ \widehat{\tau}_{R}(\mathbf{x})\ge K_{\gamma}\right\} +\mathbf{1}\left\{ \widehat{\tau}_{R}(\mathbf{x})<K_{\gamma}\right\} \right)\right]\nonumber \\
 & \ge\mathbb{E}\left[\frac{1}{2}(n\land\widehat{\tau}_{R}(\mathbf{x}))\mathbf{1}\left\{ \widehat{\tau}_{R}(\mathbf{x})\ge K_{\gamma}\right\} +\sum_{k=0}^{n\land\widehat{\tau}_{R}-1}(-G_{\gamma})\mathbf{1}\left\{ \widehat{\tau}_{R}(\mathbf{x})<K_{\gamma}\right\} \right]\nonumber \\
 & \ge\frac{1}{2}\mathbb{E}\left[n\land\widehat{\tau}_{R}(\mathbf{x})\right]-K_{\gamma}(1+G_{\gamma})\;,\label{eq:pren}
\end{align}
where the second inequality follows from $F\ge-G_{\gamma}$. 

By Proposition \ref{Prop_cont_exit time finite} and \eqref{eq:tauhat},
we have $\widehat{\tau}_{R}(\mathbf{x})<\infty$ almost surely, and
therefore by the monotone convergence theorem, 
\[
\mathbb{E}\left[n\land\widehat{\tau}_{R}(\mathbf{x})\right]\nearrow\mathbb{E}\widehat{\tau}_{R}(\mathbf{x})\;\;\;\;\text{as }n\rightarrow\infty\;.
\]
On the other hand, since 
\[
|\mathbf{\widehat{\mathbf{X}}}_{n\wedge\widehat{\tau}_{R}(\mathbf{x})}|^{\gamma}\le\max\left\{ |\mathbf{\widehat{\mathbf{X}}}_{\widehat{\tau}_{R}(\mathbf{x})}|^{\gamma},\,R^{\gamma}\right\} \;,
\]
and since the random variable at the right-hand side has finite expectation
by Lemma \ref{lem:dtau}, by the dominated convergence theorem (along
with the fact that $\widehat{\tau}_{R}(\mathbf{x})<\infty$ almost
surely), we get 
\[
\lim_{n\rightarrow\infty}\mathbb{E}|\mathbf{\widehat{\mathbf{X}}}_{n\wedge\widehat{\tau}_{R}(\mathbf{x})}|^{\gamma}=\mathbb{E}|\mathbf{\widehat{\mathbf{X}}}_{\widehat{\tau}_{R}(\mathbf{x})}|^{\gamma}\le C_{\gamma}\left(R^{\gamma}+1\right)\;,
\]
where the last inequality follows from Lemma \ref{lem:dtau}. Therefore,
letting $n\rightarrow\infty$ in \eqref{eq:pren} and applying $\tau_{R}(\mathbf{x})\le\widehat{\tau}_{R}(\mathbf{x})N_{\gamma}$,
we get \eqref{eq:pp1}. 
\end{proof}

\subsection{\label{sec53}Lower bound of mean exit time}

For the lower bound, we need the following construction of the Lyapunov
function associated with process \eqref{eq:AMP}. 
\begin{lem}
\label{lem:lyap0}Let $r>0$. Suppose that $\mathbf{U}$ is a $d\times d$
random matrix and $\mathbf{V}$ is a $d$-dimensional random vector
satisfying
\begin{equation}
\mathbb{E}\left[\Vert\mathbf{U}\Vert^{r}\right]\le1,\;\;\;\;\text{and}\;\;\;\;\mathbb{E}\left[|\mathbf{V}|^{r}\right]\in(0,\,\infty)\;,\label{eq:condU}
\end{equation}
and $\mathbb{E}\left[\Vert\mathbf{U}\Vert^{s}\right]<1$ for all $s\in(0,\,r)$.
Then, there exist constants $c_{1},\dots,c_{\lfloor r\rfloor}>0$
such that the function 
\begin{equation}
g(\mathbf{x})=g_{\mathbf{U},\mathbf{V}}(\mathbf{x}):=\begin{cases}
|\mathbf{x}|^{r}+c_{1}|\mathbf{x}|^{r-1}+\cdots+c_{\lfloor r\rfloor}|\mathbf{x}|^{r-\lfloor r\rfloor} & \text{if}\ r>1\\
|\mathbf{x}|^{r} & \text{if}\ r\le1
\end{cases}\label{eq:gx}
\end{equation}
satisfies, for all $\mathbf{x}\in\mathbb{R}^{d}$, 
\[
\mathbb{E}\left[g(\mathbf{U}\mathbf{x}+\mathbf{V})\right]-g(\mathbf{x})\le\mathbb{E}g(\mathbf{V})\;.
\]
\end{lem}

We postpone the proof of this lemma to the next subsection and prove
Proposition \ref{prop:cont_lbd} by assuming it. For $\gamma\in(0,\,\alpha)$,
by Theorem \ref{Thm_kes0}-(4), we have 
\[
\lim_{n\rightarrow\infty}\left(\mathbb{E}\Vert\mathbf{A}_{1}\cdots\mathbf{A}_{n}\Vert^{\gamma}\right)^{1/n}=h_{\mathbf{A}}(\gamma)<1\;.
\]
Thus, we can find $L_{\gamma}\ge1$ such that 
\begin{equation}
\ensuremath{\mathbb{E}\Vert\mathbf{A}_{1}\cdots\mathbf{A}_{L_{\gamma}}\Vert^{\gamma}<1\;}.\label{eq:condnorm}
\end{equation}
For $n\in\mathbb{Z}^{+}$ and $i\in\llbracket0,\,L_{\gamma}-1\rrbracket$,
denote by $\widehat{\mathcal{F}}_{n}^{(i)}$ the $\sigma$-algebra
generated by $(\mathbf{A}_{k},\,\mathbf{B}_{k})_{k\in\llbracket1,\,nL_{\gamma}+i\rrbracket}$
and write 
\[
\widehat{\mathbf{X}}_{n}^{(i)}(\mathbf{x}):=\mathbf{X}_{nL_{\gamma}+i}(\mathbf{x})\;\;\;\;;\;\mathbf{x}\in\mathbb{R}^{d}
\]
which is $\mathcal{\widehat{F}}_{n}^{(i)}$-measurable random variable.
We again note that $\widehat{\mathcal{F}}_{n}^{(i)}$ and $\widehat{\mathbf{X}}_{n}^{(i)}(\mathbf{x})$
depend on $\gamma$ although we did not highlight the dependency.
The following lemma is a direct consequence of Lemma \ref{lem:lyap0}. 
\begin{lem}
\label{lem:lya2}For $\gamma\in(0,\,\alpha)$, there exist constants
$c_{1},\,\dots,\,c_{\lfloor\gamma\rfloor}>0$ (depending on $\gamma$)
such that the function $g=g_{\gamma}:\mathbb{R}^{d}\rightarrow\mathbb{R}$
defined by 
\[
g(\mathbf{x})=|\mathbf{x}|^{\gamma}+c_{1}|\mathbf{x}|^{\gamma-1}+\cdots+c_{\lfloor\gamma\rfloor}|\mathbf{x}|^{\gamma-\lfloor\gamma\rfloor}
\]
satisfies, for all $n\in\mathbb{Z}^{+}$, $i\in\llbracket0,\,L_{\gamma}-1\rrbracket$,
and $\mathbf{x}\in\mathbb{R}^{d}$, 
\[
\mathbb{E}\left[g(\widehat{\mathbf{X}}_{n+1}^{(i)}(\mathbf{x}))\,|\,\widehat{\mathcal{F}}_{n}^{(i)}\right]\le g(\widehat{\mathbf{X}}_{n}^{(i)}(\mathbf{x}))+\mathbb{E}\left[g(\mathbf{X}_{L_{\gamma}}(\mathbf{0}))\right]\;.
\]
\end{lem}

\begin{proof}
Fix $\gamma\in(0,\,\alpha)$. Fix $\mathbf{x}\in\mathbb{R}^{d}$ and
simply write $\mathbf{\widehat{\mathbf{X}}}_{n}^{(i)}=\mathbf{\widehat{\mathbf{X}}}_{n}^{(i)}(\mathbf{x})$.
Then, for all $n\in\mathbb{Z}^{+}$ and $i\in\llbracket0,\,L_{\gamma}-1\rrbracket$,
we can write 
\begin{equation}
\mathbf{\widehat{\mathbf{X}}}_{n+1}^{(i)}=\mathbf{U}_{n}^{(i)}\mathbf{\widehat{\mathbf{X}}}_{n}^{(i)}+\mathbf{V}_{n}^{(i)}\label{eq_amp2-1}
\end{equation}
where 
\begin{align*}
\mathbf{U}_{n}^{(i)} & =\mathbf{A}_{(n+1)L_{\gamma}+i}\mathbf{A}_{(n+1)L_{\gamma}+(i-1)}\cdots\mathbf{A}_{nL_{\gamma}+i+1}\;,\\
\mathbf{V}_{n}^{(i)} & =\mathbf{A}_{(n+1)L_{\gamma}+i}\cdots\mathbf{A}_{nL_{\gamma}+(i+2)}\mathbf{B}_{nL_{\gamma}+(i+1)}+\cdots+\mathbf{A}_{(n+1)L_{\gamma}+i}\mathbf{B}_{(n+1)L_{\gamma}+(i-1)}+\mathbf{B}_{(n+1)L_{\gamma}+i}\;.
\end{align*}
In particular, by \eqref{eq:condnorm}, we have $\mathbb{E}\Vert\mathbf{U}_{n}^{(i)}\Vert^{\gamma}<1$
and therefore, by Proposition \ref{lem:lyap0}, there exists a function
$g=g_{\gamma}:\mathbb{R}^{d}\rightarrow\mathbb{R}$ (independent of
$i$) of the form 
\[
g(\mathbf{x})=|\mathbf{x}|^{\gamma}+c_{1}|\mathbf{x}|^{\gamma-1}+\cdots+c_{\lfloor\gamma\rfloor}|\mathbf{x}|^{\gamma-\lfloor\gamma\rfloor}
\]
with $c_{1},\,\dots,\,c_{\lfloor\gamma\rfloor}>0$ such that (cf.
\eqref{eq_amp2-1})
\begin{equation}
\mathbb{E}\left[g(\widehat{\mathbf{X}}_{n+1}^{(i)})\,|\,\widehat{\mathcal{F}}_{n}^{(i)}\right]\le g(\widehat{\mathbf{X}}_{n}^{(i)})+\mathbb{E}\left[g(\mathbf{V}_{n}^{(i)})\right]\;.\label{eq:sub1}
\end{equation}
The proof is completed since $\mathbf{V}_{n}^{(i)}$ and $\mathbf{X}_{L_{\gamma}}(\mathbf{0})$
obey the same distribution by \eqref{eq:AMP_rep}. 
\end{proof}
Now we are ready to prove Proposition \ref{prop:cont_lbd}.
\begin{proof}[Proof of Proposition \ref{prop:cont_lbd}]
Fix $\gamma\in(0,\,\alpha)$ and $R>0$. Fix $\mathbf{x}\in\mathbb{R}^{d}$
and simply write $\mathbf{\widehat{\mathbf{X}}}_{n}^{(i)}=\mathbf{\widehat{\mathbf{X}}}_{n}^{(i)}(\mathbf{x})$.
Writing 
\begin{equation}
G_{\gamma}:=\mathbb{E}\left[g(\mathbf{X}_{L_{\gamma}}(\mathbf{0}))\right]\label{eq:sub2}
\end{equation}
so that $G_{\gamma}>0$ because of Lemma \eqref{lem:pre1-2}-(2).
Then, define a sequence of random variables $(M_{n}^{(i)})_{n\in\mathbb{Z}_{0}^{+}}$
by 
\[
M_{n}^{(i)}=g(\mathbf{\widehat{\mathbf{X}}}_{n}^{(i)})-nG_{\gamma}\;\;\;\;;\;n\in\mathbb{Z}_{0}^{+}
\]
so that $(M_{n}^{(i)})_{n\in\mathbb{Z}_{0}^{+}}$ is not only adapted
to the filtration $(\mathcal{\widehat{F}}_{n}^{(i)})_{n\in\mathbb{Z}_{0}^{+}}$,
but also a $(\mathcal{\widehat{F}}_{n}^{(i)})_{n\in\mathbb{Z}_{0}^{+}}$-supermartingale
by Lemma \ref{lem:lya2}. 

We next define, for $i\in\llbracket0,\,L_{\gamma}-1\rrbracket$, 
\begin{equation}
\widehat{\tau}_{R}^{(i)}(\mathbf{x})=\inf\left\{ n\in\mathbb{Z}^{+}:\tau_{R}(\mathbf{x})\le nL_{\gamma}+i\right\} \label{taui}
\end{equation}
so that $\widehat{\tau}_{R}^{(i)}(\mathbf{x})$ is a $(\mathcal{F}_{n}^{(i)})_{n\in\mathbb{Z}_{0}^{+}}$-stopping
time. Thus, by the optional stopping theorem, for all $n\in\mathbb{Z}^{+}$,
we have 
\begin{equation}
\mathbb{E}M_{n\wedge\widehat{\tau}_{R}^{(i)}(\mathbf{x})}^{(i)}\le\mathbb{E}M_{0}^{(i)}=\mathbb{E}g(\mathbf{X}_{i}(\mathbf{x}))\;.\label{eq:sup1}
\end{equation}
By Lemma \ref{lem:gub}, we have 
\[
\max_{i\in\llbracket0,\,L_{\gamma}-1\rrbracket}\mathbb{E}g(\mathbf{X}_{i}(\mathbf{x}))\le C_{\gamma}(g(\mathbf{x})+1)\;.
\]
Inserting this and the fact that 
\[
\mathbb{E}M_{n\wedge\widehat{\tau}_{R}^{(i)}(\mathbf{x})}^{(i)}=\mathbb{E}g(\widehat{\mathbf{X}}_{n\wedge\widehat{\tau}_{R}^{(i)}(\mathbf{x})}^{(i)})-\mathbb{E}[n\wedge\widehat{\tau}_{R}^{(i)}(\mathbf{x})]\cdot G_{\gamma}
\]
to \eqref{eq:sup1} yields that, for some constants $C_{\gamma},\,D_{\gamma}>0$,
\[
\mathbb{E}[n\wedge\widehat{\tau}_{R}^{(i)}(\mathbf{x})]\ge C_{\gamma}\mathbb{E}g(\widehat{\mathbf{X}}_{n\wedge\widehat{\tau}_{R}^{(i)}(\mathbf{x})}^{(i)})-D_{\gamma}(|\mathbf{x}|^{\gamma}+1)\;.
\]
Since $\mathbb{P}[\widehat{\tau}_{R}^{(i)}(\mathbf{x})<\infty]=1$
by Proposition \ref{Prop_cont_exit time finite}, we can apply the
monotone convergence theorem and Fatou's lemma at the left- and right-hand
side, respectively, along $n\rightarrow\infty$, we get 
\begin{equation}
\mathbb{E}[\widehat{\tau}_{R}^{(i)}(\mathbf{x})]\ge C_{\gamma}\mathbb{E}g(\widehat{\mathbf{X}}_{\widehat{\tau}_{R}^{(i)}(\mathbf{x})}^{(i)})-D_{\gamma}(|\mathbf{x}|^{\gamma}+1)\;.\label{eq:lbb}
\end{equation}
Noting from \eqref{taui} that 
\[
(\widehat{\tau}_{R}^{(i)}(\mathbf{x})-1)L_{\gamma}+i<\tau_{R}(\mathbf{x})
\]
and thus we get from \eqref{eq:lbb} (along with trivial bound $g(\mathbf{x})\ge|\mathbf{x}|^{\gamma}$)
that
\begin{equation}
\mathbb{E}[\tau_{R}(\mathbf{x})]\ge C_{\gamma}\max_{i\in\llbracket0,\,L_{\gamma}-1\rrbracket}\mathbb{E}\left|\widehat{\mathbf{X}}_{\widehat{\tau}_{R}^{(i)}(\mathbf{x})}^{(i)}\right|^{\gamma}-D_{\gamma}(|\mathbf{x}|^{\gamma}+1)\;.\label{eq:lbf}
\end{equation}
On the other hand, in view of the definition of the stopping time
$\widehat{\tau}_{R}^{(i)}(\mathbf{x})$, one of \\$\widehat{\mathbf{X}}_{\widehat{\tau}_{R}^{(0)}(\mathbf{x})}^{(0)},\,\widehat{\mathbf{X}}_{\widehat{\tau}_{R}^{(1)}(\mathbf{x})}^{(1)},\,\dots,\,\widehat{\mathbf{X}}_{\widehat{\tau}_{R}^{(L_{\gamma}-1)}(\mathbf{x})}^{(L_{\gamma}-1)}$
coincides with $\mathbf{X}_{\tau_{R}(\mathbf{x})}(\mathbf{x})$ and
therefore, 
\[
\max_{i\in\llbracket0,\,L_{\gamma}-1\rrbracket}\mathbb{E}\left|\mathbf{Y}_{\widehat{\tau}_{R}^{(i)}(\mathbf{x})}^{(i)}\right|^{\gamma}\ge R^{\gamma}\;.
\]
Inserting this to \eqref{eq:lbf} completes the proof. 
\end{proof}

\subsection{\label{sec54}Proof of Lemma \ref{lem:lyap0}}

To prove Lemma \ref{lem:lyap0}, we need two elementary facts. For
$a\in\mathbb{R}$ and $k\in\mathbb{Z}_{0}^{+}$, write 
\[
\begin{pmatrix}a\\
k
\end{pmatrix}=\frac{a(a-1)\cdots(a-k+1)}{k!}\;.
\]

\begin{lem}
\label{lem_ele1}For all $x,\,y\ge0$ and $r\ge0$, we have 
\[
(x+y)^{r}\le\left(x^{r}+\begin{pmatrix}r\\
1
\end{pmatrix}x^{r-1}y+\cdots+\begin{pmatrix}r\\
\lfloor r\rfloor
\end{pmatrix}x^{r-\lfloor r\rfloor}y^{\lfloor r\rfloor}\right)+y^{r}\;,
\]
where the right-hand side is $x^{r}+y^{r}$ when $r<1$. 
\end{lem}

\begin{proof}
Write $\lfloor r\rfloor=n$ so that $n\le r<n+1$. Fix $y>0$ and
$n\in\mathbb{Z}^{+}$ (since the inequality is straighforward if $y=0$
or $n=0$) and define $f:[0,\infty)\rightarrow\mathbb{R}$ as 
\[
f(x)=\left(x^{r}+\begin{pmatrix}r\\
1
\end{pmatrix}x^{r-1}y+\cdots+\begin{pmatrix}r\\
n
\end{pmatrix}x^{r-n}y^{n}\right)+y^{r}-(x+y)^{r}.
\]
Then, a straight forward computation yields
\[
f^{(n+1)}(x)=r(r-1)\cdots(r-n)(x^{r-n-1}-(x+y)^{r-n-1})\;,
\]
where $f^{(k)}$ denotes the $k$th derivative of $f$. Since $r-n-1<0$,
we have $f^{(n+1)}>0$. Since we can readily check that $f(0)=f'(0)=\cdots=f^{(n)}(0)=0$,
we can inductively show that $f^{(n)}>0$, $\dots$ , $f^{(1)}>0$
and finally, $f>0$.
\end{proof}
For a vector $\mathbf{v}=(v_{1},\,\dots,\,v_{n})\in\mathbb{R}^{n}$,
we write $\mathbf{v}>0$ if $v_{1},\,\dots,\,v_{n}>0$, and $\mathbf{v}\ge0$
if $v_{1},\,\dots,\,v_{n}\ge0$.
\begin{lem}
\label{lem:ut}Let $\mathbb{H}$ be an $n\times n$ upper triangular
matrix such that $\mathbb{H}_{1,\,1}\ge0$ and $\mathbb{H}_{i,\,i}>0$
for all $i\ge2$. Then, there exists $\mathbf{v}=(v_{1},\,\dots,\,v_{n})\in\mathbb{R}^{n}$
such that $\mathbf{v}>0$ and $\mathbf{v}\mathbb{H}\ge0$. 
\end{lem}

\begin{proof}
\noindent We use induction on $n$. There is nothing to prove when
$n=1$. Suppose next that the statement holds for $n=k$. If $\mathbb{H}$
is a $(k+1)\times(k+1)$ upper triangular matrix such that all of
its diagonal entries are positive. Let $\widetilde{\mathbb{H}}$ be
a $k\times k$ upper triangular matrix defined by $\widetilde{\mathbb{H}}_{i,\,j}=\mathbb{H}_{i,\,j}$
for all $1\le i,\,j\le k$. Then, since $\widetilde{\mathbb{H}}$
is an upper triangular matrix satisfying the condition for induction
hypothesis, we can find $\widetilde{\mathbf{v}}=(v_{1},\,\dots,\,v_{k})>0$
such that $\widetilde{\mathbf{v}}\widetilde{\mathbb{H}}\ge0$. Then,
let $\mathbf{v}=(v_{1},\,\dots,\,v_{k},\,t)$, so that 
\[
\mathbf{v}\mathbb{H}=\left(\widetilde{\mathbf{v}}\widetilde{\mathbb{H}},\,\sum_{i=1}^{k}v_{i}\mathbb{H}_{i,\,k+1}+t\mathbb{H}_{k+1,\,k+1}\right)\;.
\]
Since $\mathbb{H}_{k+1,\,k+1}>0$, we can take $t>0$ large enough
so that $\mathbf{v}\mathbb{H}\ge0$. 
\end{proof}
Let us now prove Lemma \ref{lem:lyap0}.
\begin{proof}[Proof of Lemma \ref{lem:lyap0}]
Write $\lfloor r\rfloor=n$ so that $n\le r<n+1$. We define an upper
triangular $(n+1)\times(n+1)$ matrix $\mathbb{H}=(h_{i,\,j})_{i,\,j\in\llbracket0,\,n\rrbracket}$
by
\[
\begin{cases}
h_{i,\,i}=\mathbb{E}\Vert\mathbf{U}\Vert^{r-i}-1 & \text{for }i\in\llbracket0,\,n\rrbracket\\
h_{i,\,k}=\begin{pmatrix}r-i\\
k-i
\end{pmatrix}(\mathbb{E}[\Vert\mathbf{U}\Vert^{r-k}|\mathbf{V}|^{k-i}]) & \text{for }i\in\llbracket0,\,n\rrbracket\;\text{and}\;k>i\\
0 & \text{otherwise}\;.
\end{cases}
\]
Then, for all $i\in\llbracket0,\,n\rrbracket$, by Lemma \ref{lem_ele1},
\begin{equation}
\mathbb{E}\left[|\mathbf{U}\mathbf{x}+\mathbf{V}|^{r-i}\right]-|\mathbf{x}|^{r-i}\le\mathbb{E}(\Vert\mathbf{U}\Vert\cdot|\mathbf{x}|+|\mathbf{V}|)^{r-i}-|\mathbf{x}|^{r-i}\le\sum_{k=i}^{n}h_{i,\,k}|\mathbf{x}|^{r-k}+\mathbb{E}|\mathbf{V}|^{r-i}\;.\label{eq:bdg}
\end{equation}
From \eqref{eq:condU}, we have $h_{0,\,0}=0$ and $h_{i,\,i}<0$
for all $i\in\llbracket1,\,n\rrbracket$. Namely, the diagonal entry
of the upper triangular matrix $\mathbb{H}$ is negative except for
$(1,1)$-component which is exactly zero. Hence, by Lemma \ref{lem:ut},
we can find a vector $\mathbf{c}=(1,\,c_{1},\,\dots,\,c_{n})>0$ such
that all the entries of $\mathbf{c}\mathbb{H}$  are non-positive.
By taking $g(\mathbf{x})$ as in \eqref{eq:gx} with this $c_{1},\,\dots,\,c_{n}$,
we get from \eqref{eq:bdg} that 
\begin{align*}
\mathbb{E}\left[g(\mathbf{U}\mathbf{x}+\mathbf{V})\right]-g(\mathbf{x}) & \le\mathbb{E}g(\mathbf{V})+\sum_{i=0}^{n}\sum_{k=0}^{n}c_{i}h_{i,\,k}|\mathbf{x}|^{r-k}\;.\\
 & =\mathbb{E}g(\mathbf{V})+\sum_{k=0}^{n}(\mathbf{c}\mathbb{H})_{k}|\mathbf{x}|^{r-k}\le\mathbb{E}g(\mathbf{V})\;.
\end{align*}
\end{proof}

\subsection{\label{sec55}Univariate case}

We now analyze one-dimensional, i.e., the univariate case, for which
we can obtain more concrete result than the multivariate case. We
remark here that in the univariate case, by the independence of the
sequence $(\mathbf{A}_{n})_{n\in\mathbb{Z}^{+}}$, we have 
\[
\frac{1}{n}\log\mathbb{E}\left|\mathbf{A}_{1}\cdots\mathbf{A}_{n}\right|^{s}=\mathbb{E}|\mathbf{A}|^{s}
\]
and therefore we can simplify $h_{\mathbf{A}}(\cdot)$ into 
\begin{equation}
h_{\mathbf{A}}(s)=\mathbb{E}|\mathbf{A}|^{s}\;.\label{eq:h_A-1}
\end{equation}

\begin{thm}
\label{thm_uni}Let $d=1$.
\begin{enumerate}
\item There exist constants $C_{1},\,C_{2}>0$ such that, for all $\mathbf{x}_{0}\in\mathbb{R}$,
\begin{equation}
\mathbb{E}\left[\tau_{R}(\mathbf{x}_{0})\right]\ge C_{1}R^{\alpha}-C_{2}(|\mathbf{x}_{0}|^{\alpha}+1)\;.\label{eq:ea}
\end{equation}
\item If $\alpha\ge2$, there exist constants $C_{3}>0$ such that, for
all $\mathbf{x}_{0}\in\mathbb{R}$, 
\[
\mathbb{E}\left[\tau_{R}(\mathbf{x}_{0})\right]\le C_{3}\left(R^{\alpha}+1\right)\;.
\]
\end{enumerate}
\end{thm}

Before proving this theorem, we establish an elementary inequality. 
\begin{lem}
\label{lem_ele3}For all $\alpha\ge2$, there exists a constant $C_{\alpha}>0$
such that, for all $x,\,y\in\mathbb{R}$
\[
|x+y|^{\alpha}\ge|x|^{\alpha}+\alpha|x|^{\alpha-2}xy+C_{\alpha}|y|^{\alpha}\;.
\]
\end{lem}

\begin{proof}
The statement is clear for $y=0$ or $\alpha=2$ (for which $C_{\alpha}=1$).
Hence, assume that $y\neq0$ and $\alpha>2$. By renormalize suitably,
we can assume that $y=1$ and hence it suffices to prove that there
exists a constant $C_{\alpha}>0$ such that for all $x\in\mathbb{R}$,
\[
|x+1|^{\alpha}\ge|x|^{\alpha}+\alpha|x|^{\alpha-2}x+C_{\alpha}.
\]
Let $f(x)=|x+1|^{\alpha}-\left(|x|^{\alpha}+\alpha|x|^{\alpha-2}x\right)$
and then it suffices to check that $\inf_{x\in\mathbb{R}}f(x)>0$.
By a simple computation, one can directly check that $f'<0$ on $(-\infty,\,-1]$
and $f'>0$ on $(0,\,\infty)$. Hence, it suffices to check that 
\begin{equation}
\inf_{x\in[-1,\,0]}f(x)=\inf_{x\in[0,\,1]}f(-x)>0\;.\label{eq:inff}
\end{equation}
Writing $g(x)=f(-x)$, we have 
\[
g'(x)=\alpha x^{\alpha-1}\left(-\left(\frac{1}{x}-1\right)^{\alpha-1}-1+(\alpha-1)\frac{1}{x}\right)\;.
\]
By an elementary computation, we can check that there exists a unique
solution $x_{0}\in(0,\,1)$ of $g'(x_{0})=0$ and $g'<0$ on $(0,\,x_{0})$
and $g'>0$ on $(x_{0},\,1)$. Hence, $g$ attains its minimum on
$(0,\,1)$ at $x_{0}$. Furthermore, we can check that 
\[
g(x_{0})=(\alpha-1)x_{0}^{\alpha-2}>0
\]
and therefore we get \eqref{eq:inff}. 
\end{proof}
Now we turn to the proof of Theorem \ref{thm_uni}
\begin{proof}[Proof of Theorem \ref{thm_uni}]
(1) Since $h_{\mathbf{A}}(\alpha)=1$, by \eqref{eq:h_A-1}, we have
$\mathbb{E}|\mathbf{A}|^{\alpha}=1$. Therefore, the proof of Proposition
\ref{prop:cont_lbd} works for $\gamma=\alpha$ in univariate case,
and hence we get \eqref{eq:ea}. 

\noindent (2) Let us fix $\mathbf{x}_{0}\in\mathbb{R}$. We first
assume that 
\begin{equation}
\mathbb{E}\left[|\mathbf{A}|^{\alpha-2}\mathbf{A}\mathbf{B}\right]=0\;.\label{eq:sind}
\end{equation}
By this assumption, Lemma \ref{lem_ele3}, and the fact that $\mathbb{E}[|\mathbf{A}|^{\alpha}]=1$,
for each $n\in\mathbb{Z}^{+}$, we have (regardless of the starting
point of the process \eqref{eq:AMP}), 
\begin{align*}
\mathbb{E}[|\mathbf{X}_{n+1}|^{\alpha}\ |\ \mathbf{X}_{n}]-|\mathbf{X}_{n}|^{\alpha} & =\mathbb{E}[|\mathbf{A}_{n+1}\mathbf{X}_{n}+\mathbf{B}_{n+1}|^{\alpha}-|\mathbf{A}_{n+1}\mathbf{X}_{n}|^{\alpha}\ |\ \mathbf{X}_{n}]\\
 & \ge\mathbb{E}[\alpha|\mathbf{A}_{n+1}\mathbf{X}_{n}|^{\alpha-2}(\mathbf{A}_{t+1}\mathbf{X}_{n})\mathbf{B}_{n+1}+C_{\alpha}|\mathbf{B}_{n+1}|^{\alpha}\ |\ \mathbf{X}_{n}]\\
 & =\alpha\mathbb{E}\left[|\mathbf{A}|^{\alpha-2}\mathbf{A}\mathbf{B}\right]\mathbf{X}_{n}+C_{\alpha}\mathbb{E}|\mathbf{B}|^{\alpha}\\
 & =C_{\alpha}\mathbb{E}|\mathbf{B}|^{\alpha}\;.
\end{align*}
Let $\kappa_{0}=C_{\alpha}\mathbb{E}|\mathbf{B}|^{\alpha}$ so that
by the previous observation, the sequence $(M_{n})_{n\in\mathbb{Z}^{+}}$
defined by 
\[
M_{n}:=\mathbf{X}_{n}(\mathbf{x}_{0})-\kappa_{0}n
\]
is a $(\mathcal{F}_{n})_{n\in\mathbb{Z}^{+}}$-submartingale. By the
optional stopping theorem, we get, for all $n\in\mathbb{Z}^{+}$ ,
\[
\mathbb{E}|\mathbf{X}_{n\land\tau_{R}}(\mathbf{x}_{0})|^{\alpha}-\kappa_{0}\mathbb{E}[n\land\tau_{R}(\mathbf{x}_{0})]\ge|\mathbf{x}_{0}|^{\alpha}\ .
\]
By letting $n\rightarrow\infty$ and apply the same logic with the
proof of Proposition \ref{prop:cont_ubd}, we get, for some constant
$C_{0}>0$, 
\begin{equation}
\mathbb{E}[\tau_{R}(\mathbf{x}_{0})]\le C_{0}R^{\alpha}\;.\label{eq:conc1}
\end{equation}

Now we consider the general case without the assumption \eqref{eq:sind}.
Since $\alpha\ge2$, by Theorem \ref{Thm_kes0} we have
\[
\mathbb{E}|\mathbf{A}|^{\alpha-2}\mathbf{A}\le\mathbb{E}|\mathbf{A}|^{\alpha-1}=h(\alpha-1)<h(\alpha)=\mathbb{E}|\mathbf{A}|^{\alpha}\ .
\]
Therefore, we can take $c_{0}\in\mathbb{R}$ such that 
\begin{equation}
\mathbb{E}\left[|\mathbf{A}|^{\alpha-2}\mathbf{A}\mathbf{B}\right]+c_{0}\left(\mathbb{E}|\mathbf{A}|^{\alpha}-\mathbb{E}|\mathbf{A}|^{\alpha-2}\mathbf{A}\right)=0\ .\label{eq:condm}
\end{equation}
Then, we set 

\noindent 
\[
\begin{cases}
\widetilde{\mathbf{X}}_{n}:=\mathbf{X}_{n}(\mathbf{x}_{0})-c_{0}\ , & n\ge0\;,\\
\widetilde{\mathbf{A}}_{n}:=\mathbf{A}_{n}\;\;\;\;\text{and\;\;\;\;}\widetilde{\mathbf{B}}_{n}:=\mathbf{B}_{n}+c_{0}\mathbf{A}_{n}-c_{0}\;, & n\ge1\;,
\end{cases}
\]
so that we have
\[
\widetilde{\mathbf{A}}_{n+1}\widetilde{\mathbf{X}}_{n}+\widetilde{\mathbf{B}}_{n+1}=\mathbf{A}_{n+1}(\mathbf{X}_{n}(\mathbf{x}_{0})-c_{0})+\mathbf{B}_{n+1}+c_{0}\mathbf{A}_{n+1}-c_{0}=\widetilde{\mathbf{X}}_{n+1}\ .
\]
Hence, $(\widetilde{\mathbf{X}}_{n})_{n\in\mathbb{Z}^{+}}$ is another
process \eqref{eq:AMP} associated with $(\widetilde{\mathbf{A}}_{n},\,\widetilde{\mathbf{B}}_{n})_{n\in\mathbb{Z}^{+}}$.
Since 
\[
\mathbb{E}\left[|\widetilde{\mathbf{A}}|^{\alpha-2}\widetilde{\mathbf{A}}\widetilde{\mathbf{B}}\right]=\mathbb{E}\left[|\mathbf{A}|^{\alpha-2}\mathbf{A}\mathbf{B}\right]+c\left(\mathbb{E}|\mathbf{A}|^{\alpha}-\mathbb{E}|\mathbf{A}|^{\alpha-2}\mathbf{A}\right)=0
\]
by \eqref{eq:cond2}, we get from the first part of the proof (cf.
\eqref{eq:conc1}) that 
\[
\mathbb{E}[\widetilde{\tau}_{R}]\le C_{0}R^{\alpha}
\]
where $\widetilde{\tau}_{R}:=\inf\{n\in\mathbb{Z}^{+}:|\widetilde{\mathbf{X}}_{n}|>R\}$.
By the definition of $\widetilde{\mathbf{X}}_{n}$, we have 
\[
\tau_{R}(\mathbf{x}_{0})\le\widetilde{\tau}_{R+|c_{0}|}\ ,
\]
and therefore we get 
\[
\mathbb{E}[\tau_{R}(\mathbf{x}_{0})]\le\mathbb{E}[\widetilde{\tau}_{R+|c_{0}|}]\le C_{0}(R+c_{0})^{\alpha}\;.
\]
This completes the proof. 
\end{proof}

\section{\label{sec6}Proofs of Main Results: Explosive Regime}

The purpose of the current section is to prove Theorem \ref{Thm_main_expl}.
Hence, we assume that the Lyapunov exponent $\gamma_{L}$ is positive,
namely $\gamma_{L}>0$, and moreover assume Assumption \ref{Ass_Exp}
throughout this section. The proof of Theorem
\ref{Thm_main_expl} will be divided into the lower and the upper
bound. We provide a preliminary results regarding explosive regime
in Section \ref{sec61}, and then prove the lower and the upper bound
in Sections \ref{sec62} and \ref{sec63}, respectively. 

\subsection{\label{sec61}Preliminary results}

To derive Theorem \ref{Thm_main_expl}, we need the following technical
lemma.
\begin{lem}
\label{Lem_expl_Lya}If $\mathbf{A}$ is irreducible (cf. Assumption
\ref{Ass_Exp}), we almost surely have that
\[
\lim_{n\rightarrow\infty}\frac{1}{n}\log\Vert\Pi_{n}^{-1}\Vert=-\gamma_{L}
\]
\end{lem}

\begin{proof}
This result is a direct application of Oseledet's theorem; see \cite[Theorem 4.2. page 39]{viana2014lectures}.
\end{proof}
Now we proof the boundedness of the exit time for the explosive regime. 
\begin{prop}
\label{Prop_expl_exit time finite}We have $\mathbb{P}(\tau_{R}(\mathbf{x}_{0})<\infty)=1$
for all $\mathbf{x}\in\mathbb{R}^{d}$ and $R>0$. 
\end{prop}

\begin{proof}
We fix $\mathbf{x}\in\mathbb{R}^{d}$ and $R>0$, and suppose on the
contrary that 
\[
\mathbb{P}\left[\tau_{R}(\mathbf{x}_{0})=\infty\right]>0\;.
\]
The same argument as in the proof of proposition \ref{Prop_cont_exit time finite}
for contractive regime yields the existence of probability measure
$\mu_{\infty}$ satisfying
\begin{equation}
\mathbb{P}_{\mu_{\infty}}\left[\mathbf{X}_{n}\in\mathcal{B}_{R}\right]=1\quad\label{eq:mu_inf}
\end{equation}
for all $n\in\mathbb{Z}^{+}$. That means if we pick two independent
random variables $\mathbf{Y}_{1},\mathbf{Y}_{2}\overset{d}{\sim}\mu_{\infty}$,
we have that
\[
\mathbb{P}_{\mu_{\infty}}\left[|\mathbf{X}_{n}(\mathbf{Y}_{i})|\le R\right]=1,\quad i=1,2,\,n\in\mathbb{Z}^{+}.
\]
Considering coupling method \eqref{eq:couple}
\[
|\mathbf{X}_{n}(\mathbf{Y}_{1})-\mathbf{X}_{n}(\mathbf{Y}_{2})|=|\Pi_{n}(\mathbf{Y}_{1}-\mathbf{Y}_{2})|\ge\Vert\Pi_{n}^{-1}\Vert^{-1}\cdot|\mathbf{Y}_{1}-\mathbf{Y}_{2}|,
\]
and Lemma \ref{Lem_expl_Lya} 
\[
\lim_{n\rightarrow\infty}\frac{1}{n}\log\Vert\Pi_{n}^{-1}\Vert^{-1}=\gamma_{L}>0
\]
gives us that
\[
\mathbb{P}(\mathbf{Y}_{1}=\mathbf{Y}_{2})=1
\]
if we send $n\rightarrow\infty$. Since we drew $\mathbf{Y}_{1},\mathbf{Y}_{2}$
independently from $\mu_{\infty}$, this happens only if when the
support of $\mu_{\infty}$ is a singleton $\left\{ \mathbf{x}_{\infty}\right\} $
for some $\mathbf{x}_{\infty}\in\mathbb{R}^{d}$. Repeatedly, one
step forward from \eqref{eq:mu_inf} gives
\[
\mathbb{P}_{\mathbf{x}_{\infty}}\left[\mathbf{X}_{n}(\mathbf{A}_{1}\mathbf{x}_{\infty}+\mathbf{B}_{1})\in\mathcal{B}_{R}\right]=1,\quad\forall n\in\mathbb{Z}^{+}.
\]
Again, picking $\mathbf{Y}_{1}=\mathbf{x}_{\infty}$ and $\mathbf{Y}_{2}=\mathbf{A}_{1}^{\prime}\mathbf{x}_{\infty}+\mathbf{B}_{1}^{\prime}$
where $(\mathbf{A}_{1}^{\prime},\mathbf{B}_{1}^{\prime})$ is independent
of $\omega$, the same argument above gives that
\[
\mathbb{P}(\mathbf{A}_{1}^{\prime}\mathbf{x}_{\infty}+\mathbf{B}_{1}^{\prime}=\mathbf{x}_{\infty})=1,
\]
which is contradiction to the assumption. 
\end{proof}
\begin{lem}
\label{lem_expl_exit loc}There exists a constant
$C_{1}>0$ such that, for all $R>R_{0}$ (where the constant $R_{0}$
is the one appeared in Assumption \ref{Ass_Exp}) and $\mathbf{x}_{0}\in\mathcal{B}_{R}$,
we have 

\[
\mathbb{E}\log|\mathbf{X}_{\tau_{R}}(\mathbf{x}_{0})|\le C_{1}+\log R\;.
\]
\end{lem}

\begin{proof}
Let $R>R_{0}$. For $\mathbf{x}\in\mathcal{B}_{R}$, by the layer-cake
formula and the fact that 
\[
\lim_{z\rightarrow\infty}\log z\cdot\mathbb{P}(|\mathbf{A}\mathbf{x}+\mathbf{B}|>z)=0
\]
which follows from \eqref{eq:nc2-1}, we can write
\begin{equation}
\mathbb{E}\left[\log|\mathbf{A}\mathbf{x}+\mathbf{B}|\cdot\mathbf{1}\left\{ |\mathbf{A}\mathbf{x}+\mathbf{B}|>R\right\} \right]=\log R\cdot\mathbb{P}(|\mathbf{A}\mathbf{x}+\mathbf{B}|>R)+\int_{1}^{\infty}\mathbb{P}(|\mathbf{A}\mathbf{x}+\mathbf{B}|>Rz)\cdot\frac{1}{z}dz\;.\label{eq:cond1}
\end{equation}
By \eqref{eq:nc2-1-1} of Assumption \ref{Ass_Exp}, (remind that
$\beta_{0}>1$) 
\begin{align}
 & \int_{1}^{\infty}\mathbb{P}(|\mathbf{A}\mathbf{x}+\mathbf{B}|>Rz)\cdot\frac{1}{z}dz\nonumber \\
 & =\int_{1}^{z_{0}}\mathbb{P}(|\mathbf{A}\mathbf{x}+\mathbf{B}|>Rz)\cdot\frac{1}{z}dz+\int_{z_{0}}^{\infty}\mathbb{P}(|\mathbf{A}\mathbf{x}+\mathbf{B}|>Rz)\cdot\frac{1}{z}dz\nonumber \\
 & \le\int_{1}^{z_{0}}\mathbb{P}(|\mathbf{A}\mathbf{x}+\mathbf{B}|>R)\cdot\frac{1}{z}dz+\int_{z_{0}}^{\infty}\frac{C_{0}}{(\log z)^{\beta_{0}}}\cdot\mathbb{P}(|\mathbf{A}\mathbf{x}+\mathbf{B}|>R)\cdot\frac{1}{z}dz\nonumber \\
 & =\left(\log z_{0}+\frac{C_{0}}{(\beta_{0}-1)(\log z)^{\beta_{0}-1}}\right)\mathbb{P}(|\mathbf{A}\mathbf{x}+\mathbf{B}|>R)\ .\label{eq:cond2}
\end{align}
Inserting this to \eqref{eq:cond1}, and dividing both sides by $\mathbb{P}(|\mathbf{A}\mathbf{x}+\mathbf{B}|>R)$,
we get, for some constant $C>0$, 
\[
\mathbb{E}\left[\log|\mathbf{A}\mathbf{x}+\mathbf{B}|\,\big|\,|\mathbf{A}\mathbf{x}+\mathbf{B}|>R\right]\le C+\log R\;.
\]
This along with the strong Markov property proves the lemma since
we have 
\[
\mathbb{E}\log|\mathbf{X}_{\tau_{R}}(\mathbf{x}_{0})|\le\sup_{\mathbf{x}\in\mathcal{B}_{R}}\mathbb{E}\left[\log|\mathbf{A}\mathbf{x}+\mathbf{B}|\,\big|\,|\mathbf{A}\mathbf{x}+\mathbf{B}|>R\right]\;.
\]
\end{proof}

\subsection{\label{sec62}Lower bound }

We first handle a special univariate case which will be used in the
proof of the lower bound for the multivariate case. We start from
a specific univariate case. 
\begin{lem}
\label{lem:lbun}Let $d=1$. Suppose that $\mathbf{A}>0$, $\mathbf{B}\ge0$
almost surely, $\mathbb{P}(\mathbf{B}=0)<1$,
\[
\mathbb{E}|\log\mathbf{A}|<\infty\;,\;\;\;\text{and\;\;\;}\mathbb{E}|\log\mathbf{B}|<\infty\;.
\]
 Let $\mathbf{x}_{0}\ge0$ so that by the previous conditions we have
$\mathbf{X}_{n}(\mathbf{x}_{0})\ge0$ for all $n\in\mathbb{Z}^{+}$.
\begin{enumerate}
\item We almost surely have that 
\begin{equation}
\lim_{n\rightarrow\infty}\frac{1}{n}\log\mathbf{X}_{n}=\mathbb{E}\log\mathbf{A}\;.\label{eq:asyxp}
\end{equation}
\item We almost surely have that (cf. \eqref{eq:univ_lya_h})
\begin{equation}
\lim_{R\rightarrow\infty}\frac{\tau_{R}(\mathbf{x}_{0})}{\log R}=\frac{1}{\mathbb{E}\log\mathbf{A}}=\frac{1}{\gamma_{L}}\;.\label{eq:limexp}
\end{equation}
\end{enumerate}
\end{lem}

\begin{proof}
\noindent Fix $\mathbf{x}_{0}\ge0$, and for the convenience of the
notation, we write $\mathbf{X}_{n}=\mathbf{X}_{n}(\mathbf{x}_{0})$
when we do not need to specify $\mathbf{x}_{0}$. 

\smallskip{}

\noindent (1) Let us first consider the case $\mathbf{x}_{0}>0$.
Since we assumed that $\mathbf{B}\ge0$, we have 
\[
\mathbf{X}_{n}\ge\mathbf{A}_{n}\mathbf{X}_{n-1}\ge\cdots\ge\mathbf{A}_{n}\cdots\mathbf{A}_{1}\mathbf{x}_{0},
\]
and therefore we have
\begin{equation}
\liminf_{n\rightarrow\infty}\frac{1}{n}\log\mathbf{X}_{n}\ge\liminf_{n\rightarrow\infty}\left(\frac{1}{n}\log\mathbf{x}_{0}+\frac{1}{n}\sum_{i=1}^{n}\log\mathbf{A}_{i}\right)=\mathbb{E}\log\mathbf{A}\label{eq:lbexp}
\end{equation}
almost surely by the strong law of large numbers. Note here that,
for univariate model, $\mathbb{E}\log\mathbf{A}=\gamma_{L}>0$. 

For the other direction, we observe first that, for $n\in\mathbb{Z}^{+}$,
\begin{align*}
\log\mathbf{X}_{n} & =\log(\mathbf{A}_{n}\mathbf{X}_{n-1}+\mathbf{B}_{n})\\
 & =\log\mathbf{X}_{n-1}+\log\mathbf{A}_{n}+\log\left(1+\frac{\mathbf{B}_{n}}{\mathbf{A}_{n}\mathbf{X}_{n-1}}\right)\\
 & \le\log\mathbf{X}_{n-1}+\log\mathbf{A}_{n}+\frac{\mathbf{B}_{n}}{\mathbf{A}_{n}\mathbf{X}_{n-1}}\;,
\end{align*}
where the last line used the inequality $\log(1+x)\le x$. Hence,
we get 
\[
\frac{1}{n}\log\mathbf{X}_{n}\le\frac{1}{n}\log\mathbf{x}_{0}+\frac{1}{n}\sum_{i=1}^{n}\log\mathbf{A}_{i}+\frac{1}{n}\sum_{i=1}^{n}\frac{\mathbf{B}_{i}}{\mathbf{A}_{i}\mathbf{X}_{i-1}}\;.
\]
Therefore, by the same computation with \eqref{eq:lbexp}, we get
\begin{equation}
\limsup_{n\rightarrow\infty}\frac{1}{n}\log\mathbf{X}_{n}\le\mathbb{E}\log\mathbf{A}+\limsup_{n\rightarrow\infty}\frac{1}{n}\sum_{i=1}^{n}\frac{\mathbf{B}_{i}}{\mathbf{A}_{i}\mathbf{X}_{i-1}}\;.\label{eq:ubexp}
\end{equation}
Therefore, it suffices to prove that (as, $\mathbf{A},\,\mathbf{B},\,\mathbf{X}_{n}\ge0$
for all $n\in\mathbb{Z}^{+}$) 
\begin{equation}
\sum_{k=0}^{\infty}\frac{\mathbf{B}_{k+1}}{\mathbf{A}_{k+1}\mathbf{X}_{k}}<\infty\;\;\;\text{almost surely.}\label{eq:ssb}
\end{equation}
To that end, observe first that, for i.i.d random variables $U_{1},\,U_{2},\,\dots$
with $\mathbb{E}|U_{1}|<\infty$, we have 
\begin{equation}
\lim_{n\rightarrow\infty}\frac{U_{n}}{n}=0\text{\;\;\;}\text{almost surely},\label{eq:bdavb}
\end{equation}
as we have 
\[
\text{\ensuremath{\mathbb{P}\left[\limsup_{n\rightarrow\infty}\frac{\left|U_{n}\right|}{n}>0\right]\le\lim_{\delta\rightarrow0}\mathbb{P}\left[\limsup_{n\rightarrow\infty}\frac{\left|U_{n}\right|}{n}\ge\delta\right]}}=0
\]
where the last equality follows from the Markov inequality, the Borel-Cantelli
lemma and the layer-cake formula, as 
\[
\sum_{n=1}^{\infty}\mathbb{P}\left(\frac{|U_{n}|}{n}\ge\delta\right)=\sum_{n=1}^{\infty}\mathbb{P}\left(\frac{|U_{1}|}{\delta}\ge n\right)\le\frac{\mathbb{E}\left[|U_{1}|\right]}{\delta}<\infty\;.
\]
Then, we have 

\begin{align*}
\limsup_{n\rightarrow\infty}\frac{1}{n}\log\left(\frac{\mathbf{B}_{n+1}}{\mathbf{A}_{n+1}\mathbf{X}_{n}}\right) & \le\limsup_{n\rightarrow\infty}\frac{1}{n}\log\left(\frac{\mathbf{B}_{n+1}}{\mathbf{A}_{n+1}}\right)-\liminf_{n\rightarrow\infty}\frac{1}{n}\log\mathbf{X}_{n}\\
 & =-\liminf_{n\rightarrow\infty}\frac{1}{n}\log\mathbf{X}_{n}\le-\mathbb{E}\log\mathbf{A}<0
\end{align*}
where the second line follows from \eqref{eq:bdavb}, \eqref{eq:lbexp},
and \eqref{eq:univ_lya_h}. The last bound along with Cauchy's criterion
proves \eqref{eq:ssb} and we are done. 

Now let us consider the case $\mathbf{x}_{0}=0$. In view of coupling
in \eqref{eq:couple}, we have that
\begin{equation}
x\le x'\;\;\text{implies that}\;\;\mathbf{X}_{n}(x)\le\mathbf{X}_{n}(x')\;\;\text{holds for all }n\in\mathbb{Z}^{+}\;.\label{eq:mon}
\end{equation}
Therefore, we immediately have 
\begin{equation}
\limsup_{n\rightarrow\infty}\frac{1}{n}\log\mathbf{X}_{n}(0)\le\limsup_{n\rightarrow\infty}\frac{1}{n}\log\mathbf{X}_{n}(1)=\mathbb{E}\log\mathbf{A}\ ,\label{eq:ubexp2}
\end{equation}
where the last equality holds from the previous part. 

For the lower bound in the case $\mathbf{x}_{0}=0$, we need some
notation. 
\begin{itemize}
\item Write $\omega=(\mathbf{A}_{n},\,\mathbf{B}_{n})_{n\in\mathbb{Z}^{+}}$
and denote by $\Omega$ the sample space of $\omega$ which is a countable
product of the copy of sample space of $(\mathbf{A},\,\mathbf{B})$. 
\item Write 
\[
\zeta=\inf\left\{ n\ge1:\mathbf{B}_{n}>0\right\} =\inf\left\{ n\ge1:\mathbf{X}_{n}>0\right\} 
\]
so that $\zeta$ is a geometric random variable with success probability
$\mathbb{P}(\mathbf{B}\neq0)>0$, and hence $\mathbb{P}(\zeta<\infty)=1$. 
\item For $k,\,\ell\in\mathbb{Z}^{+}$, we define the event $\Omega_{k,\,\ell}$
as 
\[
\Omega_{k,\,\ell}=\left\{ \zeta=k,\,\mathbf{B}_{k}\in\bigg[\frac{1}{\ell},\,\frac{1}{\ell-1}\bigg)\right\} 
\]
where in the case $\ell=1$, the last line reads as $\Omega_{k,\,1}=\left\{ \zeta=k,\,\mathbf{B}_{k}\in[1,\,\infty)\right\} $.
Then, it is clear that $(\Omega_{k,\,\ell})_{k,\,\ell\in\mathbb{Z}^{+}}$
is a decomposition of the sample space $\Omega$. 
\item For $k\in\mathbb{Z}^{+}$ and $x>0$, we write $(\mathbf{X}_{n}^{(k)}(x))_{n\in\mathbb{Z}_{0}^{+}}$
the process \eqref{eq:AMP} defined by $\mathbf{X}_{0}^{(k)}(x)=x$
and 
\[
\mathbf{X}_{n+1}^{(k)}(x)=\mathbf{A}_{n+k+1}\mathbf{X}_{n}^{(k)}(x)+\mathbf{B}_{n+k+1}\;\;\;;\;n\in\mathbb{Z}^{+}\;.
\]
Note that $(\mathbf{X}_{n}^{(k)}(x))_{n\in\mathbb{Z}_{0}^{+}}$ has
the same law with $(\mathbf{X}_{n}(x))_{n\in\mathbb{Z}_{0}^{+}}$
and moreover is independent of $(\mathbf{A}_{n},\,\mathbf{B}_{n})_{n=1}^{k}$. 
\end{itemize}
Let us fix $k,\,\ell\in\mathbb{Z}^{+}$, and we claim now that, conditioning
on $\Omega_{k,\,\ell}$, we have
\begin{equation}
\liminf_{n\rightarrow\infty}\frac{1}{n}\log\mathbf{X}_{n}(0)\ge\mathbb{E}\log\mathbf{A}\label{eq:conas}
\end{equation}
almost surely. Proving this claim along with \eqref{eq:ubexp2} completes
the proof of part (1). 

To prove this claim, first observe that, conditioning on $\Omega_{k,\,\ell}$,
by \eqref{eq:mon}, we have 
\begin{equation}
\mathbf{X}_{n}(0)\ge\mathbf{X}_{n-k}^{(k)}\left(\frac{1}{\ell}\right)\;\;\;\;\text{for all }n\ge k\;.\label{eq:conas2}
\end{equation}
Furthermore, the event $\Omega_{k,\,\ell}$ which is completely determined
by $(\mathbf{A}_{n},\,\mathbf{B}_{n})_{n=1}^{k}$ is independent of
$(\mathbf{X}_{n}^{(k)}(x))_{n\in\mathbb{Z}_{0}^{+}}$, and therefore,
conditioning on $\Omega_{k,\,\ell}$ does not affect the evolution
of the process $(\mathbf{X}_{n}^{(k)}(x))_{n\in\mathbb{Z}_{0}^{+}}$
and therefore, conditioning on $\Omega_{k,\,\ell}$, we have
\begin{equation}
\liminf_{n\rightarrow\infty}\frac{1}{n}\log\mathbf{X}_{n}^{(k)}\left(\frac{1}{\ell}\right)=\mathbb{E}\log\mathbf{A}\label{eq:conas3}
\end{equation}
almost surely in $(\mathbf{A}_{n},\,\mathbf{B}_{n})_{n=k+1}^{\infty}$,
and hence in $(\mathbf{A}_{n},\,\mathbf{B}_{n})_{n\in\mathbb{Z}^{+}}$,
by the previous step. Summing up \eqref{eq:conas2} and \eqref{eq:conas3},
we have 
\[
\liminf_{n\rightarrow\infty}\frac{1}{n}\log\mathbf{X}_{n}(0)\ge\liminf_{n\rightarrow\infty}\frac{1}{n}\log\mathbf{X}_{n-k}^{(k)}\left(\frac{1}{\ell}\right)=\liminf_{n\rightarrow\infty}\frac{1}{n-k}\log\mathbf{X}_{n-k}^{(k)}\left(\frac{1}{\ell}\right)=\mathbb{E}\log\mathbf{A}
\]
almost surely. This completes the proof of \eqref{eq:conas}. \smallskip{}

\noindent (2) Fix an arbitrary $\epsilon\in(0,\,\gamma_{L})$. By
\eqref{eq:asyxp}, there exists $N=N(\epsilon)\in\mathbb{Z}^{+}$
such that, for all $n>N$,
\begin{equation}
\frac{1}{n}\log\mathbf{X}_{n}\in(\gamma_{L}-\epsilon,\,\gamma_{L}+\epsilon)\ .\label{eq:bdt}
\end{equation}
Let $R>\max\left\{ \mathbf{X}_{0},\,\mathbf{X}_{1},\,\dots,\,\mathbf{X}_{N},\,e^{N(\gamma_{L}+\epsilon)}\right\} $.
Then, for all $n\le\left\lfloor \frac{1}{\gamma_{L}+\epsilon}\log R\right\rfloor $
so that $e^{n(\gamma_{L}+\epsilon)}<R$, we have 
\[
\mathbf{X}_{n}\le\max\left\{ \mathbf{X}_{0},\,\mathbf{X}_{1},\,\dots,\,\mathbf{X}_{N},\,e^{n(\gamma_{L}+\epsilon)}\right\} <R
\]
and therefore $\tau_{R}(\mathbf{x}_{0})\ge\left\lfloor \frac{1}{\gamma_{L}+\epsilon}\log R\right\rfloor $.
On the other hand, for $n=\left\lceil \frac{1}{\gamma_{L}-\epsilon}\log R\right\rceil >N$
(where the last inequality holds since $R>e^{N(\gamma_{L}+\epsilon)}$),
by \eqref{eq:bdt} it holds that
\[
\mathbf{X}_{n}>e^{n(\gamma_{L}-\epsilon)}\ge R\ ,
\]
and therefore, $\tau_{R}(\mathbf{x}_{0})\le\left\lceil \frac{1}{\gamma_{L}-\epsilon}\log R\right\rceil $.
Summing up, we have 
\[
\left\lfloor \frac{1}{\gamma_{L}+\epsilon}\log R\right\rfloor \le\tau_{R}(\mathbf{x}_{0})\le\left\lceil \frac{1}{\gamma_{L}-\epsilon}\log R\right\rceil \;,
\]
and thus we get 

\[
\frac{1}{\gamma_{L}+\epsilon}\le\liminf_{R\rightarrow\infty}\frac{\tau_{R}(\mathbf{x}_{0})}{\log R}\le\limsup_{R\rightarrow\infty}\frac{\tau_{R}(\mathbf{x}_{0})}{\log R}\le\frac{1}{\gamma_{L}+\epsilon}\ .
\]
Letting $\epsilon\searrow0$ completes the proof. 
\end{proof}
Now we are ready to prove the lower bound for multivariate case. 
\begin{prop}
\label{prop:exp_lb}For all $\mathbf{x}_{0}\in\mathbb{R}^{d}$, we
almost surely have that 
\[
\liminf_{R\rightarrow\infty}\frac{\tau_{R}(\mathbf{x}_{0})}{\log R}\ge\frac{1}{\gamma_{L}}\;.
\]
\end{prop}

\begin{proof}
We fix $\mathbf{x}_{0}\in\mathbb{R}^{d}$. Then, we fix a positive
integer $N_{0}$. Then, by the submultiplicativeness of the matrix
norm, we have 
\[
\frac{1}{N_{0}}\mathbb{E}\log\Vert\mathbf{A}_{1}\cdots\mathbf{A}_{N_{0}}\Vert\ge\frac{1}{kN_{0}}\mathbb{E}\log\Vert\mathbf{A}_{1}\cdots\mathbf{A}_{kN_{0}}\Vert
\]
for all $k\ge1$. Hence, by letting $k\rightarrow\infty$ and then
recalling the definition of the Lyapunov exponent $\gamma_{L}>0$,
we get 
\[
\mathbb{E}\log\Vert\mathbf{A}_{1}\cdots\mathbf{A}_{N_{0}}\Vert\ge N_{0}\gamma_{L}\;.
\]

Observe that, for all $n\in\mathbb{Z}^{+}$, we can write
\[
\mathbf{X}_{n+N_{0}}(\mathbf{x}_{0})=\mathbf{A}_{n+N_{0}}\cdots\mathbf{A}_{n+1}\mathbf{X}_{n}(\mathbf{x}_{0})+(\mathbf{A}_{n+N_{0}}\cdots\mathbf{A}_{n+2}\mathbf{B}_{n+1}+\cdots+\mathbf{B}_{n+N_{0}})\;.
\]
and therefore, we have 
\begin{equation}
|\mathbf{X}_{n+N_{0}}(\mathbf{x}_{0})|\le\Vert\mathbf{A}_{n+N_{0}}\cdots\mathbf{A}_{n+1}\Vert\cdot|\mathbf{X}_{n}(\mathbf{x}_{0})|+|\mathbf{A}_{n+N_{0}}\cdots\mathbf{A}_{n+2}\mathbf{B}_{n+1}+\cdots+\mathbf{B}_{n+N_{0}}|.\label{eq:bdj}
\end{equation}
We fix $s\in\llbracket1,\,N_{0}\rrbracket$ and let, for $k\in\mathbb{Z}^{+}$,
\begin{align*}
a_{k} & =\Vert\mathbf{A}_{kN_{0}+s}\cdots\mathbf{A}_{(k-1)N_{0}+s+1}\Vert\\
b_{k} & =|\mathbf{A}_{kN_{0}+s}\cdots\mathbf{A}_{(k-1)N_{0}+s+2}\mathbf{B}_{(k-1)N_{0}+s+1}+\cdots+\mathbf{B}_{kN_{0}+s}|\;.
\end{align*}
We next consider an univariate process $(Y_{n}^{(s)})_{n\in\mathbb{Z}_{0}^{+}}$
defined by $Y_{0}^{(s)}=|\mathbf{X}_{s}(\mathbf{x}_{0})|$ and 
\[
Y_{k}^{(s)}=a_{k}Y_{k-1}^{(s)}+b_{k}\;\;\;\;;\;k\in\mathbb{Z}^{+}.
\]
Then, by \eqref{eq:bdj}, for all $k\in\mathbb{Z}^{+}$, we have 
\begin{equation}
|\mathbf{X}_{kN_{0}+s}(\mathbf{x}_{0})|\le|Y_{k}^{(s)}|\;.\label{eq:X,Y ineq}
\end{equation}
Define $\zeta_{R}^{(s)}=\inf\left\{ k:|Y_{k}^{(s)}|>R\right\} $ so
that by Lemma \ref{lem:lbun}, we almost surely have 
\begin{equation}
\lim_{R\rightarrow\infty}\frac{\zeta_{R}^{(s)}}{\log R}=\frac{1}{\mathbb{E}\log a_{1}}\;.\label{eq:uni_Y_exit}
\end{equation}
Note here that the condition $\mathbb{P}(b_{k}=0)<1$ of Lemma \ref{lem:lbun} is a consequence of Lemma \ref{lem:pre1-2}-(2) with $\mathbf{x}=\mathbf{0}$,
as $b_{k}$ has the same distribution of $\mathbf{X}_{N_{0}}(\mathbf{0})$. 

We next claim that 
\begin{equation}
\tau_{R}(\mathbf{x}_{0})\ge\min_{1\le s\le N_{0}}\left\{ N_{0}\zeta_{R}^{(s)}+s\right\} \;.\label{eq:rel_X,Y}
\end{equation}
This claim completes the proof since this and \eqref{eq:uni_Y_exit}
together implies that 
\[
\liminf_{R\rightarrow\infty}\frac{\tau_{R}(\mathbf{x}_{0})}{\log R}\ge\frac{N_{0}}{\mathbb{E}\log\Vert\mathbf{A}_{N_{0}}\cdots\mathbf{A}_{1}\Vert}\;,
\]
and therefore by letting $N_{0}\rightarrow\infty$ we can get the
desired bound. 

Now we turn to the proof of \eqref{eq:rel_X,Y}. We can find unique
integers $k_{0}\ge0$ and $1\le s_{0}\le N_{0}$ such that 
\[
\tau_{R}(\mathbf{x}_{0})=k_{0}N_{0}+s_{0}\;.
\]
Then, the comparison \eqref{eq:X,Y ineq} implies that $\zeta_{R}^{(s_{0})}\le k_{0}$
and therefore we get 
\[
\min_{1\le s\le N_{0}}\left\{ N_{0}\zeta_{R}^{(s_{0})}+s\right\} \le N_{0}\zeta_{R}^{(s_{0})}+s_{0}\le N_{0}k_{0}+s_{0}=\tau_{R}(\mathbf{x}_{0})
\]
as desired. 
\end{proof}

\subsection{\label{sec63}Upper bound}

We now turn to the upper bound. The purpose is to prove the following
bound.
\begin{prop}
\label{prop:exp_up}For all $\mathbf{x}_{0}\in\mathbb{R}^{d}$, there
exist constants $\kappa_{1},\,\kappa_{2}>0$ which depend only on
the distribution of $(\mathbf{A},\,\mathbf{B})$ such that 
\[
\mathbb{E}\left[\tau_{R}(\mathbf{x}_{0})\right]\le\kappa_{1}(1+\log R)-\kappa_{2}\log^{+}|\mathbf{x}_{0}|\;.
\]
\end{prop}

We note that the constants appeared in the previous proposition can
be computed explicitly in the course of proof. The proof of the upper
bound heavily relies on the construction of a submartingale. In this
construction, we heavily use the function 
\[
\log^{+}x=\log(x\lor1)=(\log x)\lor0\;\;\;\;;\;x\ge0\;,
\]
where we regard $\log^{+}0=0$. For this function, we have useful
elementary inequalities. 
\begin{lem}
\label{lem:log+}For all $x,\,y\ge0$, we have 
\begin{align*}
\log^{+}(x+y) & \le\log^{+}x+\log^{+}y+\log2\;\;\;\;\text{and}\\
\log^{+}xy & \le\log^{+}x+\log^{+}y\;.
\end{align*}
\end{lem}

\begin{proof}
We look at the first inequality. Thanks to the symmetry, it suffices
to consider following three cases separately. 
\begin{itemize}
\item $0\le x,\,y\le1$: the inequality is immediate from $x+y\le2$ and
that $\log^{+}(\cdot)$ is an increasing function.
\item $0\le x\le1\le y$: in this case, we have 
\[
\log^{+}(x+y)=\log(x+y)\le\log(2y)=\log y+\log2=\log^{+}x+\log^{+}y+\log2
\]
 
\item $x,\,y\ge1$: in this case, we have $\log^{+}=\log$ and hence it
suffices to check that $2xy\ge x+y$. This is immediate since $xy\ge x$
and $xy\ge y$. 
\end{itemize}
The proof of the second one is immediate from
\[
\log^{+}xy=(\log xy)\lor0=(\log x+\log y)\lor0\le(\log^{+}x+\log^{+}y)\lor0=\log^{+}x+\log^{+}y\ .
\]
\end{proof}
\noindent We from now on fix $\mathbf{x}_{0}$ and construct a sub-martingales
to complete the proof of Proposition \ref{prop:exp_up}. By lemma
\ref{Lem_expl_Lya}, we are able to take $N_{0}$ such that for all
$n\ge N_{0}$ 
\begin{equation}
\frac{1}{n}\mathbb{E}\log\Vert\Pi_{n}^{-1}\Vert^{-1}\ge\frac{2}{3}\gamma_{L}\;.\label{eq:pi_n_gamma_L}
\end{equation}

The first step is to construct a sequence of stopping times $(\sigma_{j})_{j=0}^{\infty}$
along which we construct the submartingale. The following lemma is
used in the construction. Recall that $(\mathcal{F}_{n})_{n\in\mathbb{Z}_{0}^{+}}$
is the filtration associated with the process $(\mathbf{X}_{n}(\cdot))_{n\in\mathbb{Z}_{0}^{+}}$.
\begin{lem}
\label{lem:stop}Let $\tau$ be a $(\mathcal{F}_{n})_{n\in\mathbb{Z}_{0}^{+}}$-stopping
time and $f:\mathbb{R}^{d}\rightarrow\mathbb{Z}^{+}$ be a Lebesgue
measurable function. Then, for all $\mathbf{x}_{0}\in\mathbb{R}^{d}$,
\[
\sigma=\tau+f(\mathbf{X}_{\tau}(\mathbf{x}_{0}))
\]
is also a $(\mathcal{F}_{n})_{n\in\mathbb{Z}_{0}^{+}}$-stopping time.
\end{lem}

\begin{proof}
Since 
\[
\{\sigma=n\}=\bigcup_{k,\,\ell\in\mathbb{Z}^{+}:k+\ell=n}\left(\{\tau=k\}\cap\left\{ f(\mathbf{X}_{k}(\mathbf{x}_{0}))=\ell\right\} \right)
\]
and since both $\{\tau=k\}$ and $\left\{ f(\mathbf{X}_{k}(\mathbf{x}_{0}))=\ell\right\} $
are $\mathcal{F}_{k}$-measurable set, $\{\sigma=n\}\in\mathcal{F}_{n}$.
\end{proof}
We let $\sigma_{0}=0$ and suppose that the sequence $\sigma_{0},\,\dots,\,\sigma_{j}$
has been constructed. Then, for all $n\in\mathbb{Z}^{+}$, by Lemma
\ref{lem:log+}, we have 

\begin{align*}
 & \mathbb{E}\left[\log^{+}|\mathbf{X}_{\sigma_{j}+n+1}(\mathbf{x}_{0})|+\log^{+}|\mathbf{X}_{\sigma_{j}+n}(\mathbf{x}_{0})|\ |\ \mathcal{F}_{\sigma_{j}}\right]\\
 & =\mathbb{E}\left[\log^{+}|\mathbf{X}_{n}^{\prime}(\mathbf{X}_{\sigma_{j}+1}(\mathbf{x}_{0}))|+\log^{+}|\mathbf{X}_{n}^{\prime}(\mathbf{X}_{\sigma_{j}}(\mathbf{x}_{0}))|\ |\ \mathcal{F}_{\sigma_{j}}\right]\\
 & =\mathbb{E}\left[\log^{+}|\mathbf{X}_{n}^{\prime}(\mathbf{X}_{\sigma_{j}+1}(\mathbf{x}_{0}))-\mathbf{X}_{n}^{\prime}(\mathbf{X}_{\sigma_{j}}(\mathbf{x}_{0}))|\ |\ \mathcal{F}_{\sigma_{j}}\right]-\log2\\
 & \ge\mathbb{E}\log\Vert\Pi_{n}^{-1}\Vert^{-1}+\mathbb{E}\left[\log|\mathbf{X}_{\sigma_{j}+1}(\mathbf{x}_{0})-\mathbf{X}_{\sigma_{j}}(\mathbf{x}_{0})|\;|\;\mathcal{F}_{\sigma_{j}}\right]-\log2
\end{align*}
where the last line follows from \eqref{eq:couple} and the fact that
$\log^{+}\ge\log$. Here, we note that $(\mathbf{X}_{n}'(\cdot))_{n\in\mathbb{Z}^{+}}$
denotes an process \eqref{eq:AMP} which has the same law with $(\mathbf{X}_{n}(\cdot))_{n\in\mathbb{Z}^{+}}$
but uses the random matrices/vectors $(\mathbf{A}_{n},\,\mathbf{B}_{n})_{n\in\mathbb{Z}^{+}}$
independent to the ones associated with $(\mathbf{X}_{n}(\cdot))_{n\in\mathbb{Z}^{+}}$.
Then, by \eqref{eq:pi_n_gamma_L}, \eqref{eq:ass_expl} of Assumption
\ref{Ass_Exp}, and the strong Markov property, there exists a constant
$C_{1}>0$ (which will be fixed throughout the current subsection)
such that, for all $n\ge N_{0}$, 
\begin{equation}
\mathbb{E}\left[\log^{+}|\mathbf{X}_{\sigma_{j}+n+1}(\mathbf{x}_{0})|+\log^{+}|\mathbf{X}_{\sigma_{j}+n}(\mathbf{x}_{0})|\ |\ \mathcal{F}_{\sigma_{j}}\right]\ge\frac{2n}{3}\gamma_{L}-C_{1}\;.\label{eq:bdd1}
\end{equation}
Therefore, we can define a random time $\sigma_{j+1}$ which is finite
almost surely by 
\begin{equation}
\sigma_{j+1}=\sigma_{j}+\inf\left\{ n\ge1:\mathbb{E}\left[\log^{+}|\mathbf{X}_{\sigma_{j}+n}(\mathbf{x}_{0})|\ |\ \mathcal{F}_{\sigma_{j}}\right]\ge\frac{n}{4}\gamma_{L}+\log^{+}|\mathbf{X}_{\sigma_{j}}(\mathbf{x}_{0})|\right\} \;.\label{eq:bdd2}
\end{equation}
Then, since we have assumed that $\sigma_{j}$ is a stopping time,
the random time $\sigma_{j+1}$ is again a stopping time by Lemma
\ref{lem:stop}. Note that this recursive construction guarantees
that $\sigma_{j}<\infty$ almost surely for all $j\in\mathbb{Z}^{+}$. 

Summing up, we constructed an increasing sequence of finite stopping
times $(\sigma_{j})_{j=0}^{\infty}$ with $\sigma_{0}=0$ that satisfies,
for all $j\ge0$, 
\begin{equation}
\mathbb{E}\left[\log^{+}|\mathbf{X}_{\sigma_{j+1}}(\mathbf{x}_{0})|\ |\ \mathcal{F}_{\sigma_{j}}\right]\ge\frac{\gamma_{L}}{4}(\sigma_{j+1}-\sigma_{j})+\log^{+}|\mathbf{X}_{\sigma_{j}}(\mathbf{x}_{0})|\;.\label{eq:subm}
\end{equation}

\begin{lem}
\label{lem:int}With the notation above, for all $k\in\mathbb{Z}^{+}$,
we have $\mathbb{E}[\sigma_{k}]<\infty$ and $\mathbb{E}[\log^{+}|\mathbf{X}_{\sigma_{k}}(\mathbf{x}_{0})|]<\infty$. 
\end{lem}

\begin{proof}
We proceed with induction. First, the statement is clear for $k=0$
as $\sigma_{k}=0$. 

Next we assume that the statement of the lemma holds for $k=j$ and
look at the case $k=j+1$. We first observe from \eqref{eq:bdd1}
that, with 
\[
n_{j}:=\left\lceil \frac{6}{\gamma_{L}}\left(2\log^{+}|\mathbf{X}_{\sigma_{j}}(\mathbf{x}_{0})|+C_{1}\right)\right\rceil \;,
\]
we have 
\begin{align*}
\mathbb{E}\left[\log^{+}|\mathbf{X}_{\sigma_{j}+n_{j}+1}(\mathbf{x}_{0})|+\log^{+}|\mathbf{X}_{\sigma_{j}+n_{j0}}(\mathbf{x}_{0})|\ |\ \mathcal{F}_{\sigma_{j}}\right] & \ge\frac{2n_{0}}{3}\gamma_{L}-C_{1}\;.\\
 & \ge\frac{n_{0}}{2}\gamma_{L}+2\log^{+}|\mathbf{X}_{\sigma_{j}}(\mathbf{x}_{0})|
\end{align*}
and therefore we can conclude from \eqref{eq:bdd2} that 
\begin{equation}
\sigma_{j+1}-\sigma_{j}\le n_{j}+1\le\frac{6}{\gamma_{L}}\left(2\log^{+}|\mathbf{X}_{\sigma_{j}}(\mathbf{x}_{0})|+C_{1}\right)+2\;.\label{eq:bdgap}
\end{equation}
Hence, by the induction hypothesis, we get 
\[
\mathbb{E}\sigma_{j+1}\le\mathbb{E}\sigma_{j}+\frac{12}{\gamma_{L}}\mathbb{E}\log^{+}|\mathbf{X}_{\sigma_{j}}(\mathbf{x}_{0})|+\frac{6C_{1}}{\gamma_{L}}+2<\infty\;.
\]

Next look at $\mathbb{E}\log^{+}|\mathbf{X}_{\sigma_{j+1}}(\mathbf{x}_{0})|$.
By Lemma \ref{lem:log+}, we have 
\begin{align*}
\mathbb{E}\left[\log^{+}|\mathbf{X}_{n+1}(\mathbf{x}_{0})|\ |\ \mathcal{F}_{n}\right] & \le\mathbb{E}\left[\log^{+}|\mathbf{A}_{n+1}\mathbf{X}_{n}(\mathbf{x}_{0})|+\log^{+}|\mathbf{B}_{n+1}|\ |\ \mathcal{F}_{n}\right]+\log2\\
 & \le\log^{+}|\mathbf{X}_{n}(\mathbf{x}_{0})|+\mathbb{E}\log^{+}\Vert\mathbf{A}\Vert+\mathbb{E}\log^{+}|\mathbf{B}|+\log2\;.
\end{align*}
Hence, by letting $C_{2}:=\mathbb{E}\log^{+}\Vert\mathbf{A}\Vert+\mathbb{E}\log^{+}|\mathbf{B}|+\log2>0$
and defining 
\[
Z_{n}=\log^{+}|\mathbf{X}_{n}(\mathbf{x}_{0})|-C_{2}n\;\;\;\;;\;n\in\mathbb{Z}^{+}\;,
\]
we can notice that $(Z_{n})_{n\in\mathbb{Z}_{0}^{+}}$ is a $(\mathcal{F}_{n})_{n\in\mathbb{Z}_{0}^{+}}$-supermartingale.
Thus, by the optimal stopping theorem, we have 
\begin{equation}
\mathbb{E}\log^{+}|\mathbf{X}_{n\land\sigma_{j+1}}(\mathbf{x}_{0})|\le C_{2}\mathbb{E}\left[n\land\sigma_{j+1}\right]+\log^{+}|\mathbf{x}_{0}|\label{eq:dc1}
\end{equation}
for all $n\in\mathbb{Z}^{+}$. Since $\sigma_{j}<\infty$ almost surely,
by Fatou's lemma (at the left-hand side) and dominated convergence
theorem (at the right-hand side, as we already verified that $\mathbb{E}(\sigma_{j+1})<\infty$),
we get 
\begin{align}
\mathbb{E}\log^{+}|\mathbf{X}_{\sigma_{j+1}}(\mathbf{x}_{0})| & \le C_{2}\mathbb{E}\sigma_{j+1}+\log^{+}|\mathbf{x}_{0}|<\infty\;.\label{eq:dc2}
\end{align}
This completes the induction step. 
\end{proof}
Define a new filtration $(\mathcal{G}_{j})_{j=0}^{\infty}$ by $\mathcal{G}_{j}=\mathcal{F}_{\sigma_{j}}$
and then we define, for $j\ge0$, 
\begin{equation}
M_{j}:=\log^{+}|\mathbf{X}_{\sigma_{j}}(\mathbf{x}_{0})|-\frac{\gamma_{L}}{4}\sigma_{j}\;.\label{eq:Mj}
\end{equation}
Note that by Lemma \ref{lem:int}, we have $\mathbb{E}|M_{j}|<\infty$
for all $j\in\mathbb{Z}^{+}$. Then, by \eqref{eq:subm}, the sequence
$(M_{j})_{j=0}^{\infty}$ is a $(\mathcal{G}_{j})_{j=0}^{\infty}$-submartingale.
Next we define 
\[
\theta=\theta_{R}:=\inf\left\{ j\ge0:\max\left\{ |\mathbf{X}_{0}(\mathbf{x}_{0})|,\,|\mathbf{X}_{1}(\mathbf{x}_{0})|,\,\dots,\,|\mathbf{X}_{\sigma_{j}}(\mathbf{x}_{0})|\right\} \ge R\right\} \;.
\]
Here, $\theta$ is a $(\mathcal{G}_{j})_{j=0}^{\infty}$-stopping
time since
\[
\left\{ \theta\le n\right\} =\bigcup_{k=1}^{\sigma_{n}}\left\{ \,|\mathbf{X}_{k}(\mathbf{x}_{0})|\ge R\,\right\} \;.
\]
Note that $\theta<\infty$ almost surely by Proposition \ref{Prop_expl_exit time finite}. 
\begin{lem}
\label{lem:extloc}With the notation above, there exists a constant
$C_{3}>0$ such that 
\[
\mathbb{E}\log^{+}|\mathbf{X}_{\sigma_{\theta}}(\mathbf{x}_{0})|\le C_{3}\left(1+\log R\right)\;.
\]
\end{lem}

\begin{proof}
For $0\le n\le\sigma_{\theta}$, by the same argument with the one
from \eqref{eq:dc1} to \eqref{eq:dc2}, we get 
\begin{align*}
\mathbb{E}\left[\log^{+}|\mathbf{X}_{\sigma_{\theta}}(\mathbf{x}_{0})|\ |\ \mathcal{F}_{n}\right] & \le C_{2}(\sigma_{\theta}-n)+\log^{+}|\mathbf{X}_{n}(\mathbf{x}_{0})|
\end{align*}
Therefore, we can deduce
\begin{align}
\mathbb{E}\log^{+}|\mathbf{X}_{\sigma_{\theta}}(\mathbf{x}_{0})| & =\sum_{n=0}^{\infty}\mathbb{E}\left[\log^{+}|\mathbf{X}_{\sigma_{\theta}}(\mathbf{x}_{0})|\cdot\mathbf{1}\left\{ \tau_{R}(\mathbf{x}_{0})=n\right\} \right]\nonumber \\
 & =\sum_{n=0}^{\infty}\mathbb{E}\left[\mathbb{E}\left[\log^{+}|\mathbf{X}_{\sigma_{\theta}}(\mathbf{x}_{0})|\ |\ \mathcal{F}_{n}\right]\cdot\mathbf{1}\left\{ \tau_{R}(\mathbf{x}_{0})=n\right\} \right]\nonumber \\
 & \le\sum_{n=0}^{\infty}\mathbb{E}\left[\left(C_{2}(\sigma_{\theta}-n)+\log^{+}|\mathbf{X}_{n}(\mathbf{x}_{0})|\right)\cdot\mathbf{1}\left\{ \tau_{R}(\mathbf{x}_{0})=n\right\} \right]\nonumber \\
 & =C_{2}\mathbb{E}\left[\sigma_{\theta}-\tau_{R}(\mathbf{x}_{0})\right]+\mathbb{E}\log^{+}|\mathbf{X}_{\tau_{R}}(\mathbf{x}_{0})|\;.\label{eq:bdlo1}
\end{align}
Since $\sigma_{\theta-1}<\tau_{R}(\mathbf{x}_{0})\le\sigma_{\theta}$
by definition, we obtain from \eqref{eq:bdgap} and the fact $|\mathbf{X}_{\sigma_{\theta-1}}(\mathbf{x}_{0})|\le R$
that 
\[
\sigma_{\theta}-\sigma_{\theta-1}\le\frac{6}{\gamma_{L}}\left(2\log^{+}|\mathbf{X}_{\sigma_{\theta-1}}(\mathbf{x}_{0})|+C_{1}\right)+2\le\frac{6}{\gamma_{L}}\left(2R+C_{1}\right)+2\ ,
\]
so we get
\[
\mathbb{E}\left[\sigma_{\theta}-\tau_{R}(\mathbf{x}_{0})\right]\le\mathbb{E}\left[\sigma_{\theta}-\sigma_{\theta-1}\right]\le\frac{6}{\gamma_{L}}\left(2R+C_{1}\right)+2\ .
\]
Injecting this and Lemma \ref{lem_expl_exit loc} to \ref{eq:bdlo1}
completes the proof. 
\end{proof}
Now we complete the proof of Proposition \ref{prop:exp_up}. 
\begin{proof}[Proof of Proposition \ref{prop:exp_up}]
Since $(M_{j})_{j=0}^{\infty}$ is a $(\mathcal{G}_{j})_{j=0}^{\infty}$-submartingale,
and $\theta$ is a $(\mathcal{G}_{j})_{j=0}^{\infty}$ stopping time,
by the optional stopping theorem, we get, for all $j\in\mathbb{Z}^{+}$,
\begin{equation}
\mathbb{E}M_{\theta\land j}\ge\mathbb{E}M_{0}=\log^{+}|\mathbf{x}|\;.\label{eq:ineq1}
\end{equation}
On the other hand, by definition \eqref{eq:Mj} of $M_{j}$, the fact
that $|\mathbf{X}_{\sigma_{\theta\land j}}(\mathbf{x}_{0})|\le|\mathbf{X}_{\sigma_{\theta}}(\mathbf{x}_{0})|\wedge R$,
and Lemma \ref{lem:extloc}, we obtain, for all $j\in\mathbb{Z}^{+}$,
\begin{align*}
\mathbb{E}M_{\theta\land j} & \le\mathbb{E}\log^{+}(|\mathbf{X}_{\sigma_{\theta}}(\mathbf{x}_{0})|\wedge R)-\frac{\gamma_{L}}{4}\mathbb{E}\left[\sigma_{\theta\land j}\right]\\
 & \le C(1+\log R)-\frac{\gamma_{L}}{4}\mathbb{E}\left[\sigma_{\theta\land j}\right]\;.
\end{align*}
Combining this bound with \eqref{eq:ineq1}, and then letting $j\rightarrow\infty$,
we get
\[
\mathbb{E}\left[\sigma_{\theta}\right]\le C(1+\log R)-C'\log^{+}|\mathbf{x}|
\]
Since $\tau_{R}(\mathbf{x}_{0})\le\sigma_{\theta}$, the proof is
completed. 
\end{proof}

\subsection{Proof of Theorem \ref{Thm_main_expl}}

We are now finally ready to prove Theorem \ref{Thm_main_expl}. 
\begin{proof}[Proof of Theorem \ref{Thm_main_expl}]
By Proposition \ref{prop:exp_lb} and Fatou's theorem, we get
\begin{equation}
\liminf_{R\rightarrow\infty}\frac{\mathbb{E}\left[\tau_{R}(\mathbf{x}_{0})\right]}{\log R}\ge\mathbb{E}\left[\liminf_{R\rightarrow\infty}\frac{\tau_{R}(\mathbf{x}_{0})}{\log R}\right]=\frac{1}{\gamma_{L}}\;.\label{eq:expl_lb}
\end{equation}
On the other hand, by Proposition \ref{prop:exp_up},
we have 
\begin{equation}
\limsup_{R\rightarrow\infty}\frac{\mathbb{E}\left[\tau_{R}(\mathbf{x}_{0})\right]}{\log R}\le\kappa_{1}\label{eq:expl_ub}
\end{equation}
where $\kappa_{1}$ is the constant appeared in the statement of Proposition
\ref{prop:exp_up}. Combining \eqref{eq:expl_lb} and \eqref{eq:expl_ub}
completes the proof. 
\end{proof}

\appendix

\section{Verifications of Assumptions for Example Models}

In this appendix, we confirm that the time series model (namely, the
ARCH and the GARCH models) introduced in Section \ref{sec32} satisfy
the assumptions of Section \ref{sec2} so that we can apply Theorems
\ref{Thm_main_cont} and \ref{Thm_main_expl} according to the sign
of the Lyapunov exponent $\gamma_{L}$. 

\subsection{Model review }

We recall from Section \ref{sec32} that ARCH($p$) and GARCH(1, $p$)
models can be viewed as the process $\mathbf{X}_{n}$ defined in Definition
\ref{def:AMP} with
\begin{equation}
\mathbf{A}_{n}=\begin{pmatrix}\alpha_{1}W_{n}^{2}+\beta_{1} & \ \alpha_{2}W_{n}^{2}+\beta_{2} & \cdots & \alpha_{p-1}W_{n}^{2}+\beta_{p-1} & \ \alpha_{p}W_{n}^{2}+\beta_{p}\\
1 & 0 & \cdots & 0 & 0\\
0 & 1 & \cdots & 0 & 0\\
\vdots & \vdots & \cdots & \vdots & \vdots\\
0 & 0 & \cdots & 1 & 0
\end{pmatrix}\;,\;\;\;\ \mathbf{B}_{n}=\begin{pmatrix}\alpha_{0}W_{n}^{2}+\beta_{0}\\
0\\
0\\
\vdots\\
0
\end{pmatrix}\;,\label{eq:tAB}
\end{equation}
where $(W_{n})_{n\ge0}$ are i.i.d. normally distributed random variables
(with mean $m$ and variance $\sigma^{2}$) and $\alpha_{i},\,\beta_{i}\ge0$
for all $0\le i\le p$. In particular, ARCH and GARCH models satisfy
the following assumptions. 
\begin{itemize}
\item ARCH($p$): $\beta_{0}=\beta_{1}=\cdots=\beta_{p}=0$ and $\alpha_{0},\,\alpha_{p}>0$. 
\item GARCH(1, $p$): $\alpha_{2}=\cdots=\alpha_{p}=0$, $\alpha_{0}=0$,
and $\alpha_{1},\,\beta_{0},\,\beta_{p}>0$. 
\end{itemize}
In addition, we assume $\alpha_{1}\neq0$ for the ARCH model for technical
reason, as without this condition the Assumption \ref{Ass_Cont5}
is invalid for e.g., ARCH(2) model. We note that this assumption is
satisfied in the usual real-world time series model. We focus this
section only on the ARCH model, as the proof for the GARCH model is
very similar to that for the ARCH model as they share the expression
\eqref{eq:tAB}. 
\begin{notation}
We denote in this appendix by $W$ the normal random variable with
mean $m$ and variance $\sigma^{2}$ (so that $W$ has the same distribution
with $W_{n}$, $n\ge0$, appeared above) and denote by $\mathbf{A}$
and $\mathbf{B}$ the matrix and vector obtained from $\mathbf{A}_{n}$
and $\mathbf{B}_{n}$ in \eqref{eq:tAB} by replacing $W_{n}$ with
$W$. Since the model depends only on the square of $W_{n}$, we can
assume without generality that $m\ge0$. 
\end{notation}

\subsection{Lyapunov exponent }

In the lemma below, we prove that the Lyapunov exponent $\gamma_{L}$
can be either positive or negative, so that both contractive and explosive
regimes are required to be analyzed. 
\begin{lem}
\label{lem:t_lya}For the ARCH$(p)$ model, we have that 
\[
\lim_{\alpha_{1}\rightarrow\infty}\gamma_{L}=\infty\;\;\;\;\text{and\;\;\;\;}\lim_{(\alpha_{1},\,\dots,\,\alpha_{p})\rightarrow(0,\,\dots,\,0)}\gamma_{L}=-\infty\;.
\]
\end{lem}

\begin{proof}
Since every elements of $\mathbf{A}$ are non-negative, the same holds
for $\mathbf{A}_{1}\cdots\mathbf{A}_{n}$ and hence we have 
\begin{equation}
\Vert\mathbf{A}_{1}\cdots\mathbf{A}_{n}\Vert\ge[\mathbf{A}_{1}\cdots\mathbf{A}_{n}]_{1,\,1}\ge[\mathbf{A}_{1}]_{1,\,1}\cdots[\mathbf{A}_{n}]_{1,\,1}\ge\alpha_{1}^{n}\prod_{t=1}^{n}W_{t}^{2}\;,\label{eq:mmnbd}
\end{equation}
where $[\mathbf{U}]_{i,\,j}$denotes the ($i,\,j$)-th component of
the matrix $\mathbf{U}$. Therefore, we immediately have 
\[
\gamma_{L}=\lim_{n\rightarrow\infty}\frac{1}{n}\mathbb{E}\log\Vert\mathbf{A}_{1}\cdots\mathbf{A}_{n}\Vert\ge\log\alpha_{1}+\log\mathbb{E}W^{2}\;.
\]
This proves the first assertion of the Lemma. 

For the second assertion, we write $\mathbf{A}_{n}=W_{n}^{2}\mathbf{D}+\mathbf{E}$
where
\begin{equation}
\mathbf{D}:=\begin{pmatrix}\alpha_{1} & \alpha_{2} & \cdots & \alpha_{p-1} & \alpha_{p}\\
0 & 0 & \cdots & 0 & 0\\
0 & 0 & \cdots & 0 & 0\\
\vdots & \vdots & \cdots & \vdots & \vdots\\
0 & 0 & \cdots & 0 & 0
\end{pmatrix},\ \mathbf{E}:=\begin{pmatrix}0 & 0 & \cdots & 0 & 0\\
1 & 0 & \cdots & 0 & 0\\
0 & 1 & \cdots & 0 & 0\\
\vdots & \vdots & \cdots & \vdots & \vdots\\
0 & 0 & \cdots & 1 & 0
\end{pmatrix}\ .\label{eq:mDE}
\end{equation}
Then, since $\mathbf{E}^{p}=\mathbf{O}$, we can apply the triangle
inequality and the submultiplicativity of the matrix norm at the right-hand
side (after expanding) of 
\[
\mathbf{A}_{1}\cdots\mathbf{A}_{p}=(W_{1}^{2}\mathbf{D}+\mathbf{E})\cdots(W_{p}^{2}\mathbf{D}+\mathbf{E})
\]
to deduce (since $\Vert\mathbf{E}\Vert=1$)
\[
\Vert\mathbf{A}_{1}\cdots\mathbf{A}_{p}\Vert\le(W_{1}^{2}\Vert\mathbf{D}\Vert+1)\cdots(W_{p}^{2}\Vert\mathbf{D}\Vert+1)-1\ .
\]
Since $\Vert\mathbf{D}\Vert\le\sum_{i=1}^{p}\alpha_{i}$ , we have
\[
\limsup_{(\alpha_{1},\,\dots,\,\alpha_{p})\rightarrow\mathbf{0}}\Vert\mathbf{A}_{1}\cdots\mathbf{A}_{p}\Vert=0\ .
\]
Then, we can apply Jensen's inequality and dominated convergence theorem
to get
\[
\limsup_{(\alpha_{1},\,\dots,\,\alpha_{p})\rightarrow\mathbf{0}}\mathbb{E}\log\Vert\mathbf{A}_{1}\cdots\mathbf{A}_{p}\Vert\le\limsup_{(\alpha_{1},\,\dots,\,\alpha_{p})\rightarrow\mathbf{0}}\log\mathbb{E}\Vert\mathbf{A}_{1}\cdots\mathbf{A}_{p}\Vert=-\infty\ .
\]
This proves the second assertion since we have the following upper
bound on $\gamma_{L}$ thanks to the submultiplicativity of the matrix
norm and independence of the matrices $\mathbf{A}_{n}$: 
\[
\gamma_{L}=\lim_{m\rightarrow\infty}\frac{1}{mp}\log\mathbb{E}\Vert\mathbf{A}_{1}\cdots\mathbf{A}_{mp}\Vert\le\frac{1}{p}\log\mathbb{E}\Vert\mathbf{A}_{1}\cdots\mathbf{A}_{p}\Vert\ .
\]
\end{proof}

\subsection{Assumptions for contractive regime}

We first assume that $\gamma_{L}<0$ so that we are in the contractive
regime. In this regime, we shall verify that Assumptions \ref{Ass_Cont1},
\ref{Ass_Cont2}, \ref{Ass_Cont3}, \ref{Ass_Cont4} and \ref{Ass_Cont5}
are valid for the ARCH model.

\subsubsection*{Verification of Assumption \ref{Ass_Cont1}}

Since we have assumed $\gamma_{L}<0$ in this subsection, it suffices
to verify (2), (3), and (4) of Assumption \ref{Ass_Cont1}. First,
(2) and (3) are direct from the fact that the entries of $\mathbf{A}$
and $\mathbf{B}$ are normal and therefore have exponentially small
tail. 

For (4), suppose that $\mathbf{x}=(x_{1},\,\dots,x_{p})\in\mathbb{R}^{p}$
satisfies $\mathbf{A}\mathbf{x}+\mathbf{B}=\mathbf{x}$ almost surely.
Then, by comparing the $2$nd to the $p$th coordinates of each side,
we get $x_{1}=x_{2}=\cdots=x_{p}$. Furtheremore, by comparing the
first coordinate, we get 
\[
\sum_{i=1}^{p}(\alpha_{i}W^{2}+\beta_{i})x_{1}+\alpha_{0}W^{2}+\beta_{0}=x_{1}
\]
almost surely. Thus, we have 
\[
\left(\sum_{i=1}^{p}\alpha_{i}\right)x_{1}+\alpha_{0}=0\;\;\;\;\;\text{and}\;\;\;\;\left(\sum_{i=1}^{p}\beta_{i}\right)x_{1}+\beta_{0}=x_{1}\;.
\]
For the ARCH model, this implies $x_{1}=0$ as we have $\beta_{0}=\beta_{1}=\cdots=\beta_{p}=0$.
Summing up, we must have $\mathbf{x}=\mathbf{0}$ and thus, from $\mathbf{A}\mathbf{x}+\mathbf{B}=\mathbf{x}$
we get $\mathbf{B}=\mathbf{0}$ almost surely. This is a contradiction
and therefore we verified all the requirements of Assumption \ref{Ass_Cont1}.

\subsubsection*{Verification of Assumption \ref{Ass_Cont2}}

It suffices to check $\lim_{s\rightarrow\infty}h_{\mathbf{A}}(s)=\infty$.
By \eqref{eq:mmnbd}, we have 
\[
\left(\mathbb{E}\Vert\mathbf{A}_{1}\cdots\mathbf{A}_{n}\Vert^{s}\right)^{1/n}\ge\alpha_{1}^{s}\mathbb{E}W^{2s}\ ,
\]
and therefore $h_{\mathbf{A}}(s)\ge\alpha_{1}^{s}\mathbb{E}W^{2s}$.
On the other hand, since $W\sim\mathcal{N}(m,\,\sigma^{2})$, it follows
from the well-known identity that 
\[
\mathbb{E}W^{2s}\ge\mathbb{E}\left|W-m\right|^{2s}=\frac{(2s)!}{2^{s}s!}\sigma^{2s}\ \ \ \ ;\ s\in\mathbb{Z}^{+}\ ,
\]
and hence we can conclude that 
\[
h_{\mathbf{A}}(s)\ge\alpha_{1}^{s}\mathbb{E}W^{2s}\ge\frac{(2s)!}{s!}\left(\frac{\alpha\sigma_{1}^{2}}{2}\right)^{s}\ge\left(\frac{\alpha\sigma_{1}^{2}}{2}s\right)^{s}\;.
\]
This completes the proof of $\lim_{s\rightarrow\infty}h_{\mathbf{A}}(s)=\infty$
and hence the confirmation of Assumption \ref{Ass_Cont2}.

\subsubsection*{Verification of Assumption \ref{Ass_Cont3}}

For Assumption \ref{Ass_Cont3}, the unboundedness of $\mathbf{B}$
(thanks to the unboundedness of the normal random variable) directly
implies that of the support of $\nu_{\infty}$. 

\subsubsection*{Verification of Assumption \ref{Ass_Cont4}}

We start from the following elementary fact. 
\begin{lem}
\label{lem:normal ineq}For all $s>0$, there exists $u_{0}=u_{0}(s)>0$
such that, for all $u_{1},\,u_{2}$ such that $u_{0}\le u_{1}\le u_{2}$,
we have 
\[
\frac{\mathbb{P}(W>u_{2})}{\mathbb{P}(W>u_{1})}\le2\left(\frac{u_{1}}{u_{2}}\right)^{s}\;.
\]
\end{lem}

\begin{proof}
First, we assume that $m=0,\,\sigma=1$ so that $W$ is a standard
normal random variable. Then, using basic inequality
\begin{equation}
\frac{a}{a^{2}+1}\cdot\frac{1}{\sqrt{2\pi}}e^{-\frac{a^{2}}{2}}\le\mathbb{P}(W>a)\le\frac{1}{a}\cdot\frac{1}{\sqrt{2\pi}}e^{-\frac{a^{2}}{2}}\ \ \ \ ;\ a>0\ ,\label{eq:normal approx}
\end{equation}
we get for any $u_{1},\,u_{2}>1$, 
\begin{align}
\frac{\mathbb{P}(W>u_{2})}{\mathbb{P}(W>u_{1})} & \le\frac{\frac{1}{u_{2}}e^{-u_{2}^{2}/2}}{\frac{u_{1}}{u_{1}^{2}+1}e^{-u_{1}^{2}/2}}\le\frac{\frac{1}{u_{2}}e^{-u_{2}^{2}/2}}{\frac{1}{2u_{1}}e^{-u_{1}^{2}/2}}=\left(\frac{2u_{1}}{u_{2}}\right)e^{-\frac{u_{2}^{2}-u_{1}^{2}}{2}}\ .\label{eq:wu1u2}
\end{align}
Observing that if $u_{1}\ge\sqrt{s}$, using elementary inequality
$x\le e^{x-1}$, we have that
\[
\left(\frac{u_{2}}{u_{1}}\right)^{s-1}\le\left(\frac{u_{2}}{u_{1}}\right)^{s}\le e^{s\cdot\frac{u_{2}-u_{1}}{u_{1}}}\le e^{u_{1}(u_{2}-u_{1})}\le e^{\frac{(u_{2}+u_{1})(u_{2}-u_{1})}{2}}\ ,
\]
hence combining with \eqref{eq:wu1u2} yields the desired conclusion
with $u_{0}=\max\left\{ 1,\,\sqrt{s}\right\} $.

For general $m,\,\sigma$, we can deduce from the previous conclusion
that, for all $u_{1},\,u_{2}$ such that 
\[
\max\left\{ m+\sigma,\,m\sqrt{s}+\sigma\right\} \le u_{1}\le u_{2}\;,
\]
we have that 
\begin{equation}
\frac{\mathbb{P}(W>u_{2})}{\mathbb{P}(W>u_{1})}\le2\left(\frac{(u_{1}-m)/\sigma}{(u_{2}-m)/\sigma}\right)^{s}\le2\left(\frac{u_{1}-m}{u_{2}-m}\right)^{s}\le2\left(\frac{u_{1}}{u_{2}}\right)^{s}\;,\label{eq:wu1u2m}
\end{equation}
where the last inequality holds as $m<u_{1}\le u_{2}$. 
\end{proof}
Now we are ready to check Assumption \ref{Ass_Cont4}. Let $\mathbf{x}=(x_{1},\,\dots,x_{p})\in\mathbb{R}^{p}$
with $|\mathbf{x}|\le R$ and set $x_{0}:=1$. Then, we can write
\begin{equation}
|\mathbf{A}\mathbf{x}+\mathbf{B}|^{2}=\left(\left(\sum_{k=0}^{p}\alpha_{k}x_{k}\right)W^{2}+\left(\sum_{k=0}^{p}\beta_{k}x_{k}\right)\right)^{2}+\sum_{k=1}^{p-1}x_{k}^{2}\ .\label{eq:ax+B}
\end{equation}
Therefore, Assumption \ref{Ass_Cont4} is immediate if $\sum_{k=0}^{p}\alpha_{k}x_{k}=0$.
Suppose from now on that $\sum_{k=0}^{p}\alpha_{k}x_{k}\neq0$. 

Then, since $|\mathbf{x}|\le R$, we get from \eqref{eq:ax+B} that,
for any $z,\,R>0$, 
\begin{align*}
\mathbb{P}(|\mathbf{A}\mathbf{x}+\mathbf{B}|>R) & \ge\mathbb{P}\left(\left|\left(\sum_{k=0}^{p}\alpha_{k}x_{k}\right)W^{2}+\left(\sum_{k=0}^{p}\beta_{k}x_{k}\right)\right|>R\right)\\
 & \ge\mathbb{P}\left(\left|\sum_{k=0}^{p}\alpha_{k}x_{k}\right|W^{2}>\left(1+\sum_{k=0}^{p}|\beta_{k}|\right)R\right)\;.
\end{align*}
On the other hand, again by \eqref{eq:ax+B} again, we get 
\begin{align*}
\mathbb{P}(|\mathbf{A}\mathbf{x}+\mathbf{B}|>zR) & \le\mathbb{P}\left(\left(\left(\sum_{k=0}^{p}\alpha_{k}x_{k}\right)W^{2}+\left(\sum_{k=0}^{p}\beta_{k}x_{k}\right)\right)^{2}+(p-1)R^{2}>z^{2}R^{2}\right)\\
 & \le\mathbb{P}\left(2\left(\sum_{k=0}^{p}\alpha_{k}x_{k}\right)^{2}W^{4}>(z^{2}-C_{1})R^{2}\right)
\end{align*}
where $C_{1}:=p-1+2\left(\sum_{k=0}^{p}|\beta_{k}|\right)^{2}$. With
these estimates and Lemma \ref{lem:normal ineq} with $s:=2(\alpha+1)$
where $\alpha$ is from Assumption \ref{Ass_Cont2}, we get, for any
$R>0$ and $z>\left\{ 2\left(1+\sum_{k=0}^{p}|\beta_{k}|\right)^{2}+C_{1}\right\} ^{1/2}$,
\begin{align*}
\frac{\mathbb{P}(|\mathbf{A}\mathbf{x}+\mathbf{B}|>zR)}{\mathbb{P}(|\mathbf{A}\mathbf{x}+\mathbf{B}|>R)} & \le\frac{C_{2}}{(z^{2}-C_{1})^{(\alpha+1)/2}}
\end{align*}
for some constant $C_{2}>0$ independent of $\mathbf{x},\,z$, and
$R$. Assumption \ref{Ass_Cont4} immediately follows from this bound. 

\subsubsection*{Verification of Assumption \ref{Ass_Cont5}}

We shall now prove that, for any $\mathbf{x}_{0},\,\mathbf{y}_{0}\in\mathbb{R}^{p}\setminus\left\{ \mathbf{0}\right\} $, 

\begin{equation}
\mathbb{P}(\mathbf{x}_{0}\cdot\mathbf{A}_{1}\cdots\mathbf{A}_{2p-1}\mathbf{y}_{0}=0)<1\ ,\label{eq:xAy-1}
\end{equation}
so that Assumption \ref{Ass_Cont5} holds. Suppose on the contrary
that we have 

\begin{equation}
\mathbb{P}(\mathbf{x}_{0}\cdot\mathbf{A}_{1}\cdots\mathbf{A}_{2p-1}\mathbf{y}_{0}=0)=1\ ,\label{eq:xAy}
\end{equation}
for some $\mathbf{x}_{0},\,\mathbf{y}_{0}\in\mathbb{R}^{p}\setminus\left\{ \mathbf{0}\right\} $. 

Recall the matrices $\mathbf{D}$ and $\mathbf{E}$ defined in \eqref{eq:mDE}
so that we can write $\mathbf{A}_{n}=W_{n}^{2}\mathbf{D}+\mathbf{E}.$
Thanks to this expression, if random vectors $\mathbf{X},\,\mathbf{Y}$
independent of $\mathbf{A}_{i}$ satisfy $\mathbf{X}\cdot\mathbf{A}_{i}\mathbf{Y}=0$
almost surely, then we can deduce that $\mathbf{X}\cdot\mathbf{D}\mathbf{Y}=\mathbf{X}\cdot\mathbf{E}\mathbf{Y}=0$
almost surely as well. This observation enable us to deduce from \eqref{eq:xAy}
that
\begin{equation}
\mathbf{x}_{0}\cdot\mathbf{M}_{1}\cdots\mathbf{M}_{2p-1}\mathbf{y}_{0}=0,\quad\forall\mathbf{M}_{i}\in\left\{ \mathbf{D},\,\mathbf{E}\right\} ,\quad\forall i\in\llbracket1,\,2p-1\rrbracket\ .\label{eq:xMy}
\end{equation}
Since $\mathbf{x}_{0},\,\mathbf{y}_{0}$ are nonzero vectors, there
exist indices $r,\,s\in\llbracket1,\,p\rrbracket$ such that
\begin{equation}
\mathbf{e}_{r}\cdot\mathbf{x}_{0}\neq0,\ \mathbf{e}_{s}\cdot\mathbf{y}_{0}\neq0\ ,\label{eq:eres}
\end{equation}
where $\left\{ \mathbf{e}_{i}:i\in\llbracket1,\,p\rrbracket\right\} $
is the standard orthonormal basis of $\mathbb{R}^{p}$. Then, observing
that the first coordinate of
\[
\mathbf{x}_{0}\cdot\mathbf{E}^{r-1}\mathbf{D}^{2p-r-s+1}\mathbf{E}^{s-1}\mathbf{y}_{0}
\]
equals to $\alpha_{1}^{2p-r-s}(\mathbf{e}_{r}\cdot\mathbf{x}_{0})(\mathbf{e}_{s}\cdot\mathbf{y}_{0})$
which is nonzero by our assumpition that $\alpha_{1}\neq0$ and \eqref{eq:eres}.
This contradicts to \eqref{eq:xMy} and we are done. 

\subsection{Assumptions for explosive regime}

In this section, we assume that $\gamma_{L}>0$ and verify Assumption
\ref{Ass_Exp}. Thanks to (4) of Remark \ref{rem:Ass_Exp}, we only
need to verify (2) of the assumption. To prove (2), we need the following
elementary lemma. 
\begin{lem}
\label{lem: norm log}For a normal random variable $W$, for all $\delta>0$,
we have 
\[
\inf_{0<\epsilon\le1,\,y>\delta}\mathbb{E}\log|\epsilon W^{2}-y|>-\infty\ \;\;\;\text{and}\;\;\;\;\inf_{y\in\mathbb{R}}\mathbb{E}\log|W^{2}-y|>-\infty\;.
\]
\end{lem}

\begin{proof}
Recall that $W\sim N\left(m,\,\sigma^{2}\right)$. We
shall assume for simplicity that $m=0$ and $\sigma=1$, since the
proof for the general case is essentially same (with a slightly more
involved computation). Also, it suffices to show only for $\delta<1$
and $y>2\delta$ for the first inequality since the infimum is taken
on the set $\left\{ 0<\epsilon\le1,\,y>\delta\right\} $ which is
decreasing respect to $\delta$.

For the first bound, fix $0<\epsilon\le1,\,y>2\delta$. Then, since
$y>2\delta$, we have 
\begin{align*}
\mathbb{E}\log|\epsilon W^{2}-y| & =\mathbb{E}\log|\epsilon W^{2}-y|\mathbf{1}\left\{ |\epsilon W^{2}-y|\ge\delta\right\} +\mathbb{E}\log|\epsilon W^{2}-y|\mathbf{1}\left\{ |\epsilon W^{2}-y|<\delta\right\} \\
 & \ge\log\delta+\mathbb{E}\left(\log|\epsilon W^{2}-y|\mathbf{1}\left\{ y-\delta<\epsilon W^{2}<y+\delta\right\} \right)\ .
\end{align*}
The last expectation can be bounded from below by 
\begin{align}
 & \mathbb{E}\log|\epsilon W^{2}-y|\mathbf{1}\left\{ y-\delta<\epsilon W^{2}<y+\delta\right\} \nonumber \\
 & =2\int_{\sqrt{(y-\delta)/\epsilon}}^{\sqrt{(y+\delta)/\epsilon}}\frac{1}{\sqrt{2\pi}}e^{-\frac{w^{2}}{2}}\log|\epsilon w^{2}-y|dw\nonumber \\
 & \ge\sqrt{\frac{2}{\pi}}\exp\left(-\frac{y-\delta}{2\epsilon}\right)\int_{\sqrt{(y-\delta)/\epsilon}}^{\sqrt{(y+\delta)/\epsilon}}\log|\epsilon w^{2}-y|dw\nonumber \\
 & =\sqrt{\frac{2}{\pi}}\exp\left(-\frac{y-\delta}{2\epsilon}\right)\int_{\sqrt{(y-\delta)/y}}^{\sqrt{(y+\delta)/y}}\sqrt{\frac{y}{\epsilon}}\left(\log y+\log|w^{2}-1|\right)dw\nonumber \\
 & \ge\sqrt{\frac{2}{\pi}}\exp\left(-\frac{y-\delta}{2\epsilon}\right)\left(\frac{\sqrt{y+\delta}-\sqrt{y-\delta}}{\sqrt{\epsilon}}\log y+\int_{\sqrt{(y-\delta)/y}}^{\sqrt{(y+\delta)/y}}\log\left|w^{2}-1\right|dw\right)\ .\label{eq:epy}
\end{align}
Since
\[
\int_{\sqrt{(y-\delta)/y}}^{\sqrt{(y+\delta)/y}}\log\left|w^{2}-1\right|dw\ge\int_{-\sqrt{2}}^{\sqrt{2}}\log\left|w^{2}-1\right|dw>-\infty
\]
considering the range of $w$ where $\log|w^{2}-1|$ is negative,
and the last line of \eqref{eq:epy} is a continuous function of $(\epsilon,\,y)$
in the region $(0,\,1]\times(2\delta,\,\infty)$ which converges to
$0$ if $\epsilon\rightarrow0$ or $y\rightarrow\infty$, the desired
infimum is a finite value.

For the second bound of the lemma, since 
\[
\inf_{|y|>1}\mathbb{E}\log|W^{2}-y|>-\infty
\]
by the first part, and since $\log|W^{2}-y|\ge\log|W^{2}|$ if $y<0$,
it suffices to show that 
\[
\inf_{0\le y\le1}\mathbb{E}\left[\log|W^{2}-y|\mathbf{1}\left\{ |W^{2}-y|<1\right\} \right]>-\infty\;.
\]
Since 
\begin{align*}
\int_{|w^{2}-y|<1}\log|w^{2}-y|dw & =\int_{-\sqrt{y+1}}^{\sqrt{y+1}}(\log|w-\sqrt{y}|+\log|w+\sqrt{y}|)dw\\
 & =2\int_{-\sqrt{y+1}}^{\sqrt{y+1}}\log|w+\sqrt{y}|dw\ge2\int_{-1}^{1}\log|w|dw>-\infty\ ,
\end{align*}
and the probability density function of $W$ has a maximum value $1/\sqrt{2\pi}$, we get {[}note that $\log|w^{2}-y|<0$ on the interval of integration{]}
\[
\inf_{0\le y\le1}\mathbb{E}\log|W^{2}-y|\ge\frac{2}{\sqrt{2\pi}}\int_{-1}^{1}\log|w|dw>-\infty\ ,
\]
hence the proof is completed. 
\end{proof}
Now, we are ready to verify (2) of Assumption \ref{Ass_Exp}. Inspired
by the previous lemma, let us write 
\[
C_{W}:=\inf_{y\in\mathbb{R}}\mathbb{E}\log|W^{2}-y|>-\infty\ ,
\]
Fix $\mathbf{x}=(x_{1},\,\dots,x_{p})\in\mathbb{R}^{p}$. If $\mathbf{x}$
satisfies $\sum_{i=1}^{p}\alpha_{i}x_{i}+\alpha_{0}\neq0$, it holds
that
\begin{align*}
\mathbb{E}\log|\mathbf{A}\mathbf{x}+\mathbf{B}-\mathbf{x}| & =\frac{1}{2}\mathbb{E}\log\left(\left(\left(\sum_{i=1}^{p}\alpha_{i}x_{i}+\alpha_{0}\right)W^{2}-x_{1}\right)^{2}+\sum_{i=1}^{p-1}(x_{i}-x_{i+1})^{2}\right)\\
 & \ge\max\left\{ \mathbb{E}\log\left|\left(\sum_{i=1}^{p}\alpha_{i}x_{i}+\alpha_{0}\right)W^{2}-x_{1}\right|,\,\frac{1}{2}\log\left(\sum_{i=1}^{p-1}(x_{i}-x_{i+1})^{2}\right)\right\} \\
 & \ge\max\left\{ \log\left|\sum_{i=1}^{p}\alpha_{i}x_{i}+\alpha_{0}\right|+C_{W},\,\frac{1}{2}\log\left(\sum_{i=1}^{p-1}(x_{i}-x_{i+1})^{2}\right)\right\} \ ,
\end{align*}
defining $\log0:=-\infty$. If $\sum_{i=1}^{p}\alpha_{i}x_{i}+\alpha_{0}=0$,
it holds that
\[
\mathbb{E}\log|\mathbf{A}\mathbf{x}+\mathbf{B}-\mathbf{x}|\ge\inf_{\sum_{i=1}^{p}\alpha_{i}x_{i}+\alpha_{0}=0}\frac{1}{2}\mathbb{E}\log\left(x_{1}^{2}+\sum_{i=1}^{p-1}(x_{i}-x_{i+1})^{2}\right)>-\infty\ ,
\]
since $x_{1}^{2}+\sum_{i=1}^{p-1}(x_{i}-x_{i+1})^{2}\rightarrow0$
is equivalent to $\mathbf{x}\rightarrow\mathbf{0}$. Hence, if some
sequence of vectors $(\mathbf{x}_{n})_{n\in\mathbb{N}}$ satisfies
\begin{equation}
\lim_{n\rightarrow\infty}\mathbb{E}\log|\mathbf{A}\mathbf{x}_{n}+\mathbf{B}-\mathbf{x}_{n}|=-\infty\label{eq:cinn}
\end{equation}
then we must have 
\[
\lim_{n\rightarrow\infty}\mathbf{x}_{n}=-(\kappa,\,\kappa,\,\dots,\,\kappa)\ ,
\]
where $\kappa:=\alpha_{0}/(\alpha_{1}+\cdots+\alpha_{p})$ by the
observations above. 

On the other hand, if $\mathbf{x}$ is close enough to $-(\kappa,\,\kappa,\,\dots,\,\kappa)$
so that
\[
\left|\sum_{i=1}^{p}\alpha_{i}x_{i}+\alpha_{0}\right|<1\text{ and }x_{1}<-\frac{\kappa}{2}\ ,
\]
it holds that
\begin{align*}
\mathbb{E}\log|\mathbf{A}\mathbf{x}+\mathbf{B}-\mathbf{x}| & \ge\mathbb{E}\log\left|\left(\sum_{i=1}^{p}\alpha_{i}x_{i}+\alpha_{0}\right)W^{2}-x_{1}\right|\\
 & \ge\inf_{|\epsilon|<1,\,|y|>\kappa/2}\mathbb{E}\log|\epsilon W^{2}-y|\ .
\end{align*}
This contradicts with \eqref{eq:cinn}, and (2) of Assumption \ref{Ass_Exp}
is verified.

\bibliographystyle{abbrv}

\begin{thebibliography}{10}

\bibitem{arora2022understanding}
S.~Arora, Z.~Li, and A.~Panigrahi.
\newblock Understanding gradient descent on the edge of stability in deep learning.
\newblock In {\em International Conference on Machine Learning}, pages 948--1024. PMLR, 2022.

\bibitem{barnsley2014fractals}
M.~F. Barnsley.
\newblock {\em Fractals everywhere}.
\newblock Academic press, 2014.

\bibitem{basrak2002regular}
B.~Basrak, R.~A. Davis, and T.~Mikosch.
\newblock Regular variation of garch processes.
\newblock {\em Stochastic processes and their applications}, 99(1):95--115, 2002.

\bibitem{bovier2016metastability}
A.~Bovier and F.~Den~Hollander.
\newblock {\em Metastability: a potential-theoretic approach}, volume 351.
\newblock Springer, 2016.

\bibitem{bovier2001metastability}
A.~Bovier, M.~Eckhoff, V.~Gayrard, and M.~Klein.
\newblock Metastability in stochastic dynamics of disordered mean-field models.
\newblock {\em Probability Theory and Related Fields}, 119:99--161, 2001.

\bibitem{bovier2004metastability}
A.~Bovier, M.~Eckhoff, V.~Gayrard, and M.~Klein.
\newblock Metastability in reversible diffusion processes. i. sharp asymptotics for capacities and exit times.
\newblock {\em J. Eur. Math. Soc.(JEMS)}, 6(4):399--424, 2004.

\bibitem{bovier2005metastability}
A.~Bovier, M.~Eckhoff, V.~Gayrard, and M.~Klein.
\newblock Metastability in reversible diffusion processes ii: Precise asymptotics for small eigenvalues.
\newblock {\em Journal of the European Mathematical Society}, 7(1):69--99, 2005.

\bibitem{buraczewski2016stochastic}
D.~Buraczewski, E.~Damek, T.~Mikosch, et~al.
\newblock {\em Stochastic models with power-law tails: The equation $X= AX+ B$}.
\newblock Springer, 2016.

\bibitem{buraczewski2012asymptotics}
D.~Buraczewski, E.~Damek, and M.~Mirek.
\newblock Asymptotics of stationary solutions of multivariate stochastic recursions with heavy tailed inputs and related limit theorems.
\newblock {\em Stochastic Processes and their Applications}, 122(1):42--67, 2012.

\bibitem{cassandro1984metastable}
M.~Cassandro, A.~Galves, E.~Olivieri, and M.~E. Vares.
\newblock Metastable behavior of stochastic dynamics: a pathwise approach.
\newblock {\em Journal of statistical physics}, 35:603--634, 1984.

\bibitem{cohen2021gradient}
J.~M. Cohen, S.~Kaur, Y.~Li, J.~Z. Kolter, and A.~Talwalkar.
\newblock Gradient descent on neural networks typically occurs at the edge of stability.
\newblock In {\em International Conference on Learning Representations}, 2021.

\bibitem{collamore2018large}
J.~F. Collamore and S.~Mentemeier.
\newblock Large excursions and conditioned laws for recursive sequences generated by random matrices.
\newblock {\em The Annals of Probability}, 46(4):2064--2120, 2018.

\bibitem{damek2024analysing}
E.~Damek and S.~Mentemeier.
\newblock Analysing heavy-tail properties of stochastic gradient descent by means of stochastic recurrence equations.
\newblock {\em arXiv preprint arXiv:2403.13868}, 2024.

\bibitem{damian2022self}
A.~Damian, E.~Nichani, and J.~D. Lee.
\newblock Self-stabilization: The implicit bias of gradient descent at the edge of stability.
\newblock {\em arXiv preprint arXiv:2209.15594}, 2022.

\bibitem{eyring1935chemical}
H.~Eyring.
\newblock The activated complex in chemical reactions.
\newblock {\em The Journal of Chemical Physics}, 3(2):107--115, 1935.

\bibitem{freidlin1970onsmall}
M.~I. Freidlin and A.~D. Wentzell.
\newblock On small random perturbations of dynamical systems.
\newblock {\em Russian Mathematical Surveys}, 25(1):1, Feb 1970.

\bibitem{freidlin1973some}
M.~I. Freidlin and A.~D. Wentzell.
\newblock Some problems concerning stability under small random perturbations.
\newblock {\em Theory of Probability \& Its Applications}, 17(2):269--283, 1973.

\bibitem{freidlin1984random}
M.~I. Freidlin and A.~D. Wentzell.
\newblock {\em Random Perturbations of Dynamical Systems}.
\newblock Springer, New York, NY, 1984.

\bibitem{furstenberg1960products}
H.~Furstenberg and H.~Kesten.
\newblock Products of random matrices.
\newblock {\em The Annals of Mathematical Statistics}, 31(2):457--469, 1960.

\bibitem{gaudilliere2014dirichlet}
A.~Gaudilliere and C.~Landim.
\newblock A dirichlet principle for non reversible markov chains and some recurrence theorems.
\newblock {\em Probability Theory and Related Fields}, 158:55--89, 2014.

\bibitem{glasstone1941theory}
S.~Glasstone, K.~J. Laidler, and H.~Eyring.
\newblock {\em The theory of rate processes: the kinetics of chemical reactions, viscosity, diffusion and electrochemical phenomena}.
\newblock McGraw-Hill, New York, 1941.

\bibitem{goldie1991}
C.~M. Goldie.
\newblock Implicit renewal theory and tails of solutions of random equations.
\newblock {\em The Annals of Applied Probability}, 1(1):126--166, 1991.

\bibitem{goldie2000stability}
C.~M. Goldie and R.~A. Maller.
\newblock Stability of perpetuities.
\newblock {\em The Annals of Probability}, 28(3):1195--1218, 2000.

\bibitem{gurbuzbalaban2021heavy}
M.~Gurbuzbalaban, U.~Simsekli, and L.~Zhu.
\newblock The heavy-tail phenomenon in {SGD}.
\newblock In {\em International Conference on Machine Learning}, pages 3964--3975. PMLR, 2021.

\bibitem{hodgkinson2021multiplicative}
L.~Hodgkinson and M.~Mahoney.
\newblock Multiplicative noise and heavy tails in stochastic optimization.
\newblock In {\em International Conference on Machine Learning}, pages 4262--4274. PMLR, 2021.

\bibitem{hutchinson1981fractals}
J.~E. Hutchinson.
\newblock Fractals and self similarity.
\newblock {\em Indiana University Mathematics Journal}, 30(5):713--747, 1981.

\bibitem{imkeller2006first}
P.~Imkeller and I.~Pavlyukevich.
\newblock First exit times of sdes driven by stable {L}\'evy processes.
\newblock {\em Stochastic Processes and their Applications}, 116(4):611--642, 2006.

\bibitem{imkeller2008levy}
P.~Imkeller and I.~Pavlyukevich.
\newblock L\'evy flights: transitions and meta-stability.
\newblock {\em Journal of Physics A: Mathematical and General}, 39(15):L237, 2008.

\bibitem{pavlyukevich2008metastable}
P.~Imkeller and I.~Pavlyukevich.
\newblock Metastable behaviour of small noise {L}\'evy-driven diffusions.
\newblock {\em ESAIM: PS}, 12:412--437, 2008.

\bibitem{imkeller2010first}
P.~Imkeller, I.~Pavlyukevich, and M.~Stauch.
\newblock First exit times of non-linear dynamical systems in {$\mathbb R^d$} perturbed by multifractal {L}{\'e}vy noise.
\newblock {\em Journal of Statistical Physics}, 141(1):94--119, 2010.

\bibitem{jastrzkebskirelation}
S.~Jastrzebski, Z.~Kenton, N.~Ballas, A.~Fischer, Y.~Bengio, and A.~Storkey.
\newblock On the relation between the sharpest directions of {DNN} loss and the {SGD} step length.
\newblock In {\em International Conference on Learning Representations}.

\bibitem{jastrzebskibreak}
S.~Jastrzebski, M.~Szymczak, S.~Fort, D.~Arpit, J.~Tabor, K.~Cho, and K.~Geras.
\newblock The break-even point on optimization trajectories of deep neural networks.
\newblock In {\em International Conference on Learning Representations}.

\bibitem{kesten1973random}
H.~Kesten.
\newblock Random difference equations and renewal theory for products of random matrices.
\newblock {\em Acta Mathematica}, 131(1):207--248, 1973.

\bibitem{konstantinides2005large}
D.~G. Konstantinides and T.~Mikosch.
\newblock Large deviations and ruin probabilities for solutions to stochastic recurrence equations with heavy-tailed innovations.
\newblock {\em Annals of probability}, pages 1992--2035, 2005.

\bibitem{kramers1940brownian}
H.~A. Kramers.
\newblock Brownian motion in a field of force and the diffusion model of chemical reactions.
\newblock {\em Physica}, 7(4):284--304, 1940.

\bibitem{landim2014metastability}
C.~Landim.
\newblock Metastability for a non-reversible dynamics: the evolution of the condensate in totally asymmetric zero range processes.
\newblock {\em Communications in Mathematical Physics}, 330:1--32, 2014.

\bibitem{lee2022non}
J.~Lee and I.~Seo.
\newblock Non-reversible metastable diffusions with gibbs invariant measure i: Eyring--kramers formula.
\newblock {\em Probability Theory and Related Fields}, 182(3):849--903, 2022.

\bibitem{liu2000generalized}
Q.~Liu.
\newblock On generalized multiplicative cascades.
\newblock {\em Stochastic processes and their applications}, 86(2):263--286, 2000.

\bibitem{nguyen2019first}
T.~H. Nguyen, U.~Simsekli, M.~Gurbuzbalaban, and G.~Richard.
\newblock First exit time analysis of stochastic gradient descent under heavy-tailed gradient noise.
\newblock In H.~Wallach, H.~Larochelle, A.~Beygelzimer, F.~d\textquotesingle Alch\'{e}-Buc, E.~Fox, and R.~Garnett, editors, {\em Advances in Neural Information Processing Systems}, volume~32. Curran Associates, Inc., 2019.

\bibitem{Olivieri_Vares_2005}
E.~Olivieri and M.~E. Vares.
\newblock {\em Large Deviations and Metastability}.
\newblock Encyclopedia of Mathematics and its Applications. Cambridge University Press, 2005.

\bibitem{rudin1987real}
W.~Rudin.
\newblock {\em Real and complex analysis}.
\newblock Springer, 2016.

\bibitem{slowik2012note}
M.~Slowik.
\newblock A note on variational representations of capacities for reversible and nonreversible markov chains.
\newblock {\em Unpublished, Technische Universit{\"a}t Berlin}, 2012.

\bibitem{srikant2019finite}
R.~Srikant and L.~Ying.
\newblock Finite-time error bounds for linear stochastic approximation andtd learning.
\newblock In {\em Conference on Learning Theory}, pages 2803--2830. PMLR, 2019.

\bibitem{viana2014lectures}
M.~Viana.
\newblock {\em Lectures on Lyapunov exponents}, volume 145.
\newblock Cambridge University Press, 2014.

\bibitem{wang2022eliminating}
X.~Wang, S.~Oh, and C.-H. Rhee.
\newblock Eliminating sharp minima from {SGD} with truncated heavy-tailed noise.
\newblock In {\em International Conference on Learning Representations}, 2022.

\bibitem{wang2023large}
X.~Wang and C.-H. Rhee.
\newblock Large deviations and metastability analysis for heavy-tailed dynamical systems.
\newblock {\em arXiv preprint arXiv:2307.03479}, 2023.

\end{thebibliography}

\ifshowtheoremlist
\newpage
\linkdest{location of theorem list}
\renewcommand{\listtheoremname}{List of Definitions and Notations}
\listoftheorems[numwidth=2.5em, ignoreall, show={defn,notation}]
\renewcommand{\listtheoremname}{List of Theorems}
\listoftheorems[numwidth=2.5em, ignoreall, show={thm,lem,prop}]
\renewcommand{\listtheoremname}{List of Assumptions}
\listoftheorems[numwidth=2.5em, ignoreall, show={assumption}]
\renewcommand{\listtheoremname}{List of Remarks}
\listoftheorems[numwidth=2.5em, ignoreall, show={rem}]
\normalsize
\fi

\ifshownotationindex
\newpage
\linkdest{location of notation index}
\notationindex
\fi

\ifshowequationlist
\newpage
\section*{List of Numbered Equations}
\linkdest{location of equation number list}
\fi

\ifshowtoc
\newpage
\setcounter{tocdepth}{2}
\linkdest{location of table of contents}
\tableofcontents
\fi

\ifshownavigationpage
\newpage
\normalsize

\section*{Navigation Links}
\ifshowtoc
    \noindent
    \hyperlink{location of table of contents}{Table of Contents}
    \bigskip
\fi

\ifshowequationlist
    \noindent
    \hyperlink{location of equation number list}{Numbered Equations}
    \bigskip
\fi

\ifshowtheoremlist
    \noindent 
    \hyperlink{location of theorem list}{Theorem List}
    \bigskip
\fi

\ifshownotationindex
    \noindent
    \hyperlink{location of notation index}{Notation Index}
    
\fi

\fi

\end{document}